%% file: scene.arxiv.tex
\newif\ifdraft
\draftfalse

\documentclass[reqno,a4paper]{amsart}

\usepackage{mathrsfs,amsmath,amssymb,graphicx}
\usepackage{stmaryrd}
\usepackage[unicode]{hyperref}
\usepackage{latexsym}
\usepackage{layout}
\usepackage[english]{babel}
\usepackage{color}
\usepackage{fancyvrb}
\usepackage{color}
\usepackage{amsmath,amscd}
\usepackage{subfig} 

\usepackage{mathtools}
\usepackage{tikz} 

\usetikzlibrary{matrix,arrows,decorations,shapes,calc}
\tikzset{every picture/.append style={remember picture},
na/.style={baseline=-.5ex}}

\theoremstyle{plain}
\newtheorem{theorem}{Theorem}[section]
\newtheorem{lemma}[theorem]{Lemma}
\newtheorem{proposition}[theorem]{Proposition}
\newtheorem{corollary}[theorem]{Corollary}

\newtheorem{construction}[theorem]{Construction}
\newtheorem{definition}{Definition}[section]
\theoremstyle{remark}
\newtheorem{remark}{Remark}[section]
\newtheorem{example}{Example}[section]

\newcommand{\inside}{E}
\newcommand{\outside}{F}
\newcommand{\nada}[1]   {}

\newcommand{\closedintervalzeroone}{I}

\newcommand{\piout}{\pi_1(\outside,\ast)}
\newcommand{\knottable} {transverse~}
\newcommand{\fund} {\mathcal F}

\newcommand{\R} {\mathbb R}
\newcommand{\Sbb} {\mathbb S}
\newcommand{\SSS} {\mathcal S}
\newcommand{\sphere} {\Sbb^3}
\newcommand{\tempo} {\tau}

\definecolor{mygray}{rgb}{0.92,0.92,0.92}
\newcommand{\graybox}[1]{\setlength{\fboxsep}{1pt}\colorbox{mygray}{#1}}
\newcommand{\triplet}{{(\inside,\Sigma,\outside)}}
\newcommand{\tripletDK}{{(\inside^{D,K},\Sigma^{D,K},\outside^{D,K})}}
\newcommand{\half}{{\mathsf h}}
\newcommand{\id}{{\operatorname{id}}}

\newcommand{\HKfiveone}{{\operatorname{HK}5_1}}
\newcommand{\HKsixone}{{\operatorname{HK}6_1}}
\newcommand{\HKsixtwo}{{\operatorname{HK}6_2}}
\newcommand{\HKsixfour}{{\operatorname{HK}6_4}}
\newcommand{\HKsixeleven}{{\operatorname{HK}6_{11}}}
\newcommand{\Kninefortitwo}{{\operatorname{K}9_{42}}}
\newcommand{\Ktenseventyone}{{\operatorname{K}10_{71}}}

\newcommand{\draftMMM}[1]{\ifdraft{\color{blue}#1}\fi}

\newcommand{\draftYYY}[1]{\ifdraft{\color{orange}#1}\fi}

\numberwithin{equation}{section}
\numberwithin{figure}{section}

\title{On closed
oriented surfaces in the 3-sphere 
}

\author{Giovanni Bellettini}
\address{Dipartimento di Ingegneria dell'Informazione e Scienze Matematiche, Universit\`a di Siena, 53100 Siena, Italy,
and International Centre for Theoretical Physics ICTP,
Mathematics Section, 34151 Trieste, Italy
}
\email{bellettini@diism.unisi.it}
\author{Maurizio Paolini}
\address{Dipartimento di Matematica e Fisica, Universit\`a Cattolica del Sacro Cuore, 25121 Brescia, Italy}
\email{maurizio.paolini@unicatt.it}
\author{Yi-Sheng Wang}
\address{National Center for Theoretical Sciences, Mathematics Division, Taipei 106, Taiwan}
\email{yisheng@ncts.ntu.edu.tw}

\date{\today}

\begin{document}

%{\tiny \tableofcontents}
\thanks{}

%$3D$ objects or more precisely $3$-manifolds in the $3$-sphere $\sphere$ are geometric objects
%studied by both computer scientists and low-dimensional topologists; it is a subject closely 
%related to knotted surfaces in $\sphere$ and handlebody knot theory.
\begin{abstract}
\draftYYY{For the published version, check the folder ``shortened\_version''.}
In this paper we study 
%$3$-manifolds with connected boundary in an oriented $3$-sphere $\mathbb S^3$,
%which is equivalent to 
embeddings of oriented connected closed surfaces in $\mathbb S^3$.
We define a complete invariant, the fundamental span, for such embeddings, generalizing the notion of the peripheral system of a knot group.
From the fundamental span, several computable invariants are derived 
and employed 
to study handlebody knots, bi-knotted surfaces, and chirality of knots.  
These invariants are capable to distinguish 
inequivalent handlebody knots and bi-knotted surfaces 
with homeomorphic complements. Particularly, we obtain an alternative proof of 
the inequivalence of Ishii et al.'s handlebody knots $5_{1}$ and $6_{4}$, and also construct 
an infinite family of pairs of inequivalent bi-knotted surfaces 
with homeomorphic complements. 
An interpretation of Fox's invariant in terms of the fundamental span 
is discussed and used to show    
$9_{42}$ and $10_{71}$ in the Rolfsen knot table are chiral; their chirality is known to be undetectable 
by the Jones and HOMFLY-PT polynomials. 
\end{abstract}

\maketitle

%%%%%%%%%%%%%%%%%%%%%%%%%%%%%%%%%%%%%%%%%%%%%%%%%%%%%%%%%%%%%%%%%%%%%%%%%%%%%

%%%%%%%%%%%%%%%%%%%%%%%%%%%%%%%%%%%%%%%%%%%%%%%%%%%%%%%%%%%%%%%%%%%%%%%%%%%%%
\section{Introduction}\label{sec:intro}
By the Gordon-Luecke theorem \cite{GorLue:89} and 
Waldhausen's theorem \cite{Wal:68},
%\cite{Whi:87}, 
the knot type of a knot 
$K\subset\sphere$ is determined, up to mirror image, 
by its knot group $G_K$ and peripheral system. 
More precisely, the peripheral system is 
the subgroup of $G_K$ generated by elements 
represented by a meridian $m$ 
and a preferred longitude $l$ through a base point.
%, the longitude with vanishing linking number with $K$, 
If furthermore, we require $m$ and $l$ are positively oriented 
with respect to the orientation of 
$\sphere$,
then the knot group together with the conjugacy classes of 
the elements $[m]$, $[l]$
completely determines the knot type of $K\subset \sphere$.

Taking a tubular neighborhood of a knot, one can view a knot as 
an embedded solid torus in $\sphere$.
The solid torus inherits a natural orientation from $\sphere$, 
and induces an orientation on its boundary.   
In this way, a knot can be thought of as 
an embedded oriented surface in $\sphere$ \cite{Suz:75}. 
The assignment from the category of knots 
to the category of oriented connected closed surfaces in $\sphere$ 
is one-to-one, and therefore to
distinguish the knot types of two knots amounts to 
determining 
whether the associated embedded oriented surfaces are
ambient isotopic. The aim of this paper is to 
construct a complete invariant 
for oriented connected closed surfaces of arbitrary genus smoothly
embedded in $\sphere$, generalizing the knot group with
its peripheral system, and examine the topology of 
connected closed surfaces in $\sphere$
via computable invairants derived therefrom.

Any oriented connected closed surface $\Sigma$ in $\sphere$ 
gives rise to an oriented $3$-dimensional submanifold $\inside$ in $\sphere$ which is the closure of the connected component in $\sphere\setminus \inside$ satisfying $\partial \inside=\Sigma$ (i.e.\ with a compatible 
orientation). 
%of $\inside$ 
%with $\Sigma$. 
Conversely, given any connected $3$-dimensional submanifold with connected boundary
$\inside\subset\sphere$, 
$\partial \inside$ is an oriented connected closed surface in $\sphere$. 
In particular, 
the notion of oriented connected closed surfaces in $\sphere$ 
generalizes embeddings
of handlebodies in $\sphere$, namely handlebody knots.
%
%The above remark suggests another point of view emphasizing on 
%the thickened solid set (the tubular neighborhood of $K$), 
%which is the ``inside'', denoted by $\inside$ throughout the paper.  
%The closure of the complement of $\inside$ in $\sphere$ is the ``outside'' and denoted by %$\outside$, 
%which is of paramount importance in knot theory.    
%Another point of view without referring to the embedding 
%map $\Sigma\hookrightarrow \sphere$ or $\inside\hookrightarrow \sphere$ 
%is to consider it 
%
An oriented connected closed surfaces $\Sigma$ in $\sphere$
can be viewed as a partition of $\sphere$, in which we
let $\inside$ be the ``inside'' and the ``outside'' is the closure $\outside$ 
of the complement of $\inside$. 
We denote by the triplet $(\inside,\Sigma,\outside)$  
such a partition.   
%
%$$
%\sphere = \inside \cup \outside, \qquad \inside\cap \outside = \Sigma, 
%\qquad \Sigma = \partial \inside = -\partial \outside.
%$$
%\footnote{We have $\Sigma = \partial \inside = 
%-\partial \outside$ if we consider orientation.}
%
For instance, given a knot $K$, $\inside$ is a tubular neighborhood of $K$.

On the other hand, if no prescribed orientation of $\Sigma$ is given, 
there is no way to distinguish between
$\inside$ and $\outside$.
Embeddings of unoriented surfaces in $\sphere$ 
are studied, for example, in
\cite{Fox:48}, \cite{Hom:54}. 
In the genus-one case, 
it is equivalent to knots in $\sphere$.
%, due to the fact
%that a torus in $\sphere$ bounds a solid torus at least on one side \cite[p. %107]{Rol:03}\footnote{
%The surface bounds a solid torus on \emph{both} sides only when it corresponds
%to the unknot.}.
More generally, an unoriented surface $\Sigma$ 
of genus $g$ in $\sphere$ with the closure of
one component of $\sphere\setminus\Sigma$ a handlebody is equivalent to a handlebody knot. 
There is an obvious forgetful functor from the category of 
oriented connected surfaces in $\sphere$ 
to the category of connected surfaces in $\sphere$ by ignoring
the distinction between the inside and outside (see Diagram \ref{diag:different_connected_embeddings}).

The denomination ``inside'' and ``outside'' is borrowed from the setting of
\cite{BeBePaPa:15}, where
the ambient space is $\R^3$, topologically equivalent to $\sphere$ 
with ``the point at infinity" $\infty$ removed, 
and $\inside$, a $3$-submanifold
with (a not necessarily connected boundary)
of $\R^3$, is called a {\it scene}.
In topology, studies from this point of view can be found in 
\cite{Tsu:70}, \cite{Tsu:75}.
This motivates our choice of the name ``scene'' 
for the triplet $(\inside, \Sigma, \outside)$
(Definition \ref{def:scene}).

Having an infinite point allows us to orient $\Sigma$ naturally 
and thus distinguish
between $\inside$ and $\outside$. 
Discussion from this point of view can be found in \cite{BeFrGh:12}.
Apart from this, $\inside$ and $\outside$ are treated equally
in the sense that they play an equally 
important role in the triplet $(\inside,\Sigma,\outside)$. 
%The notion of connected scenes   
%includes ``theory of knots'' and ``theory of handlebody knots'' 
%by viewing the handlebody part as the inside $\inside$.
%; it provides a 
%perhaps more natural perspective on such embeddings.

In this paper we confine ourselves 
to the case of {\it connected scenes} $(\inside,\Sigma,\outside)$ 
({\it i.e.} $\Sigma$ is connected).
For each component in $(\inside,\Sigma,\outside)$,  
we consider its fundamental group, and the
three fundamental groups are interrelated 
via the homomorphisms induced by the inclusions $i_\inside: \Sigma \to \inside$, 
$i_\outside: \Sigma\to \outside$. 
The orientation information can be captured by 
taken into account
the intersection form
on the first homology group $H_1(\Sigma)$ of $\Sigma$.
In the case of knots, the fundamental group of $F$ 
corresponds to the knot group, whereas the kernel of
$\pi_1(\Sigma,\ast) \to \pi_1(\inside,\ast)$ and the image of
$\pi_1(\Sigma,\ast) \to \pi_1(\outside,\ast)$, $\ast\in\Sigma$,
along with the intersection form on $H_1(\Sigma)$ give
the peripheral system. 

%\footnote{This comes from a choice of an orientation for $\sphere$ 
%and deciding that, {\it e.g.}, $\inside$ is the ``inside''
%.}
%\draftYYY{Do we still need this footnote; it seems we have explained this in previous paragraphs!?}
This leads to our definition 
of a {\it group span with pairing} (Definition \ref{def:group_span_with_pairing}), which is 
an ordered triplet of groups $(G,\Upsilon,H)$ along with two connecting homomorphisms
$i_G:\Upsilon\rightarrow G$, $i_H:\Upsilon\rightarrow H$ and a pairing on 
the abelianization of $\Upsilon$. A connected scene $\SSS=(\inside,\Sigma,\outside)$ 
induces a natural group span with pairing, called the {\it fundamental span of} $\SSS$,
$(\pi_1(\inside,\ast),\pi_1(\Sigma,\ast),\pi_1(\outside,\ast),{i_\inside}_\ast,{i_\outside}_{\ast},d)$, where
$\ast\in\Sigma$ is a base point, $\pi_1(\cdot)$ is the fundamental group,
and $d$ is the intersection form on the first integral homology group $H_1(\Sigma,\mathbb{Z})$ of $\Sigma$   
(Definition \ref{def:fundamental_span}). 
One of the main results
in the paper is 
Theorem \ref{teo:complete_invariant}, where we show that 
the fundamental span is a complete invariant for connected
scenes, up to ambient isotopy. 
\nada{
In the special case of knots, the peripheral subgroup 
is the image of $i_{\outside_*}$, and one can distinguish the meridian
from the longitude by considering the kernel of $i_{\inside_*}$.
Devising a peripheral system from these pieces of information
is crucial in situations where the diagram of the knot is not directly
available: note carefully that this also happens
when $g=1$ and what is known is e.g. the apparent contour of a highly deformed
knotted solid torus.
\draftMMM{Changed the previos paragraph, the change can be shown with:
svn diff intro.tex -r 17081:17082}
}

We remark that Theorem \ref{teo:complete_invariant} does not imply 
the Gordon-Luecke theorem \cite{GorLue:89}.
Indeed the Gordon-Luecke theorem along with
Waldhausen's theorem imply a stronger version of Theorem \ref{teo:complete_invariant} in the case of knots:
given two embedded tori $\inside\subset\sphere$ and $\inside^\prime\subset\sphere$, 
suppose there are isomorphisms connecting $\phi_\outside:\pi_1(\outside,\ast)\rightarrow \pi_1(\outside^\prime,\ast^\prime)$
and $\phi_\Sigma:\pi_1(\Sigma,\ast)\rightarrow \pi_1(\Sigma^\prime,\ast^\prime)$ such that
${i_\outside^\prime}_\ast\circ\phi_\Sigma=\phi_\Sigma\circ {i_\outside}_\ast$, with $\phi_\Sigma$ preserving the intersection forms on $H_1(\Sigma,\mathbb{Z})$ and $H_1(\Sigma^\prime,\mathbb{Z})$.
Then the two solid tori in $\sphere$ are equivalent. 
%In other words, we do not need the information in homomorphisms
%$i_E$ and $i_E^\prime$. 
%
On the other hand, 
in order to obtain a complete invariant for oriented surfaces of
genus larger than one,
it is necessary to take into account the information hidden 
in the induced homomorphism
${i_\inside}_\ast:\pi_1(\Sigma,\ast)\rightarrow \pi_1(\inside,\ast)$,
given there are infinite many handlebody knots with
homeomorphic complements.

Having a complete invariant does not close the classification problem
of knots. In fact, the problem of distinguishing two finite presentations 
of groups, up to isomorphism, is in general unsolvable 
in the sense that there is no
algorithm that always give an answer in a finite time \cite{Dehn:11}.
Proving isomorphism is generally done by finding a sequence of Tietze moves
that connects the two presentations, whereas proving that two
presentations describe non-isomorphic groups often requires 
computable invariants, such as the Alexander invariant.
Another aim of the paper is to construct 
strong, computable invariants out of the fundamental span, and
particularly those able to differentiate connected scenes with 
homeomorphic components.

The first invariant
applies primarily for handlebody knots,
and it inspired by Fox's proof \cite{Fox:52} 
of inequivalence of the
square knot and granny knot, and Riely's work \cite{Ril:71}
and Kitano and Suzuki's work \cite{KitSuz:12} on
homomorphisms from a knot group to a finite group,
 
To construct the invariant, we consider the subgroup 
${i_\outside}_\ast(\operatorname{Ker}({i_\inside}_\ast))$,
which is the image of 
the kernel of 
\begin{equation}\label{normal_closure_of_M}
{i_\inside}_\ast:\pi_1(\Sigma,\ast)\rightarrow \pi_1(\inside,\ast)
\end{equation}  
in $\pi_1(\outside,\ast)$ under 
\begin{equation*}
{i_\outside}_\ast:\pi_1(\Sigma,\ast)\rightarrow \pi_1(\outside,\ast).
\end{equation*}  
Consider also the surjective homomorphisms from $\pi_1(\outside,\ast)$ 
to a finite group $G$  
that are not surjective after being precomposed with ${i_F}_*$.
Such homomorphisms are called proper homomorphisms  
(Definition \ref{def:proper_homomorphism}).
Then the set of the images of the subgroup ${i_\outside}_\ast(\operatorname{Ker}({i_\inside}_\ast))$ in $G$
under proper homomorphisms up to automorphisms of $G$ is an invariant
of the connected scene $\SSS=(\inside,\Sigma,\outside)$.
The invariant takes value in the set of finite sets of subgroups of $G$,
up to automorphisms of $G$, and is called 
the $G$-image of meridians of $\SSS$ (Definition \ref{def:G_image}) 
since the kernel of \eqref{normal_closure_of_M} 
can be identified with the normal closure of meridians of $\inside$. 
It turns out that the $G$-image of meridians 
(Definition \ref{def:G_image}) can see
subtle difference between handlebody knots. 
To investigate this invariant, we generalize
Motto's and Lee-Lee's constructions 
\cite{Mot:90}, 
\cite{LeeLee:12} to generate 
a wide array of inequivalent handlebody knots with homeomorphic complements; 
computing the $G$-image of meridians for these examples shows 
that it is capable to distinguish many of such handlebody knots.

Inequivalent handlebody knots with homeomorphic complements are 
first discovered by Motto \cite{Mot:90} with a geometric argument; 
in the end of his paper he asks for a computable way to detect
such handlebody knots. The computable invariant devised here partially 
answers his challenge. On the other hand, due to the
finiteness of the group $G$, 
our invariant cannot distinguish an infinite family of such handlebody knots 
as Motto's approach did.

Contrary to knots, inequivalent handlebody knots with homeomorphic complements abound. $(5_1,6_4)$ and $(5_2,6_{13})$
in Ishii et al.'s handlebody knot table \cite{IshKisMorSuz:12}, 
are two such pairs.
The inequivalence of $5_1$ and $6_4$
(resp. $5_2$ and $6_{13}$) is proved 
by a geometric means in \cite{LeeLee:12}. Our computational method 
shows that the $A_5$-image of meridians can also see the difference 
between $5_1$ and $6_2$, However, 
it fails to differentiate $5_2$ from $6_{13}$. 
Via the $A_5$-image of meridians, we further identify
two $7$ crossings handlebody knots whose complements are homeomorphic to 
the complements of some handlebody knots in Ishii et al.'s handlebody table.
\begin{center}
\begin{figure}[b]
\includegraphics[scale=.1]{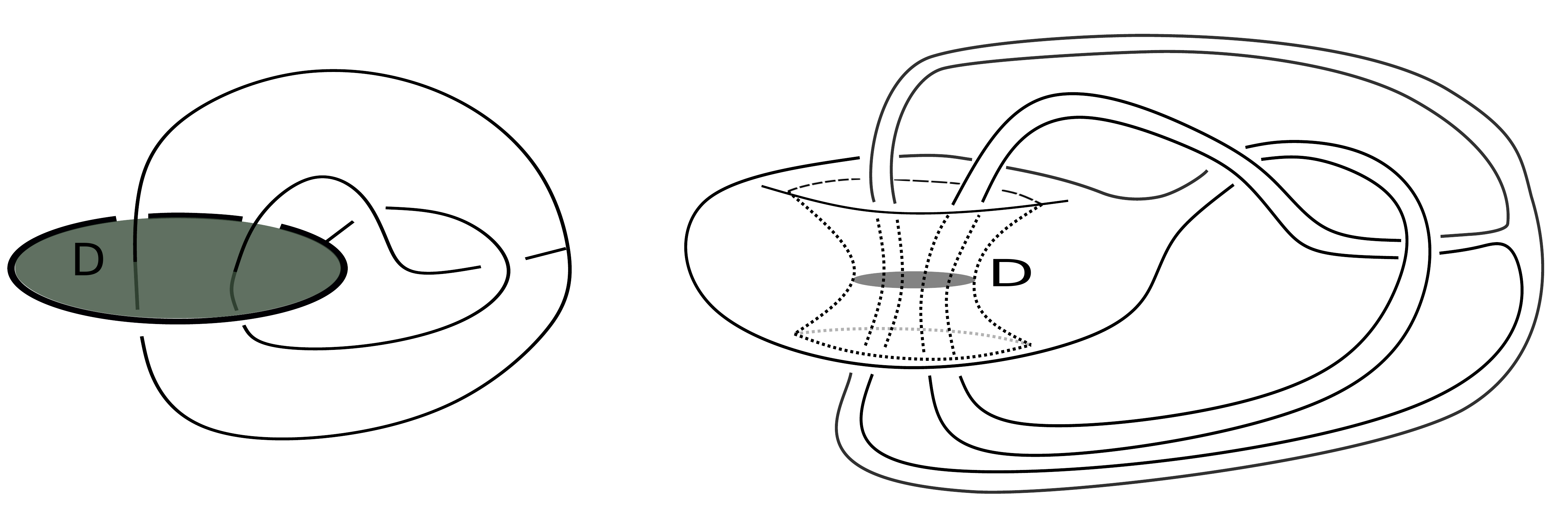}
\caption{Handcuff graph with a $2$-cell $D$ and the associated handlebody knot with a transverse disk.}
\label{Intro:knottable_disk_1}
\end{figure}
\end{center}
\begin{center} 
\begin{figure}[t] 
\includegraphics[scale=.12]{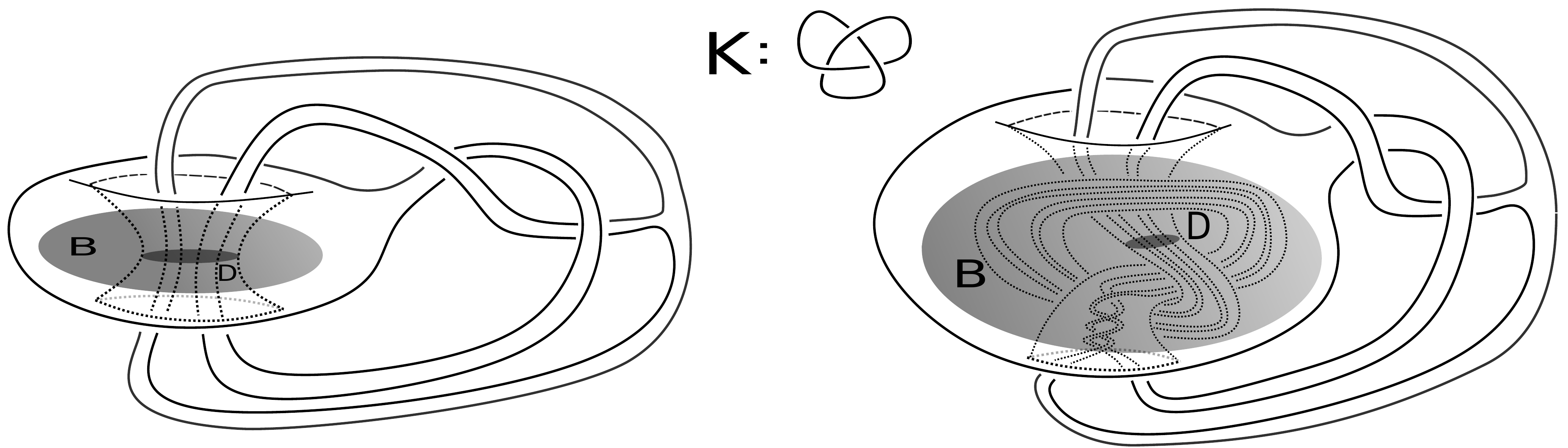}
\caption{Performing the satellite construction 
along a knot $K$ w.r.t.\ a transverse disk $D$.}
\label{Intro:knottable_disk_2} 
\end{figure}
\end{center}
  
To investigate the usefulness of the fundamental span 
in the case of bi-knotted scenes, that is
neither $\inside$ nor $\outside$ a $3$-handlebody, we introduce
the notion of transverse disks of a handlebody knot (Definitions \ref{def:potentially_explosive_disk} and 
\ref{def:truly_explosive_disk})
and the satellite construction along a knot 
with respect to (w.r.t.\ hereafter) 
a transverse disk. For instance, every handcuff graph in $\sphere$ with at
least one of its two circles unknotted in $\sphere$ admits
a natural transverse disk (Fig.\ \ref{Intro:knottable_disk_1}).
Performing the satellite construction 
amounts to replacing a $3$-ball neighborhood of
the transverse disk by a $3$-ball with knotted tubes inside (Fig. \ref{Intro:knottable_disk_2}).
The satellite construction gives 
an ample supply of bi-knotted scenes derived from handlebody knots.
In Ishii et al.'s handlebody knot table, there are
$10$ handcuff graph diagrams, each of which has 
two unknotted circles in $\sphere$, and hence admits
two natural transverse disks.
A question thus arises as to whether performing the satellite construction
w.r.t.\ the two transverse disks results 
in equivalent bi-knotted scenes. 
In some cases, the equivalence is obvious, while in the other cases, 
proving or disproving the equivalence between them could be challenging. 
We derive an invariant of irreducible handlebody knots with a transverse disk
from the fundamental span, and use it to show 
that $5_1$, $6_1$ and $6_{11}$ are the only
three among the ten handcuff diagrams, where performing the satellite construction w.r.t.\ associated disks yields 
inequivalent bi-knotted scenes.

Lastly, we examine the role of the intersection form in the fundamental span, 
a crucial ingredient in discerning the chirality of a connected scene. 
We demonstrate this by translating Fox's argument \cite{Fox:52} into an
invariant in terms of the fundamental span, 
and via the invariant, we determine the chirality of 
$9_{42}$ and $10_{71}$ in the Rolfsen knot table; their
%Over the last decades, many efforts have been made to construct 
%invariants and tools to study chiral knots. In particular, 
%the discovery of knot polynomials, such as 
%the Jones polynomial, HOMFLY-PT polynomial and Kauffman 
%polynomial, brought a breakthrough in this area. Knot polynomials
%are sensitive to chirality of a knot, and most
%of chiral knots up to ten crossings can be detected by 
%them. Nevertheless, none of them is a complete invariant for
%knot chirality.
%---there is no known complete invariant for detecting chiral knots. 
chirality is known to be undetectable by any known knot polynomials.
In \cite{RamGovKau:94}, Chern-Simons invariants are employed to determine their
chirality; another approach via the determinant of a knot is discussed in \cite{FriMilPow:17}.

The paper is organized as follows: In Section \ref{sec:scenes} we introduce
the notion of scene and equivalence of scenes.
In Section \ref{sec:fundamental1}, we construct
the assignment from the category of connected scenes to 
the category of group spans with pairing
and show the assignment is one-to-one. 
Section \ref{sec:examples} discusses constructions
that generate connected scenes with homeomorphic components. We produce several
examples and study their properties. 
Invariants of group spans with pairing defined in terms of homomorphisms of $\pi_1(\outside,\ast)$ into 
a given finite group are introduced in Section \ref{sec:representations}; they are
used in subsection \ref{sec:using} to prove statements given 
in Section \ref{sec:examples}.

%%%%%%%%%%%%%%%%%%%%%%%%%%%%%%%%%%%%%%%%%%%%%%%%%%%%%%%%%%%%%%%%%%%%%%%%%%%%%

\section*{Acknowledgements}
We are grateful to the Mathematisches Forschungsinstitute
Oberwolfach where this research started. We are also grateful
to many people for useful discussions. In particular, we are
indebted to R. Frigerio, M. Mecchia, C. Petronio  and B. Zimmermann.
The present paper benefits from the support of the GNAMPA 
(Gruppo Nazionale per l'Analisi Matematica,
la Probabilit\`a e le loro Applicazioni) of INdAM
(Istituto Nazionale di Alta Matematica), and National Center for Theoretical Sciences.
%%%%%%%%%%%%%%%%%%%%%%%%%%%%%%%%%%%%%%%%%%%%%%%%%%%%%%%%%%%%%%%%%%%%%%%%%%%%%
\section{Scenes and scenes equivalence}\label{sec:scenes}
\begin{definition}[\textbf{Scene}]\label{def:scene}
A \emph{scene} in $\sphere$ is an ordered triplet 
$\SSS = (\inside, \Sigma, \outside)$ 
of oriented 
manifolds in $\sphere$
in the smooth %or $\operatorname{PL}$ 
category\footnote{We work in the smooth category %or $\operatorname{PL}$ category 
to avoid pathological examples, such as the Alexander horned sphere \cite{Ale:24}.  
} such that $\inside$ and $\outside$ are $3$-manifolds with 
$\inside \cup F = \sphere$ and
$\inside \cap F = \Sigma$, where $\Sigma$ is a closed oriented surface 
with $\Sigma = \partial \inside = -\partial F$. 
\end{definition}

\begin{definition}[\textbf{Equivalence of scenes}]\label{def:equiv}
Two scenes $\SSS = (\inside, \Sigma, \outside)$ and $\SSS' = (\inside', \Sigma', \outside')$ (or two embeddings $E\hookrightarrow \sphere$ and $E^\prime
\hookrightarrow \sphere$) are equivalent, 
if there exists an ambient isotopy 
$\Phi_t:\sphere  \rightarrow \sphere$ 
such that $\Phi_0=\textrm{id}$ and $\Phi_1(\inside)=\inside'$

\end{definition}
The definition implies automatically
$\Phi_1(\Sigma) = \Sigma^\prime$ and 
$\Phi_1(\outside) = F^\prime$.

\begin{remark}\label{rem:Cerf}\rm
%$\Phi_t$ respects the orientation of $\sphere$, for each $t
%\in [0,1]$, and in particular, 
$\Phi_1$ is an orientation-preserving 
self-diffeomorphism of $\sphere$ 
sending $\inside$ to $\inside'$, and conversely, any orientation-preserving 
self-diffeomorphism of $\sphere$ sending $E$ to $E'$
is connected to 
the identity via an ambient isotopy 
\cite{Cer:68}.
\end{remark}

\begin{definition}[\textbf{Connected scene, genus}]\label{def:connected_scene}
A scene $\SSS = (\inside, \Sigma, \outside)$ is connected if $\Sigma$ is connected. 
The genus of a connected scene $\SSS=(\inside,\Sigma,\outside)$ is the genus of the surface $\Sigma$.
\end{definition}
\noindent

A connected scene has connected inside $\inside$ and outside $\outside$ 
as $\Sigma = \partial \inside = -\partial F$.

In this paper, we shall restrict ourselves to connected scenes. 
Relations with other embedded objects in $\sphere$, 
as discussed in the introduction, are summarized in the diagram below 
($\rightarrowtail$ stands for an injection of categories and $\twoheadrightarrow$ for a surjection): 

\begin{center}
\begin{equation}\label{diag:different_connected_embeddings} 
\begin{tikzpicture}
\node (CS) at (5,0) [rectangle, draw, align=center] {Connected Scenes\\ $(\inside,\Sigma,\outside)$};  
\node (HK) at (0,1) [rectangle, draw] {Handlebody Knots};
\node (K) at (0,-1) [rectangle, draw] {Knots};
\node (G) at (3,2) [ellipse, draw] {Spatial Graphs};
\node (S)at (9.5 ,0) [ellipse, draw] {Surfaces in $\sphere$};
\draw [>->] (1.7,.8) to (3.4,.3);
\draw [>->] (.9,-.9) to (3.4,-.3);
\draw [>->] (0,-.55) to (0,.55);
\draw [>->] (1.7,1) to [out=0, in=160] (8,.3);
\draw [>->] (.9,-1.1) to [out=0, in=-160](8,-.3);
\draw [->>] (6.7,0) to  node[above]{{\rm p}} (7.6,0);
\draw [->>] (1.3,2.1) to [out=180,in=90](0,1.5);
\end{tikzpicture}
\end{equation}
\end{center}
In view of the diagram, we call a connected scene \emph{a knot} if $\inside$ is a 
solid torus and call it \emph{a handlebody knot} if $\inside$ is a $3$-handlebody of 
genus $g \geq 1$.  
\draftMMM{Changed ``genus larger than $1$'' to ``genus $g \geq 1$''.}

\begin{definition}[\textbf{Trivial scene}]
A connected scene $(\inside, \Sigma, \outside)$
is trivial if both $\inside$ and $\outside$ are $3$-handlebodies. 
\end{definition}

Note that, by \cite[Satz $3.1$]{Wal:68i}, the Heegaard splitting of $\sphere$ 
of genus $g$ is unique, for every $g\geq 0$, namely the standard one. 
Hence, every two trivial scenes of the same genus are ambient isotopic. 
We use the symbol $H_{g}$ to
denote a handlebody of genus $g$ and $\Sigma_{g}$ a surface of genus $g$.
We drop $g$ when there is no risk of confusion.

\begin{definition}[\textbf{Bi-knotted scene}]
A bi-knotted scene is a connected scene $(\inside,\Sigma,\outside)$
with neither $\inside$ nor $\outside$ a $3$-handlebody.
\end{definition}

To understand the relation between connected scenes (or 
equivalently oriented connected closed surfaces in $\sphere$)   
and connected closed surfaces in $\sphere$ (mapping {\rm p} in
Diagram \eqref{diag:different_connected_embeddings} above)
we observe that, given a connected surface $\Sigma$ in $\sphere$, 
if the connected components $V_{1}$ and $V_{2}$ of $\sphere\setminus \Sigma$
are not homeomorphic, then the preimage of the surface in $\sphere$ under {\rm p}
contains precisely two elements---one regards $V_{1}$ as, say, the ``inside''
and $V_{2}$ as the ``outside''. 

If $V_{1}$ and $V_{2}$ are homeomorphic, the situation is subtler.
For the sake of convenience, we give the following definition.

\begin{definition}[\textbf{Symmetric scene}]\label{def:symmetric_scene}
A connected scene $(V, \Sigma, W)$ is symmetric
if $V$ and $W$ are homeomorphic.
\end{definition}

\begin{definition}[\textbf{Swappable/unswappable scene}]
A symmetric scene $(V,\Sigma , W)$ is swappable (resp. unswappable)
if it is equivalent (resp. inequivalent) to $(W,-\Sigma, V)$.  
\end{definition}

In the case of genus one, the only symmetric scene is the trivial scene
and it is swappable. In the case of genus two, by \cite[Theorem $1$]{Tsu:75}, 
every symmetric scene $(V,\Sigma_{2},W)$ must
have $V$ (resp. $W$) homeomorphic to the boundary connected 
sum of a solid torus $H_{1}$ and the complement of an open tubular neighborhood
of a knot $K_V$ (resp. $K_W$). Thus, by the Gordon-Luecke theorem
\cite{GorLue:89} and \cite[Corollary $3.4$]{Suz:75}, the knots $K_V$ and $K_W$ 
are equivalent, up to mirror image. 
Therefore, a symmetric scene $(V,\Sigma_{2},W)$
of genus $2$
is unswappable if and only if $K_V$ and $K_W$ 
are chiral knots and mirror images
to each other. 
Using a corollary of Waldhausen's theorem \cite[Theorem $3$]{Tsu:70},
we obtain the following theorem: 
\begin{theorem}
There exist unswappable symmetric connected scenes of genus $g$, for any $g>1$.
\end{theorem}

\begin{definition}[\textbf{Sum operation}]
Given two connected scenes $\SSS=(\inside,\Sigma,\outside)$ and $\SSS'=(\inside',\Sigma',\outside')$,
their connected sum $\SSS\#\SSS'$ is a connected scene given by
removing a point $p\in \Sigma$ and a point $p'\in\Sigma'$ and glue 
them together via an orientation-reversing diffeomorphism
\begin{align*}
(B_{p}\cap\Sigma, B_{p} ) \simeq (\Sbb^1, \Sbb^2)\times (0,1) 
&\rightarrow (B_{p'}\cap \Sigma', B_{p'})\simeq (\Sbb^1, \Sbb^2)\times (0,1)\\
(x,t)&\rightarrow (x,1-t),
\end{align*}
where $B_{p}$ (resp. $B_{p'}$) is a $3$-ball neighborhood of $p$ (resp. $p'$) 
in $\sphere$ with $p$ (resp. $p'$) removed. The first and last
orientation-preserving diffeomorphism identify $B_p$ and $B_p'$ with a unit 
$3$-ball, respectively.
The components of $\SSS\#\SSS'$ are denoted by $(\inside\# \inside',\Sigma\#\Sigma',\outside\# \outside')$.
\end{definition}
The sum operation is associative and commutative.

Given a connected scene $\SSS$, if it is equivalent to the connected sum of
connected scenes $\SSS_i$, $i=1,...n$, then we say $\SSS_1\#\SSS_2\#...\#\SSS_n$ 
is a decomposition of $\SSS$. 

\begin{definition}[\textbf{Prime scenes}]
A connected scene $\SSS$ is prime 
if its genus is larger than $0$ and 
admits no decomposition $\SSS=\SSS_1\#\SSS_2$ with both $\SSS_1$ and $\SSS_2$
non-trivial.
A decomposition $\SSS \simeq \SSS_1\#\SSS_2\#...\#\SSS_n$ is prime
if $\SSS_i$ is prime, $i=1,...,n$.
\end{definition}   
Note that our prime handlebody knots are called irreducible handlebody 
knots in \cite{IshKisMorSuz:12}. The notation chosen here is consistent 
with that in \cite{Suz:75}, \cite{Tsu:75} and in knot theory. In particular,
a scene $\SSS$ is prime if and only if $p(\SSS)$ regarded as an unoriented surface in
$\sphere$, is prime. On the other hand, the notions of prime $\theta$-curves and handcuff graphs 
in \cite{Mor:07}, \cite{Mor:09}\footnote{The classification of irreducible handlebody knots 
in \cite{IshKisMorSuz:12} is based on the classification of $\theta$-curves and handcuff graphs.} 
have different meanings.

The examples of unswappable scenes given above are non-prime. 
In fact, there is no unswappable prime scene with genus less than $3$.
In Section \ref{subsec:Unswappable_scenes}, we give a construction 
of unswappable prime scenes of genus $3$, as an application of the existence 
of inequivalent handlebody knots with homeomorphic complements,
and prove the following result:

\begin{theorem}\label{Unswappable_scenes_of_genus_3}
There exist infinitely many unswappable prime scenes of genus $3$.
\end{theorem}

%%%%%%%%%%%%%%%%%%%
%A connected scene $\SSS = (\inside, \Sigma, \outside)$ with $\Sigma$ a surface of
%genus $1$ and $\pi_1(\inside)$ free can be uniquely
%associated to a 
%knot in $\mathbb S^3$
%\cite[Theorem $5.2$]{Hem:04}.
%On the contrary, examples given by tubular neighborhoods of links with more than one component are \emph{not} %included.
%Clearly, if $K$ is a (tame) knot in $\mathbb S^3$, it can be associated to a 
%scene $(K_\rho, \partial K_\rho, \mathbb S^3 \setminus K_\rho)$, where $K_\rho:= 
%\{x \in \mathbb S^3: {\rm dist}(x, K) \leq \rho\}$, with $\rho>0$ sufficiently small.
%%%%%%%%%%%%%%%%%%%%%%%%%%

%%%%%%%%%%%%%%%%%%%%%%%%%%%%%%%%%%%%%%%%%%%%%%%%%%%%%%%%%%%%%%%%%%%%%%%%%%%%%

%%%%%%%%%%%%%%%%%%%%%%%%%%%%%%%%%%%%%%%%%%%%%%%%%%%%%%%%%%%%%%%%%%%%%%%%%%%%%
\section{Fundamental structure for connected scenes}\label{sec:fundamental1}
Given a connected scene $\SSS = (\inside, \Sigma, \outside)$
and a base point $\ast\in\Sigma$, the fundamental groups 
$\pi_1(\inside,\ast)$, $\pi_1(\Sigma,\ast)$, and $\pi_1(\outside,\ast)$
are related to each other via the homomorphisms ${i_\inside}_*$ and ${i_\outside}_*$
induced by the inclusions $i_\inside:\Sigma \to \inside$ and $i_\outside:\Sigma \to \outside$,
respectively. 
In general, ${i_\inside}_*$ and ${i_\outside}_*$ 
are neither injective nor surjective.
 
The following unknotting theorem is a corollary of \cite[Theorem $5.2$]{Hem:04}. 
\begin{proposition}
Let $\SSS=(\inside,\Sigma,\outside)$ be a connected scene 
and $\ast\in\Sigma$ a basepoint. Then 
$\SSS$ is trivial if and only if $\pi_{1}(\inside,\ast)$ and $\pi_1(\outside,\ast)$ are free groups.
\end{proposition} 
\begin{proof}
\cite[Theorem $5.2$]{Hem:04} asserts that a prime $3$-manifold with the fundamental group 
a free group is either a $S^{2}$ bundle over $S^{1}$ or a $3$-ball with some $1$-handles 
attached to its boundary. 

Now, since $\Sigma$ is connected, its complements $\inside$ and $\outside$ must be 
irreducible $3$-manifolds, and hence, $\pi_{2}(\inside,\ast)$ and $\pi_{2}(\outside,\ast)$ are trivial 
by the sphere theorem. 
That implies  $\inside$ (resp. $\outside$) cannot be a $S^{2}$-bundle over $S^{1}$. 
On the other hand, any $3$-manifold in $S^{3}$ is orientable, 
so $\inside$ (resp. $\outside$) must be a $3$-handlebody. 
\end{proof}
The topological type of a connected scene $\SSS=(\inside,\Sigma,\outside)$ 
is not determined solely by the fundamental groups of $\inside$ and $\outside$; 
there are inequivalent connected scenes with homeomorphic outsides 
and insides \cite{Mot:90}, \cite{LeeLee:12}. To distinguish such connected scenes,
additional structures need to be taken into account.
 
To this aim, we introduce the notion of a fundamental span,
which is an analog of a knot group with the peripheral system; 
the following definitions describes
the algebraic universe where fundamental spans live.

\begin{definition}[\textbf{Group span with pairing}]\label{def:group_span_with_pairing}
A group span with pairing is an ordered triplet of groups $(G,\Upsilon,H)$ along with two connecting 
homomorphisms
$i_{G}:\Upsilon\rightarrow G$ and $i_{H}:\Upsilon\rightarrow H$ and a non-degenerate pairing
$d:\Upsilon/[\Upsilon,\Upsilon]\times \Upsilon/[\Upsilon,\Upsilon]\rightarrow \mathbb{Z}$, where
$[A,A]$ denotes the commutator subgroup of the group $A$.
\end{definition}

\begin{definition}[\textbf{Equivalence of group spans with pairing}]
\label{def:equivalence_of_group_span_with_pairing}
Two group spans with pairing
\[(G,\Upsilon,H,i_{G},i_{H},d),\qquad
(G^{\prime},\Upsilon^{\prime},H^{\prime},i_{G^{\prime}},i_{H^{\prime}},d^{\prime})\]
are equivalent if there are isomorphisms
$G\rightarrow G^{\prime}$, $\Upsilon\rightarrow \Upsilon^{\prime}$, and $H\rightarrow H^{\prime}$
such that the diagram  
\begin{center} 
\begin{tikzpicture}
\node[black](Lu) at (0,2) {$G$};
\node(Lm) at (0,1) {$\Upsilon$};
\node(Ll) at (0,0) {$H$}; 
\node(Ru) at (2,2) {$G^{\prime}$};
\node(Rm) at (2,1) {$\Upsilon^{\prime}$};
\node(Rl) at (2,0) {$H^{\prime}$};

\path[->, black, font=\scriptsize,>=angle 90] 

(Lu) edge (Ru)  
(Lm) edge (Lu)
(Lm) edge (Ll)
(Lm) edge (Rm)
(Ll) edge (Rl) 
(Rm) edge (Ru)
(Rm) edge (Rl);
\end{tikzpicture}
\end{center}
commutes and the isomorphism $\Upsilon\rightarrow
\Upsilon^{\prime}$ preserves the pairings $d$ and $d^{\prime}$. 
\end{definition}

Given a connected scene $\SSS=(\inside,\Sigma,\outside)$ 
and a base point $\ast\in \Sigma$, the fundamental group functor 
$\pi_{1}(-)$ 
gives a group span with pairing
\begin{equation}\label{pi_1_functor}
(\pi_1(\inside,\ast),\pi_1(\Sigma,\ast),\pi_1(\outside,\ast),{i_\inside}_*,{i_\outside}_*,d),
\end{equation} 
where $d$ is the intersection form
on $H_{1}(\Sigma,\mathbb{Z})$ 
given by the orientation of $\Sigma$.
Notice that using different base points results in 
equivalent group spans with pairing; 
thus the equivalence class of \eqref{pi_1_functor} 
is independent of the choice of a base point. 

\begin{definition}[\textbf{Fundamental span}]\label{def:fundamental_span}
Given a connected scene $\SSS=(\inside,\Sigma,\outside)$, 
the equivalence class of $(\pi_1(\inside,\ast),\pi_1(\Sigma,\ast),\pi_1(\outside,\ast),{i_\inside}_*,{i_\outside}_*,d)$
is called the fundamental span of $\SSS$, and 
is denoted by $\fund(\SSS)$.
\end{definition}

Any (base point preserving) equivalence between (based) connected scenes 
induces equivalent group spans with pairing; thus, $\fund(\cdot)$ induces a mapping 
from the equivalence classes of connected scenes 
to the equivalence classes of group spans with pairing:
\[\fund:\{\text{connected scenes}\}/\simeq\hspace*{.5em}\longmapsto\hspace*{.5em}\{\text{group spans with pairing}\}/\simeq,\]
where $\simeq$ is the equivalence between connected scenes or group spans with pairing.

%When  $\inside$ is a knotted solid torus, the resulting algebraic scene contains the information of the \emph{peripheral system} for the inclusion $\Sigma\hookrightarrow \outside$; thus, by \cite{Wal:68}, the fundamental scene of a connected scene $\SSS$ completely determines the topological type of $\SSS$. Employing Suzuki's prime decomposition theorem \cite{Suz:75}, we extend this result to arbitrary connected scenes. 

\begin{theorem}[\textbf{Complete invariant}]\label{teo:complete_invariant}
The mapping $\fund$ is injective. In other words, the fundamental span
is a complete invariant for connected scenes. 
\end{theorem}

\begin{proof}
Firstly, note that 
connected scenes of different genus cannot have the same fundamental spans.

Secondly, observe that, in the case of genus $0$, $\Sigma$ is a $2$-sphere, and 
the $3$-dimensional Sch\"onflies theorem \cite{Maz:61} implies
all connected scenes are ambient isotopic.  
Thus, the theorem holds trivially in this case.
 
%$G_\inside$, $G_\Sigma$, $G_\outside$ are the trivial group. On the other hand, the $3$-dimensional Sch\"onflies theorem\cite{Moi:52} (in $\operatorname{PL}$ 
%category \draftYYY{Need to be careful. Make sure embedded polyhedrally is the same as PL-embedding}), \cite{Maz:61} (in smooth category) implies that there is only one such embedding. Hence, $\fund(\cdot)$ trivially gives a 
%complete invariant in this case.

Now, suppose there exists an equivalence between the fundamental spans of 
two connected scenes $\SSS=(\inside,\Sigma,\outside)$ and $\SSS^{\prime}=(\inside^{\prime},\Sigma^{\prime},\outside^{\prime})$
of genus $g>0$;
that is there exist isomorphisms $\phi_{\inside}$, $\phi_{\Sigma}$ and $\phi_{\outside}$
such that the diagram
\begin{center}
\begin{equation}\label{EqbetweenCS}
\begin{tikzpicture}[baseline = (current bounding box.center)]
\node(Lu) at (0,3) {$\pi_{1}(\inside,\ast)$};
\node(Lm) at (0,1.5) {$\pi_{1}(\Sigma,\ast)$};
\node(Ll) at (0,0) {$\pi_{1}(\outside,\ast)$};
\node(Ru) at (3,3) {$\pi_{1}(\inside^{\prime},\ast^{\prime})$};
\node(Rm) at (3,1.5) {$\pi_{1}(\Sigma^{\prime},\ast^{\prime})$};
\node(Rl) at (3,0) {$\pi_{1}(\outside^{\prime},\ast^{\prime})$};

\path[->, font=\scriptsize,>=angle 90]

(Lu) edge node [above]{$\phi_{\inside}$}node[below]{$\sim$}(Ru)
(Lm) edge (Lu)
(Lm) edge (Ll)
(Lm) edge node [above]{$\phi_{\Sigma}$}node[below]{$\sim$}(Rm)
(Ll) edge node [above]{$\phi_{\outside}$}node[below]{$\sim$}(Rl)
(Rm) edge (Ru)
(Rm) edge (Rl);
\end{tikzpicture}
\end{equation}
\end{center}
commutes and $\phi_{\Sigma}$ preserves the intersection forms 
on $H_{1}(\Sigma,\mathbb{Z})$ and $H_{1}(\Sigma^{\prime},\mathbb{Z})$.   

If we can show that $\SSS$ and $\SSS'$ are equivalent, 
the injectivity of $\fund$ follows.
In fact, we shall construct an equivalence of connected scenes 
that realizes the above equivalence of fundamental spans 
$\fund(\SSS)$ and $\fund(\SSS')$. We divide the proof into four steps.

\smallskip
{\it Step 1}: Realizing $\phi_\Sigma$ by an orientation preserving homeomorphism
 
The isomorphism $\phi_{\Sigma}: \pi_{1}(\Sigma,\ast)\xrightarrow{\sim}\pi_{1}(\Sigma^{\prime},\ast^{\prime})$ can be realized by an orientation-preserving diffeomorphism $f_{\Sigma}:(\Sigma,\ast)\xrightarrow{\sim} (\Sigma^{\prime},\ast^{\prime})$. To see this, we note first that $\phi_\Sigma$
can be realized by a homotopy equivalence since surfaces are Eilenberg-Maclane spaces of type $K(G,1)$
\cite[Section $8.1$]{FarMar:11}. 
Secondly, we deform the homotopy equivalence
into a homeomorphism; this can be achieved by 
employing the topological proof of the Dehn-Nielsen-Baer
theorem \cite[Section $8.3.1$]{FarMar:11}. 
The homeomorphism can be further deformed into
a diffeomorphism $f_\Sigma$ by \cite[Theorem $3.10.9$]{SteThu:97}.
Now, identify $\Sigma^\prime$
with $\Sigma$ via an orientation preserving diffeomorphism $g$, and 
observe that $(g\circ f_\Sigma)_\ast=g_\ast\circ \phi_\Sigma$ preserves 
the intersection form on $H_1(\Sigma,\mathbb{Z})$.
By the Dehn-Nielsen-Baer theorem, the self-diffeomorphism 
$g\circ f_\Sigma$ is an orientation preserving 
map, and hence, $f_\Sigma$ preserves the orientations of $\Sigma$ and $\Sigma^\prime$. 
%%%%% 

%Alternatively, one can start with a homotopy equivalence that realizes
%$\phi_{\Sigma}$---such a homotopy equivalence exists because closed surfaces 
%are classifying spaces of its fundamental group---and 
%then deform it to a homeomorphism,
%using the trick in \cite[$2.1$]{DavShe:01}. \draftGGG{Can
%this reference be made more precise? I cannot find the statement}

%can be approximated by a $\operatorname{PL}$-homeomorphism \cite[Chapter $6$]{Moi:77} or a diffeomorphism \cite[Theorem $3.10.9$]{SteThu:97}, the isomorphism $\phi_{\Sigma}: \pi_{1}(\Sigma,\ast)\xrightarrow{\sim}\pi_{1}(\Sigma^{\prime},\ast^{\prime})$ can be realized by an orientation-preserving diffeomorphism $f_{\Sigma}:(\Sigma,\ast)\xrightarrow{\sim} (\Sigma^{\prime},\ast^{\prime})$. 

\smallskip

The assertion of the theorem follows immediately if there exist diffeomorphisms 
$f_{\inside}:\inside\xrightarrow{\sim} \inside^{\prime}$ 
and $f_{\outside}:\outside\xrightarrow{\sim} \outside^{\prime}$ extending $f_{\Sigma}$. 

\smallskip
{\it Step 2}: Free product decomposition

Recall that Suzuki's $\partial$-prime decomposition theorem
\cite[Theorem $3.4$]{Suz:75} states that
every $3$-manifold that can be embedded in $\sphere$ 
has a $\partial$-prime decomposition. In particular, the pair 
$(\inside,\Sigma)$ admits a $\partial$-prime decomposition 
\begin{equation}\label{BP-decomposition of inside}
(\inside,\Sigma)=(\inside_{1},\Sigma_{g_1})\#_{b}...\#_{b}(\inside_{n},\Sigma_{g_n}),
\end{equation}  
where $\inside_{i}$ is $\partial$-prime,
$\Sigma_{g_i}=\partial \inside_i$ is a surface of genus $g_i$, 
for every $i=1,\dots,n$, and $\sharp_b$ is boundary connected sum.
We can further assume that the separating disks intersect at the base point $\ast$. 
Since a $3$-manifold that can be embedded in $\sphere$ 
is $\partial$-prime if and only if its fundamental group is indecomposable 
\cite[Proposition $2.15(5)$]{Suz:75}, this $\partial$-prime decomposition of $\inside$
induces the free product decomposition of $\pi_1(\inside,\ast)$ with indecomposable factors.

Now, we want to use $\phi_\inside$ to show that the $\partial$-prime decomposition of $\inside$ induces a $\partial$-prime decomposition of $\inside'$.

To see this, we first recall the free product decomposition theorem 
\cite[p. $27$, Sec. $35$, Vol. $2$]{Kur:60}  
which states that two free product decompositions of 
a group with indecomposable factors are isomorphic. This implies  
that the isomorphism $\phi_\inside$ induces the free product decomposition of $\pi_1(\inside',\ast)$ with indecomposable factors.

\smallskip
{\it Step 3}: $\partial$-prime decomposition

On the other hand, by Dehn's lemma, there exists a decomposition of $(\inside',\Sigma')$:
\[
(\inside',\Sigma')=(\inside_{1}',\Sigma_{g_1}')\#_{b}...\#_{b}(\inside_{n}',\Sigma_{g_n}') 
\]
induced by the disks in $\inside'$ that are bounded 
by the loops $f_\Sigma(\partial D_{i})$, $i=1,...,n$,
where $D_{i}$, $i=1,..,n$, are the separating disks 
in the $\partial$-prime decomposition of $\inside$ \cite[Condition ($\ast$), p.$186$]{Suz:75}. 
At this stage, we can extend $f_{\Sigma}$ over $\bigcup_{i=1}^{n}D_{i}$.

We want to show that this decomposition is $\partial$-prime
and induces the free product decomposition of $\pi_1(\inside',\ast)$
in Step $2$. To see this, it suffices to prove that $\pi_1(\inside_i',\ast')$
is indecomposable, which follows, provided $\phi_{\inside}$ sends $\pi_1(\inside_i,\ast)$ into $\pi_1(\inside_i',\ast')$, for every $i$.

Recall that the Kurosh subgroup theorem \cite[Section $34$]{Kur:60}
asserts that any indecomposable subgroup $H\neq \mathbb{Z}$ 
in a free product $G_{1}\ast G_{2}$ with $H\cap G_{i}$ non-empty, where $i$ is 
either $1$ or $2$, is a subgroup of $G_{i}$.
(see \cite[$4.1$]{Sta:65}.
% \cite[Corollary $4.3$, Lemma $4.4$]{Nel:08}).

Now, if $\inside_i$ is $\partial$-irreducible, then $\pi_1(\inside_i,\ast)\neq \mathbb{Z}$ \cite[Prop. $2.15$]{Suz:75}.  
If $\inside_i$ is not $\partial$-irreducible, then $\inside_i$ is a solid torus.
For the former, because $\phi_\Sigma(\pi_1(\Sigma_i,\ast))$ is in $\pi_1(\Sigma_i',\ast')$,
$\phi_{\inside}(\pi_1(\inside_i,\ast))\cap\pi_1(\inside_i',\ast')$ is nonempty. So, 
the Kurosh subgroup theorem implies that 
$\phi_{\inside}$ also sends $\pi_1(\inside_i,\ast)$ into $\pi_1(\inside_i',\ast')$.
For the latter, the induced homomorphism
from $\pi_1(\Sigma_i,\ast)$ to $\pi_1(\inside_i,\ast)$ is surjective, and hence 
$\phi_{\inside}$ also sends $\pi_1(\inside_i,\ast)$ into $\pi_1(\inside_i',\ast')$.

So far, we have established that  
$\phi_{\inside}$ (resp. $\phi_{\Sigma}=f_{\Sigma,\ast}$) preserves the free product decompositions 
(resp. with amalgamation) given by 
the $\partial$-prime decompositions of $(\inside,\Sigma)$ and $(\inside^{\prime},\Sigma^{\prime})$. 
 
In particular, $\phi_\inside$ and $\phi_\Sigma$ induce isomorphisms between the fundamental groups of 
corresponding $\partial$-prime factors. That is, they induce the following 
commutative diagram
\begin{center}
\begin{equation}\label{primepartfundamentalspan}
\begin{tikzpicture}[baseline = (current bounding box.center)] 
\node(Lu) at (0,2) {$\pi_{1}(\inside_{i},\ast)$};
\node(Ll) at (0,0) {$\pi_{1}(\Sigma_{g_{i}},\ast))$}; 
\node(Ru) at (4,2) {$\pi_{1}(\inside^{\prime}_{i},\ast^\prime)$};
\node(Rl) at (4,0) {$\pi_{1}(\Sigma_{g_{i}}^{\prime},\ast^\prime)$};

\path[->, font=\scriptsize,>=angle 90]

(Lu) edge node [above]{$\sim$}(Ru)  
(Ll) edge (Lu)
(Ll) edge node [above]{$\sim$}(Rl) 
(Rl) edge (Ru);

\end{tikzpicture}
\end{equation}
\end{center}
for every $i=1,\dots,n$. 
Note that the lower isomorphism can be realized by the restriction of $f_{\Sigma}$ on $\Sigma_{g_{i}}\cup D_i$.

\smallskip
{\it Step 4}: Applying Waldhausen's theorem to \eqref{primepartfundamentalspan}.
 
If $\inside_{i}$ is $\partial$-irreducible, one can construct a diffeomorphism 
%(or $\operatorname{PL}$-homeomorphism) 
that realizes the upper isomorphism by Waldhausen's theorem \cite[Theorem $6.1$]{Wal:68}   
and \cite[Corollary, p.$333$]{Mun:59}

If $\inside_i$ is not $\partial$-irreducible, 
then $\inside_{i}$ is a solid torus, and there is an obvious 
diffeomorphism realizing the upper isomorphism. 

Taking boundary connected sum, we obtain a diffeomorphism $f_{\inside}$ that realizes 
the upper part of diagram \eqref{EqbetweenCS} and extends $f_{\Sigma}$. 

In the same way, one can construct $f_{\outside}:\outside\rightarrow \outside^{\prime}$ 
that extends $f_{\Sigma}$ and realizes the lower part of \eqref{EqbetweenCS}. 
Thus, the connected scenes $\SSS$ and $\SSS^{\prime}$ are equivalent.
\end{proof}
   
\begin{remark}
The theorem is true even when the diagram \eqref{EqbetweenCS}
commutes only up to conjugacy or when the base point of $\inside$ or $\outside$ 
(resp. $\inside'$ or $\outside'$)
is not on $\Sigma$ (resp. $\Sigma'$).
For the latter, the homomorphisms ${i_\inside}_*$ and ${i_\outside}_*$ 
(resp. ${i_\inside'}_*$ and ${i_\outside'}_*$) depend on a choice of arcs
connecting the base points in $\inside$ and $\outside$ to $\ast\in\Sigma$.  
To see the theorem still holds true, 
we observe that, firstly by modifying $\phi_\inside$
or $\phi_\outside$, one can make the diagram commute strictly, and secondly,
one can use the same arcs that connect the base points in $\inside$ and $\outside$
to $\ast\in\Sigma$ to move the base point back to 
the common base point $\ast$ on $\Sigma$. The proof then reduces to the case
of the theorem.  
\end{remark}

%If $\Sigma$ is of genus $1$, it is equivalent to knot theory, and the 
%Gordon-Luecke theorem \cite{GorLue:89} tells
%us that the homeomorphism type of a knot complement determines the knot type, up to mirror image. If furthermore the knot is prime, then its knot type, up to mirror image, is completely determined by the fundamental group of the knot complement. On the other hand, it could be quite difficult to detect the chirality of a knot and often requires the use of knot polynomials not determined by the knot group, such as the Jones polynomial and HOMFLY polynomial, to find it.  

%In the next section, we present two examples of knots whose chirality cannot be detected by the Jones polynomial and HOMFLY polynomial but can be seen by the invariant derived from algebraic scenes.

%The situation becomes more complicated when the genus of $\Sigma$ is greater than $1$. For instance, there are inequivalent handlebody knots having the same fundamental groups \cite{LeeLee:12}, \cite{Mot:90}. Finer invariants than the fundamental group are needed to distinguish such examples.

%Since the genus of $\Sigma$ can be computed from its fundamental group (rank of $\pi_1(\Sigma)$),
%then two scenes with $\Sigma$ of different genus cannot lead to equivalent fundamental scenes.

%%%%%%%%%%%%%%%%%%%%%%%%%%%%%%%%%%%%%%%%%%%%%%%%%%%%%%%%%%%%%%%%%%%%%%%%%%%%%

%%%%%%%%%%%%%%%%%%%%%%%%%%%%%%%%%%%%%%%%%%%%%%%%%%%%%%%%%%%%%%%%%%%%%%%%%%%%%
\section{Examples}\label{sec:examples}

In this section we present methods to produce connected scenes 
with homeomorphic complements and discuss some explicit examples constructed using 
these methods.
The properties of these connected scenes are stated here; their proofs
employ invariants derived from the group span with pairing and
are given in Section \ref{sec:representations}.   

\subsection{Handlebody Knots}\label{subsec:handlebody_knots}
Our first construction concerns handlebody knots, and 
is used to produce inequivalent handlebody knots 
with homeomorphic complements and is a generalization of 
Motto's and Lee-Lee's constructions \cite{Mot:90}, \cite{LeeLee:12}. 

We begin by recalling that a Dehn twist 
of a standard cylinder $\Sbb^{1}\times 
\closedintervalzeroone$ in $\mathbb{R}^{3}$, $\closedintervalzeroone = [0,1]$,  
is a boundary-fixing self-homeomorphism given by
\begin{align*}
\Sbb^{1}\times \closedintervalzeroone& \rightarrow \Sbb^{1}\times \closedintervalzeroone\\
(p,\tempo)&\mapsto (e^{2\pi i \tempo}p,\tempo).
\end{align*}  

This homeomorphism can be 
extended to a self-homeomorphism of a standard cylindrical shell in $\mathbb{R}^{3}$, 
\begin{equation}
\begin{aligned}\label{eq:Dehn_twist}
t&:A\times \closedintervalzeroone\rightarrow A\times \closedintervalzeroone\\
&(p,\tempo)\mapsto (e^{2\pi i\tempo}p, \tempo),
\end{aligned}
\end{equation}
where $A:=\{(x,y)\in\mathbb{R}^{2}\mid \frac{1}{2}\leq x^{2}+y^{2}\leq 1 \}$ is an annulus.%%
\footnote{Using an annulus instead of a disk allows us to cover all of our examples.
We could use a disk instead of an annulus as in \cite{LeeLee:12}, 
with the central disk treated as, say, $D_{n+1}$,
for some of our examples,
{\it e.g.} annulus $A_1$ of Example \ref{exa:HK5_1} and annulus $A_1$
of Example \ref{exa:HK6_2} could be replaced by a disk.
However, this is not possible for the annuli $A_2$ in Examples 
\ref{exa:HK5_1} and \ref{exa:HK6_2}.
Also, twisting in an annulus is required for 
the construction in \cite{Mot:90} but still covers the cited examples.}
Now, suppose $\bigcup_{i=1}^{n}D_{i}$ is a union of pairwise
disjoint disks contained
in $\mathring{A}$, the interior 
of $A$. Then $t$ restricts to an embedding of 
$(A\setminus \bigcup_{i=1}^{n}\mathring{D}_{i})\times  
\closedintervalzeroone$ in $A\times \closedintervalzeroone$,
which twists the void cylinders $D_i\times \closedintervalzeroone$ 
in $A\times \closedintervalzeroone$, $i=1,...,n$.
The sign of such a twisting can be defined 
by the sign of the crossings of the void cylindrical parts 
with the inner cylinder. 
Fig. \ref{fig:Signs_of_twistings} illustrates the embeddings 
induced by $t$ and $t^{-1}$.
The embedding induced by $t^{\pm j}$ 
gives $j$ full $\pm$-twists.

\begin{center}
\begin{figure}[ht]
\def\svgwidth{0.68\columnwidth}
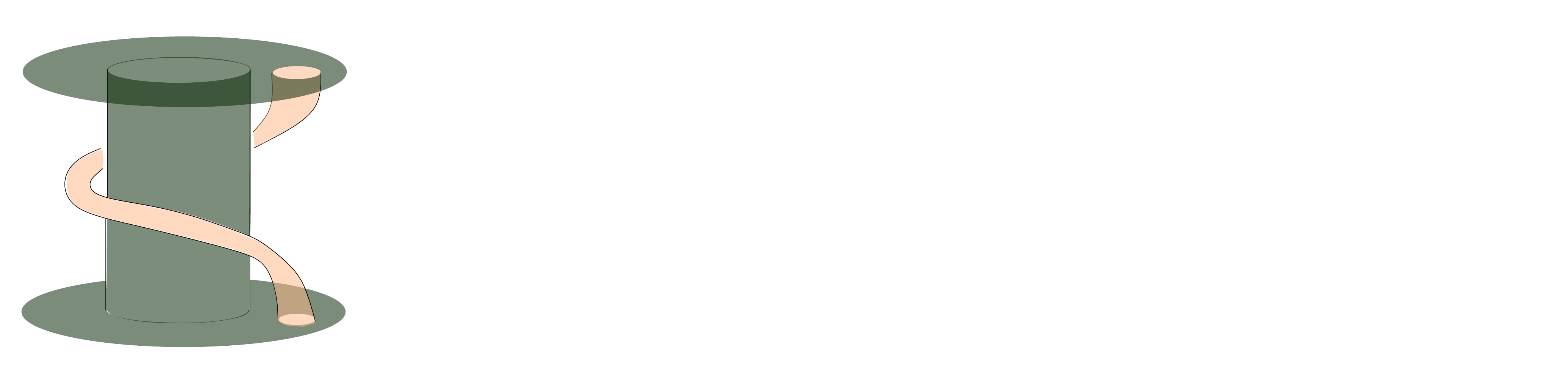
\caption{Signs of a twisting.}
\label{fig:Signs_of_twistings}
\end{figure}
\end{center}

With this in mind, 
we may now describe the generalization of 
Motto's \cite{Mot:90} and Lee-Lee's 
\cite{LeeLee:12}
constructions.
Consider a $3$-manifold $M$ embedded in $\sphere$, 
for example, the closure of the complement of 
a handlebody knot, and suppose 
there exists an oriented annulus $A$ embedded in $\sphere$ 
such that the intersection $A\cap M$ is properly embedded in $M$ 
and diffeomorphic to an annulus 
with some (open) disks removed from $\mathring{A}$, namely 
$A\setminus \bigcup_{i=1}^{n} \mathring{D}_i\subset M$, where 
$\bigcup_{i=1}^{n}D_i
= \mathring{A}\cap  \left(\overline{\sphere\setminus M}\right)$.

Let $\mathfrak{N}(A)$ 
be a tubular neighborhood of $A$ in 
$M \cup\bigcup_{i=1}^{n} \mathfrak{N}(D_{i})$
such that $\mathfrak{N}(A)\cap M$ is 
a tubular neighborhood of the surface
$A\setminus \bigcup_{i=1}^{n} \mathring{D}_{i}$ (properly embedded) in $M$,
and $\mathfrak{N}(A) \cap \left(\overline{\sphere\setminus M}\right)$ consists of 
a tubular neighborhood of $\bigcup_{i=1}^{n}D_i$ in $\overline{\sphere\setminus M}$
and a tubular neighborhood of $\partial A$ in $\partial M$, where $\mathfrak{N}(D_{i})$ is 
a tubular neighborhood of $D_{i}$ in $\overline{\sphere\setminus M}$.
Furthermore, if
one component of $\partial A$ is selected to be the inner circle, 
then $\mathfrak{N}(A)$ can be identified 
with the standard cylindrical shell in $\mathbb{R}^{3}$ described above, and thus one can determine the sign of the twisting.

Now consider 
$$
M_A
:= M \cup \mathfrak{N}(A) \subset \sphere.
$$ 
Then the twist map $t:A\times \closedintervalzeroone\rightarrow A\times \closedintervalzeroone$ in 
\eqref{eq:Dehn_twist}
induces a self-homeomorphism
\[t_{M,A}:M_A \rightarrow M_A,\]
and the composition 
\[M\subset M_A\xrightarrow{t_{M,A}} M_A\subset \sphere\] 
is a new embedding of $M$ in $\sphere$. 
More generally, composing $t_{M,A}$ (resp. $t^{-1}_{M,A}$) with itself $j$ times, 
one gets a self-homeomorphism $t^{\pm j}_{M,A}:M_A\rightarrow M_A$, 
for each $j\in\mathbb{Z}$, and thus an infinite family of embeddings of $M$ in $\sphere$. 
Note that, to produce inequivalent embeddings of $M$, it is necessary that 
$\mathring{A}\cap \left(\overline{\sphere\setminus M}\right)$ is not empty.

\noindent
\begin{example}[\textbf{$\operatorname{HK}5_{1}$}]\label{exa:HK5_1}
The handlebody knot $5_{1}$ in Ishii et al. \cite{IshKisMorSuz:12}, 
denoted by $\operatorname{HK}5_{1}$ in the present paper, 
can be represented as in
Fig. \ref{HK5_1}. 

\begin{center}
\begin{figure}[ht]
\includegraphics[scale=.12]{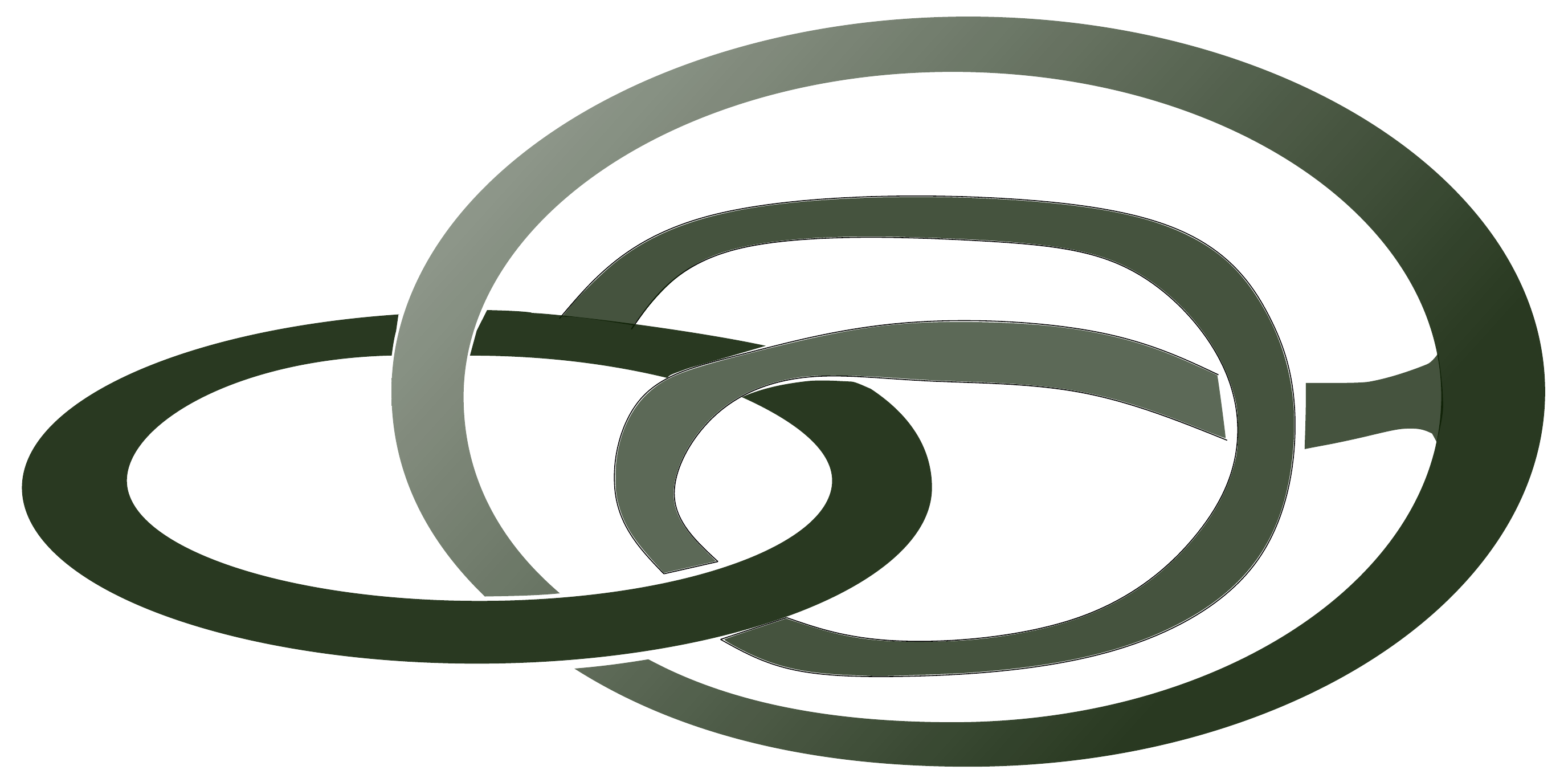}
\caption{The handlebody knot $\HKfiveone$.}
\label{HK5_1}
\end{figure}
\end{center}

Observe that there is an oriented annulus 
$A_{1}$ and a disk $D_{1}$ in $\sphere$ 
such that $A_{1}\setminus \mathring{D}_{1}$ 
is properly embedded in $M$, the closure of 
$\sphere\setminus \operatorname{HK}5_{1}$ 
(Fig. \ref{HK5_1_A1}).
 
\begin{center}
\begin{figure}[ht]
\def\svgwidth{0.68\columnwidth}
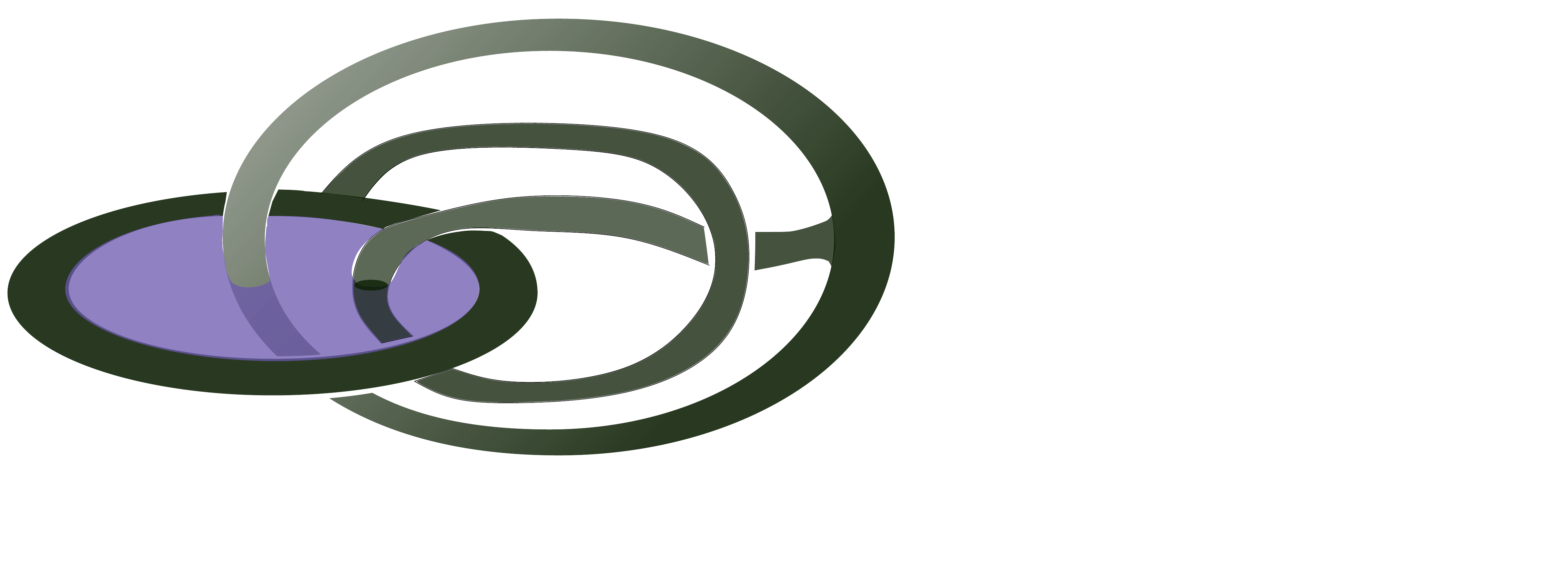
\caption{Annulus $A_{1}$.}
\label{HK5_1_A1}
\end{figure}
\end{center}

Now orient $A_1$ such that
the side with the plus sign is where 
the normal direction goes out, 
and select the obvious component of $\partial A_1$ to be the inner circle. 
Then, applying the twisting map $t_{M,A_1}$, 
we obtain a family of handlebody knots with homeomorphic complements. 
In particular, there are the $-A_{1}$-twisted $\operatorname{HK}5_{1}$ (Fig.
\ref{twistedHK51_A1}, left) and $+A_{1}$-twisted $\operatorname{HK}5_{1}$ 
(Fig. \ref{twistedHK51_A1}, right) 
obtained by applying $t^{-1}_{M,A_{1}}$ and $t_{M,A_{1}}$, respectively. 
 
\begin{figure}[ht]
    \centering
    \subfloat{ {\includegraphics[scale=.12]{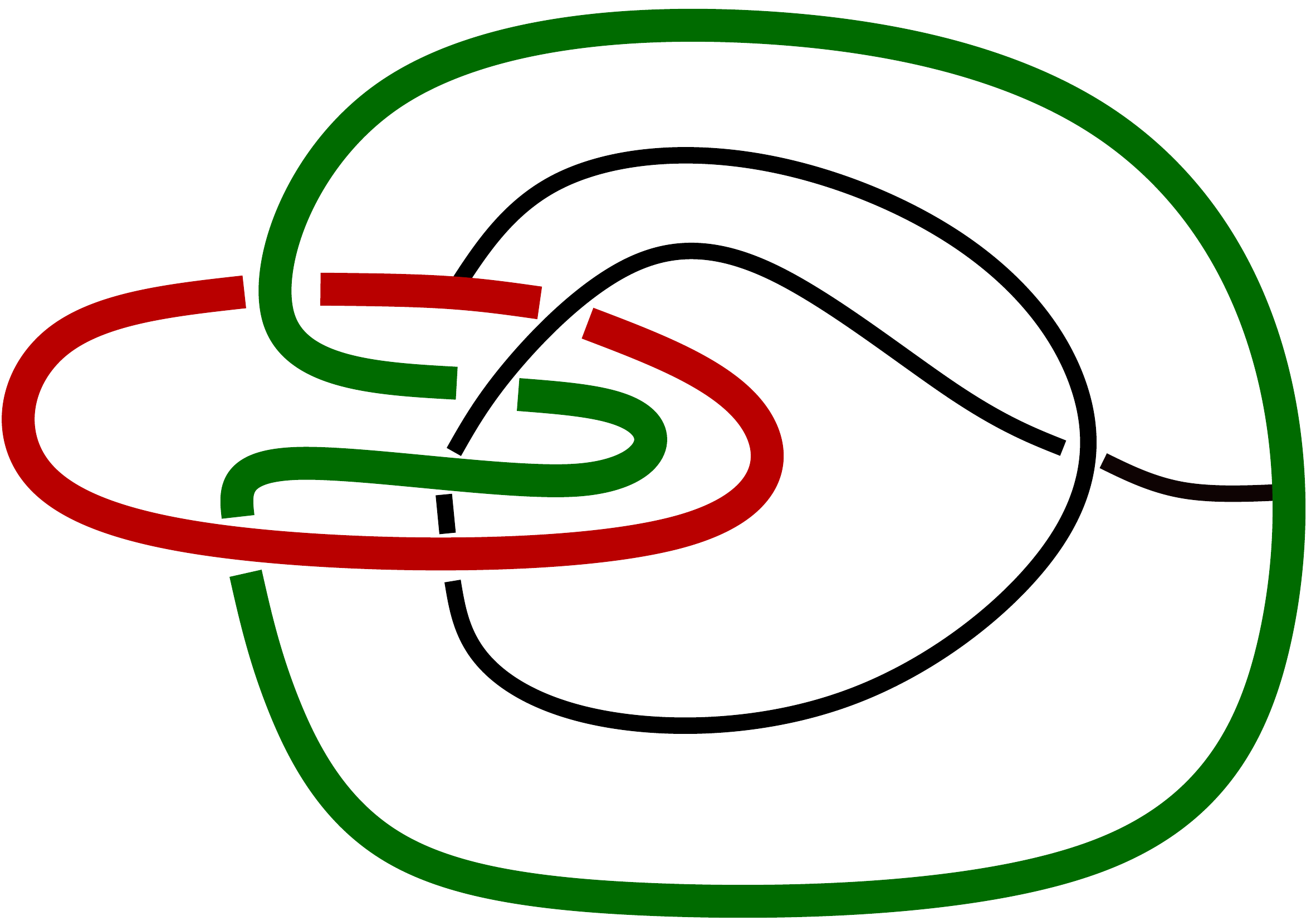}} }%
    \qquad
    \subfloat{ {\includegraphics[scale=.12]{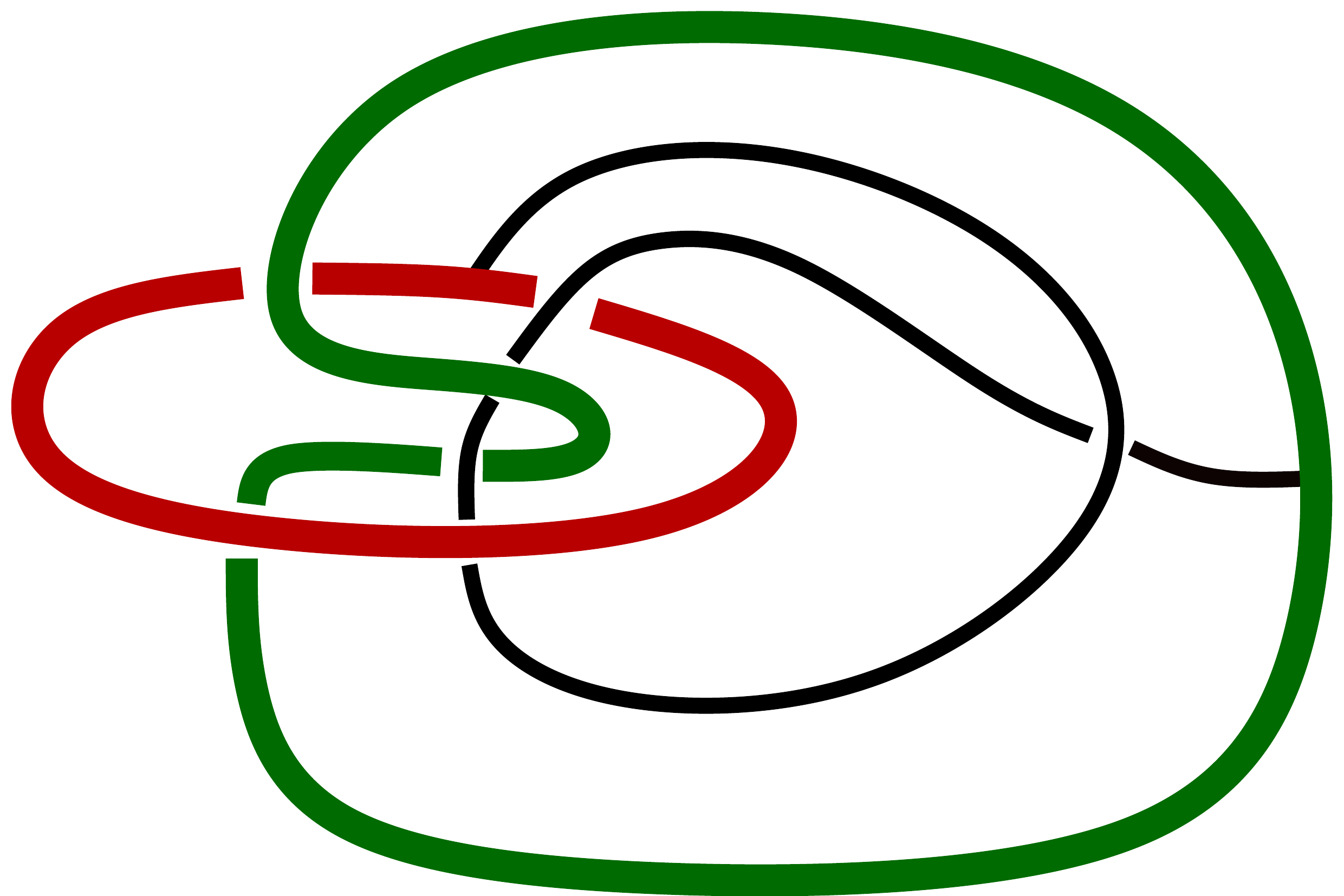}} }%
    \caption{$\mp A_{1}$-twisted $\operatorname{HK}5_1$ 
(Example \ref{exa:HK5_1}).}%
    \label{twistedHK51_A1}%
\end{figure}

%\draftMMM{[It seems it is p. 1062 $(a)$ and $(b)$ instead...]}
These three handlebody knots are in fact 
equivalent to the handlebody knots $V_{0}$, $V_{-1}$ and $V_{+1}$ in \cite{LeeLee:12}
since by the moves described in \cite[p.1062 $(a)$]{LeeLee:12}   
the annulus $A_{1}$ can be deformed into the annulus used there.
In particular, $-A_{1}$-twisted $\operatorname{HK}5_{1}$ is equivalent to
the handlebody $\operatorname{HK}6_{4}$ of \cite{IshKisMorSuz:12}.
%\draftMMM{[Added reference to HK6-4]}

There is another oriented annulus $A_{2}$ and a disk 
$D_{2}$ embedded in $\sphere$ such that $A_{2}\setminus \mathring{D}_{2}$ is properly embedded in $\sphere\setminus\operatorname{HK}5_1$ as illustrated in 
Fig. \ref{HK5_1_A2}---we select the bigger circle in $\partial A_2$
in Fig. \ref{HK5_1_A2} to be the inner circle. 

\begin{center}
\begin{figure}[ht]
\def\svgwidth{0.68\columnwidth}
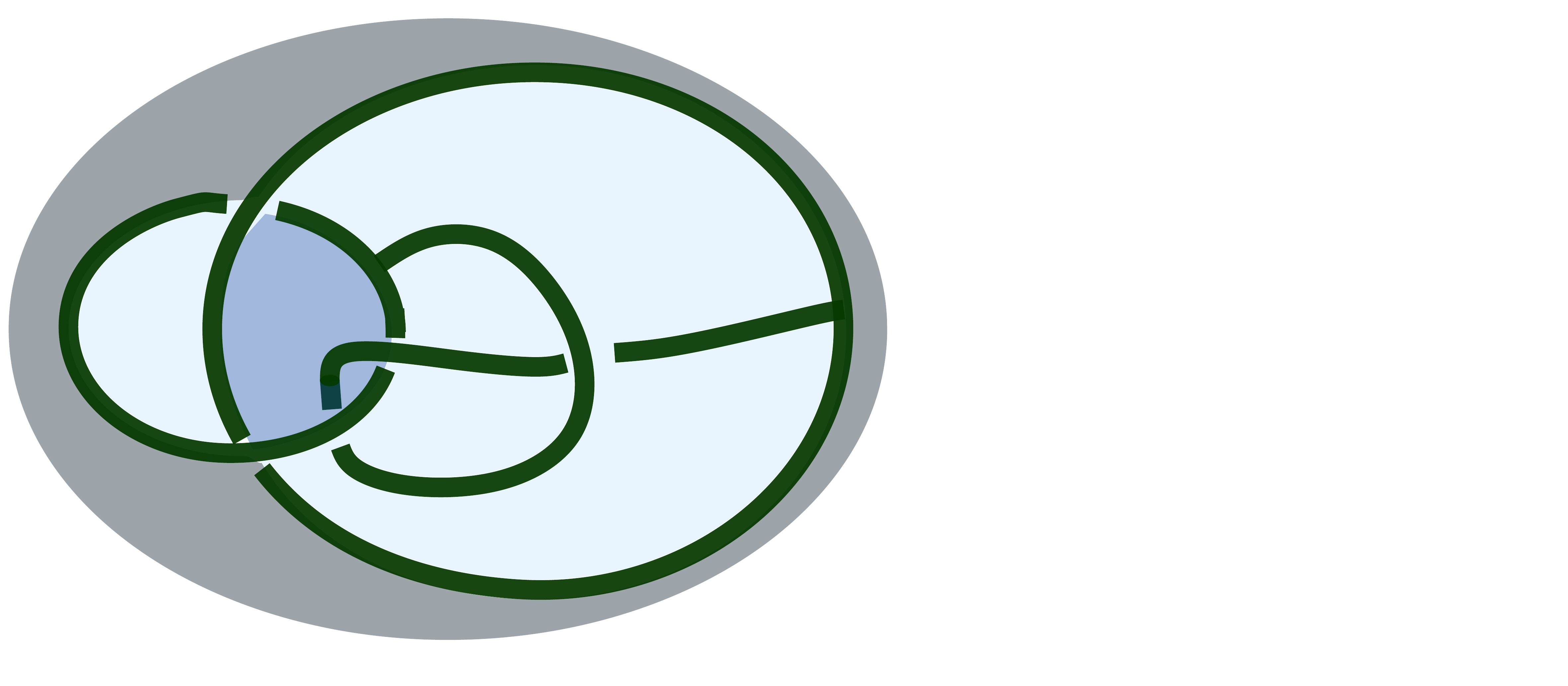
\caption{Annulus $A_{2}$.}
\label{HK5_1_A2}
\end{figure}
\end{center}

Applying $t_{M,A_{2}}$ to $\sphere\setminus\operatorname{HK}5_1$, 
we obtain another family of handlebody knots. 
Especially, there are
$\pm A_2$-twisted $\HKfiveone$ (see Fig. \ref{twistedHK51_A2};
left for the $-A_2$-twisted $\operatorname{HK}5_1$ and right for the other one).

\begin{figure}[ht]
    \centering
    \subfloat{ {\includegraphics[scale=.1]{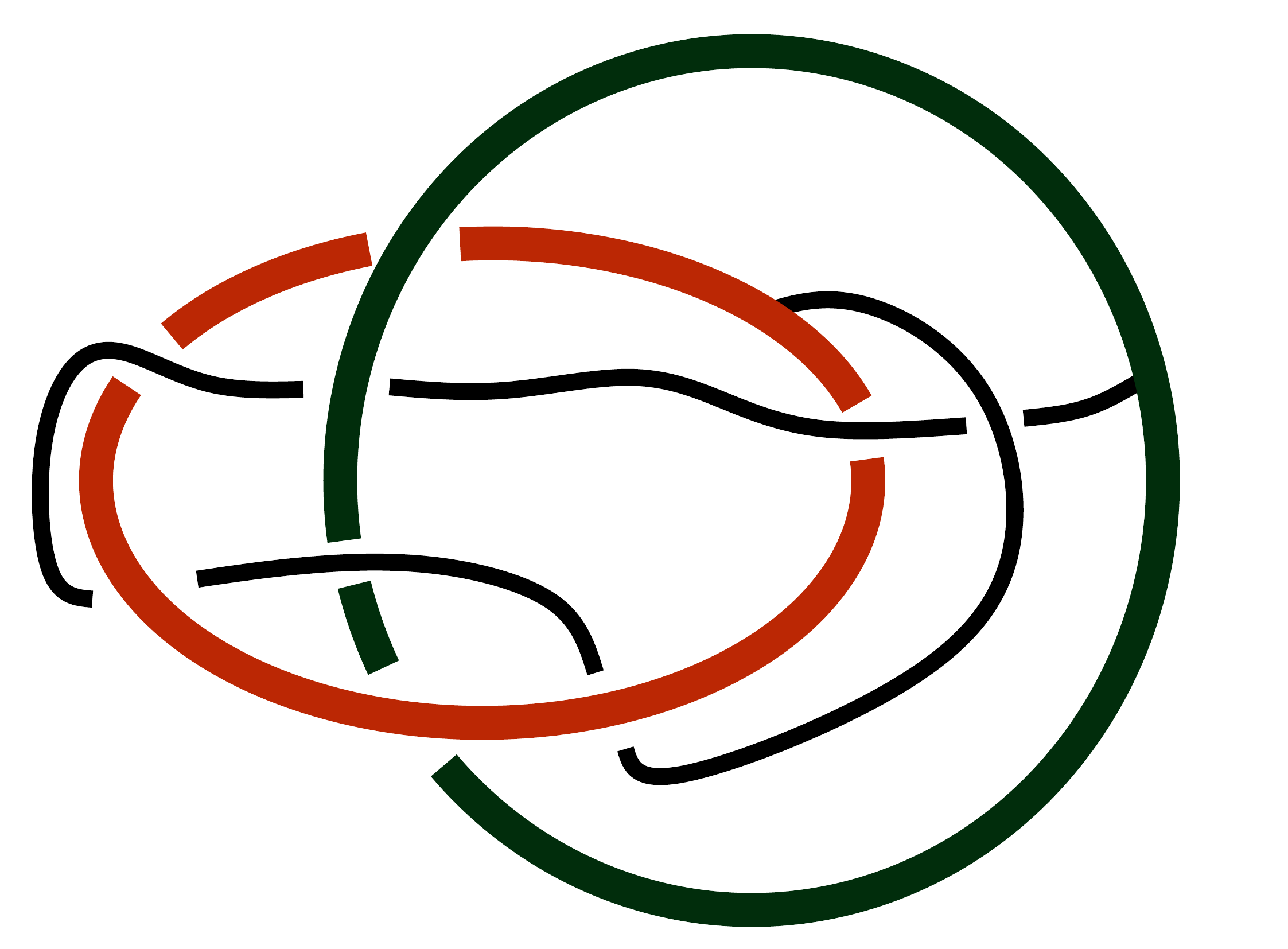}} }%
    \qquad
    \subfloat{ {\includegraphics[scale=.1]{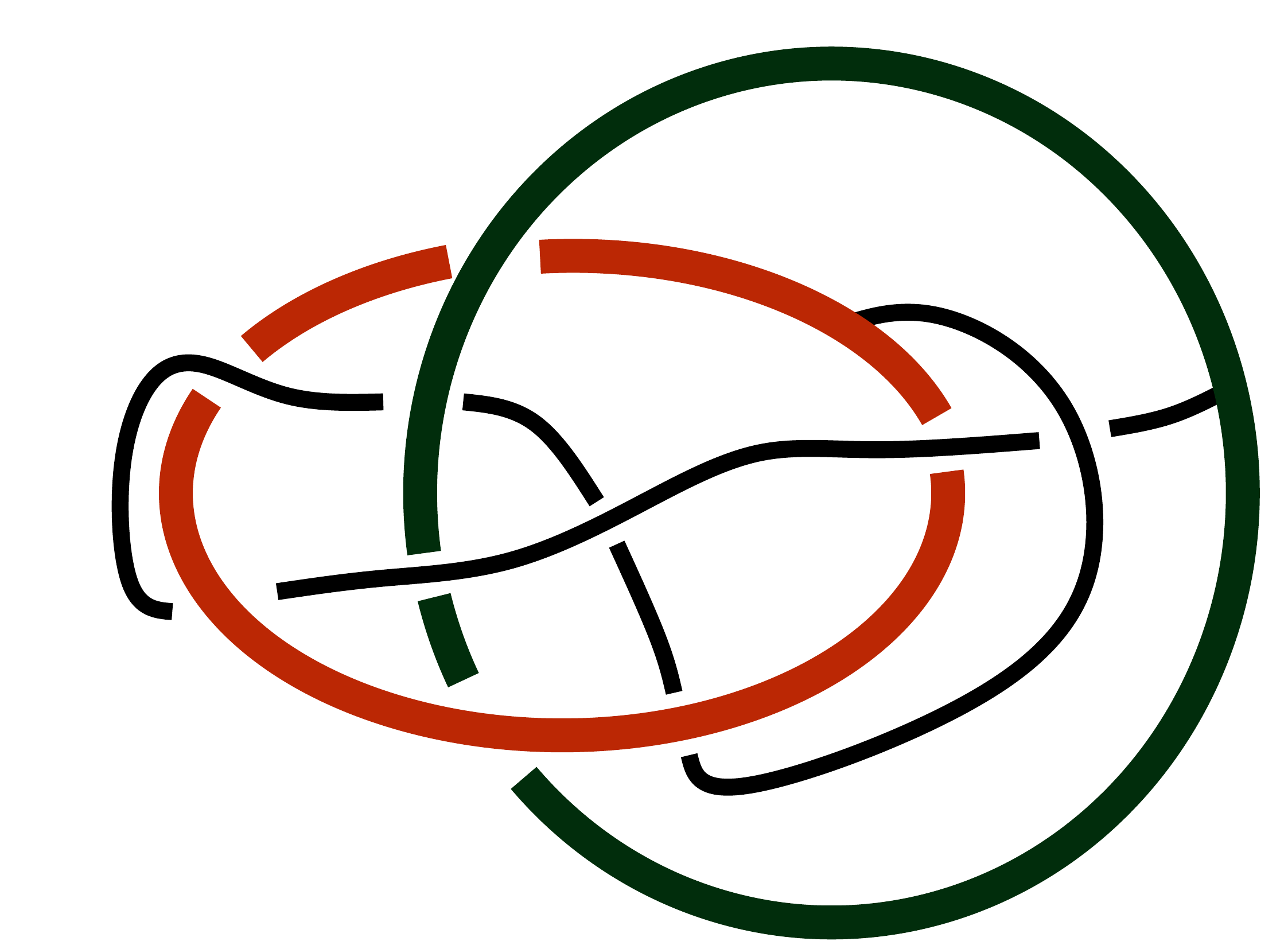}} }%
    \caption{$\mp A_{2}$-twisted $\operatorname{HK}5_1$
(Example \ref{exa:HK5_1}).}%
    \label{twistedHK51_A2}%
\end{figure}
  
\end{example}
 
\begin{example}[\textbf{$\operatorname{HK}6_{2}$}]\label{exa:HK6_2}
For our second set of examples, we consider the handlebody 
knot $\operatorname{HK}6_{2}$ 
%\ref{HK6_2}
corresponding to $\operatorname{6}_{2}$ 
in Ishii et al.'s knot table \cite{IshKisMorSuz:12},
and observe that there are two embedded oriented annuli $A_{1}$ and $A_2$ 
in $\sphere$ with their interior intersecting with  $\operatorname{HK}6_2$ at disk $D_{1}$
and $D_2$, respectively (see Fig. \ref{HK6_2_A1_A2}; there 
is an obvious choice for the inner circle for $A_1$, whereas
for $A_2$, we identify the horizontal one as the inner circle of 
a standard annulus).

\begin{center}
\begin{figure}[ht]
\includegraphics[scale=.12]{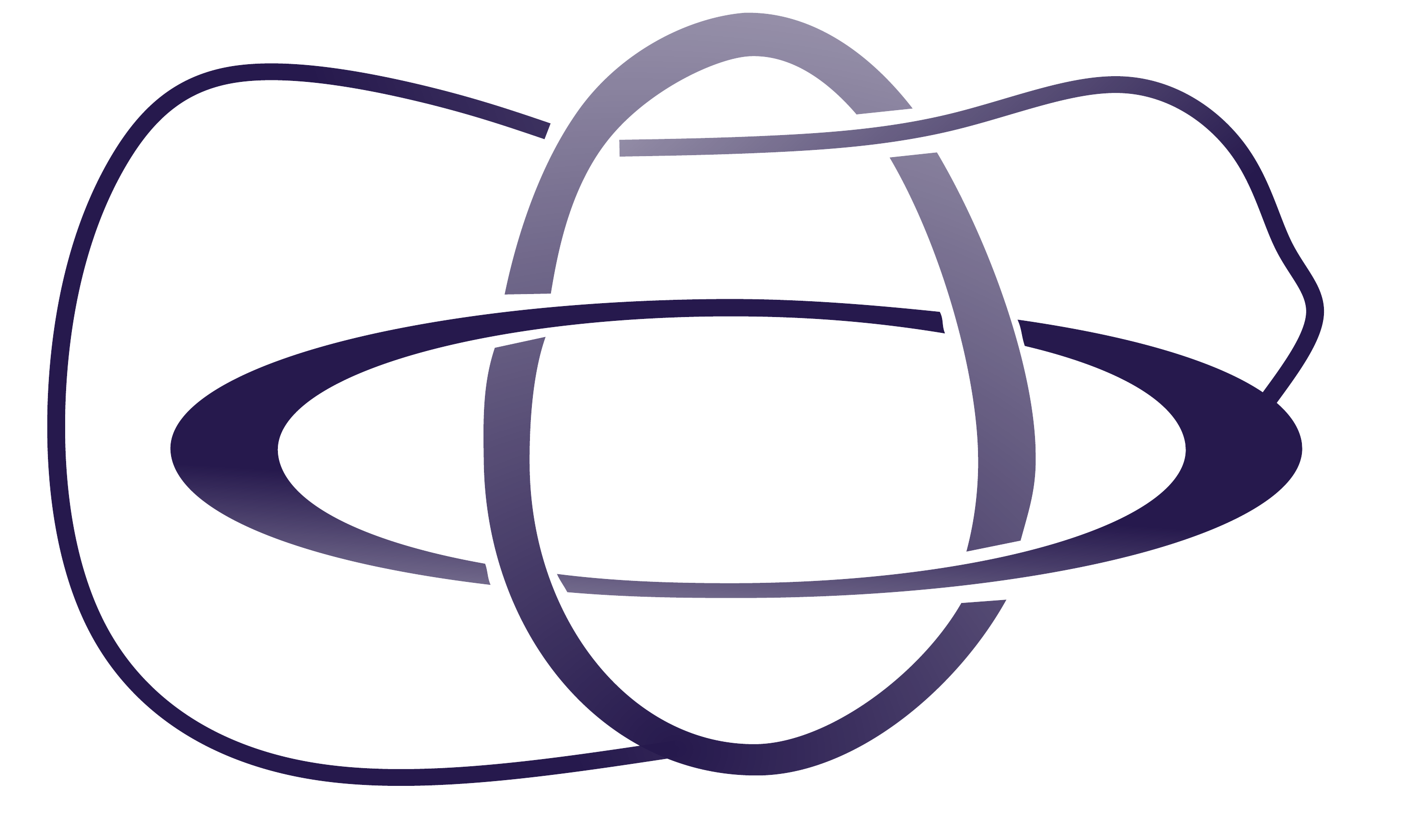}
\caption{The handlebody knot $\operatorname{HK}{6}_{2}$.}
\label{HK6_2}
\end{figure}
\end{center}

\begin{figure}[ht]
    \centering
    \subfloat{ \def\svgwidth{0.41\columnwidth}
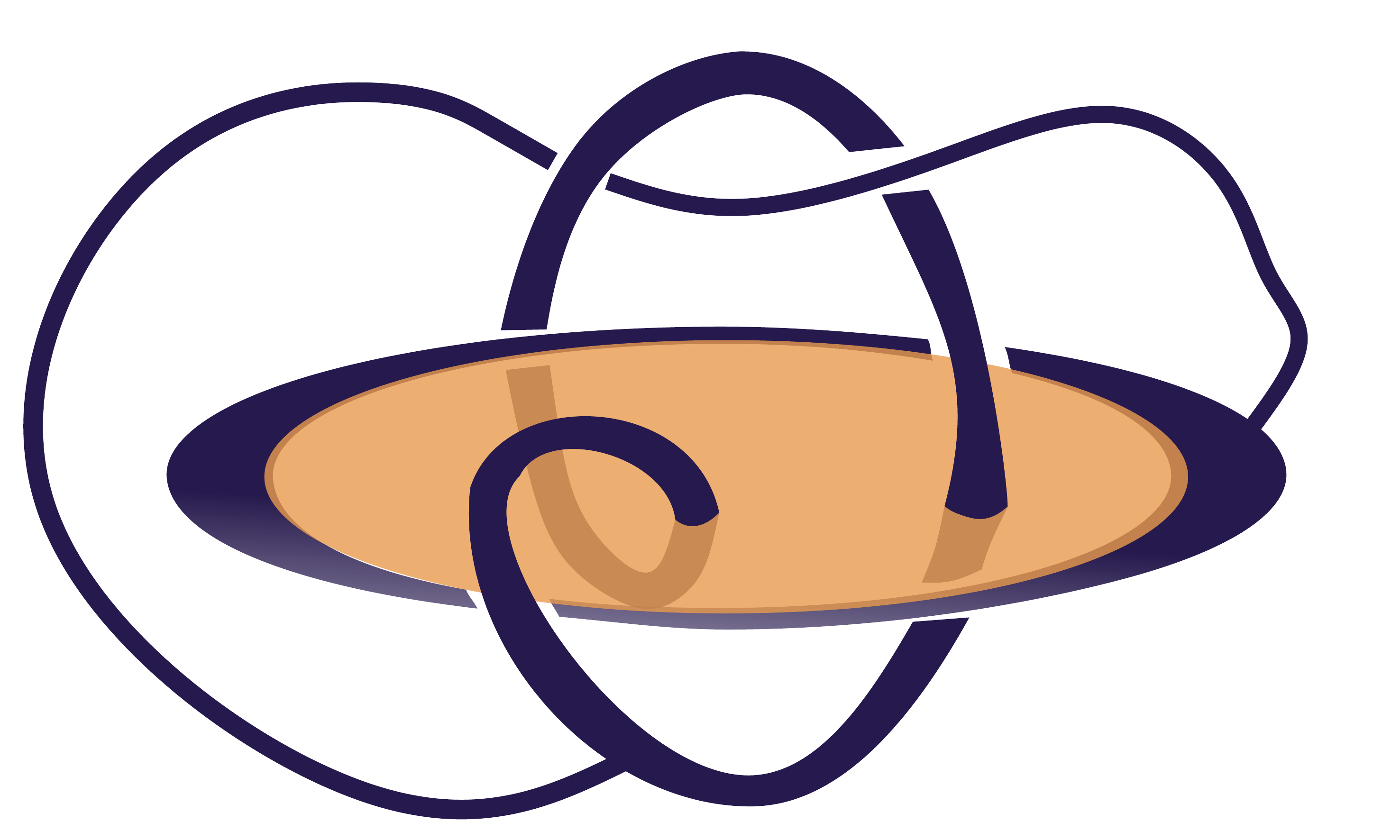
    %{\includegraphics[scale=.1]{HK6_2_A1.pdf}} 
    }%
    \qquad
    \subfloat{ 
    \def\svgwidth{0.41\columnwidth}
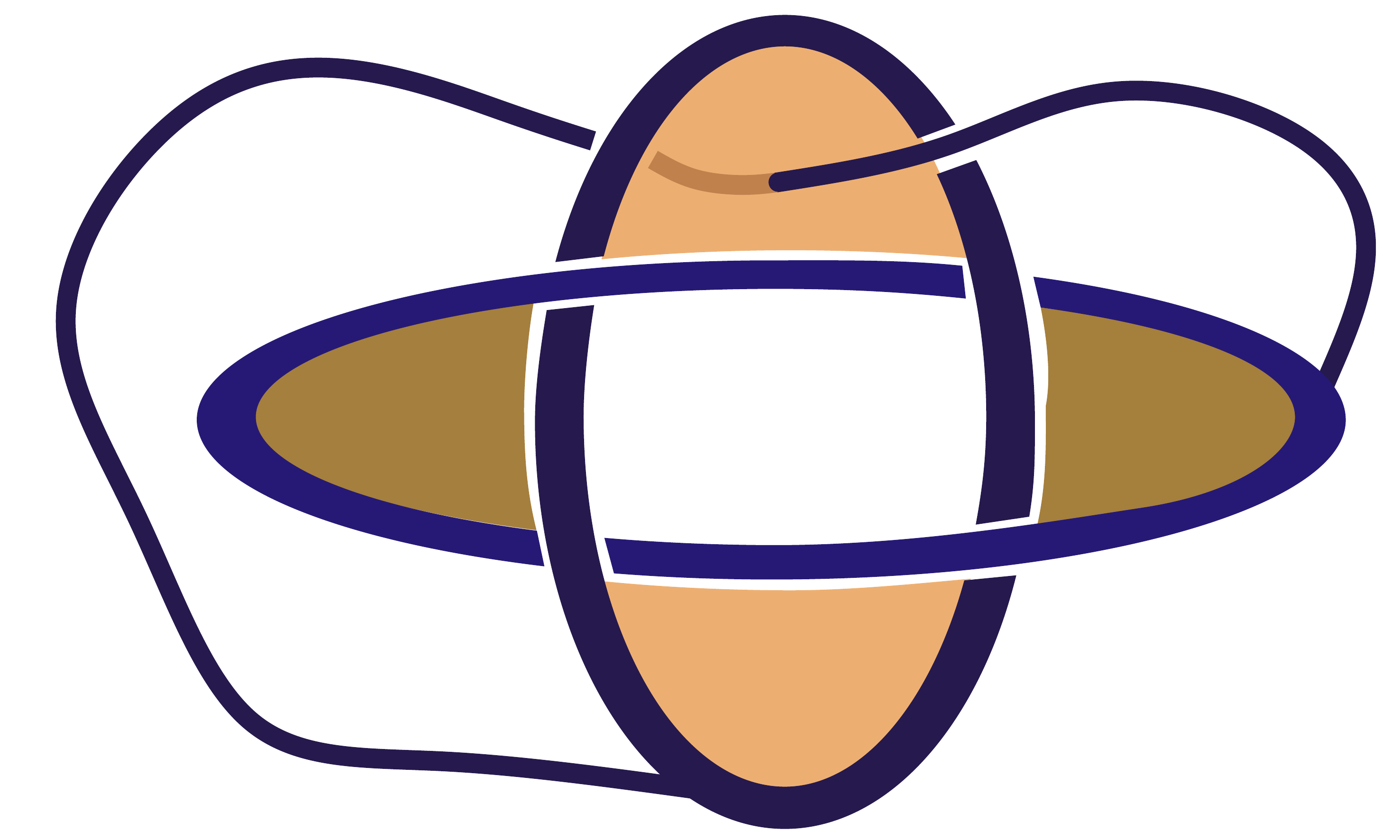     
    %{\includegraphics[scale=.1]{HK6_2_A2.pdf}} 
    }%
    \caption{Annuli $A_1$ and $A_2$.}%
    \label{HK6_2_A1_A2}%
\end{figure}

%\begin{center}
%\begin{figure}[ht]
%\includegraphics[scale=.12]{HK6_2_A1.pdf}
%\caption{Annulus $A_{1}$}
%\label{HK6_2_A1}
%\end{figure}
%\end{center}

Applying the twist construction to the annuli $A_{1}$ and $A_2$, 
we get two families of handlebody knots with homeomorphic complements. 
We record $\mp A_1$-twisted $\operatorname{HK}6_2$ and
$\mp A_2$-twisted $\operatorname{HK}6_2$ in Fig. \ref{twistedHK62_A1} 
and \ref{twistedHK62_A2} 
\begin{figure}[ht]
    \centering
    \subfloat{ {\includegraphics[scale=.1]{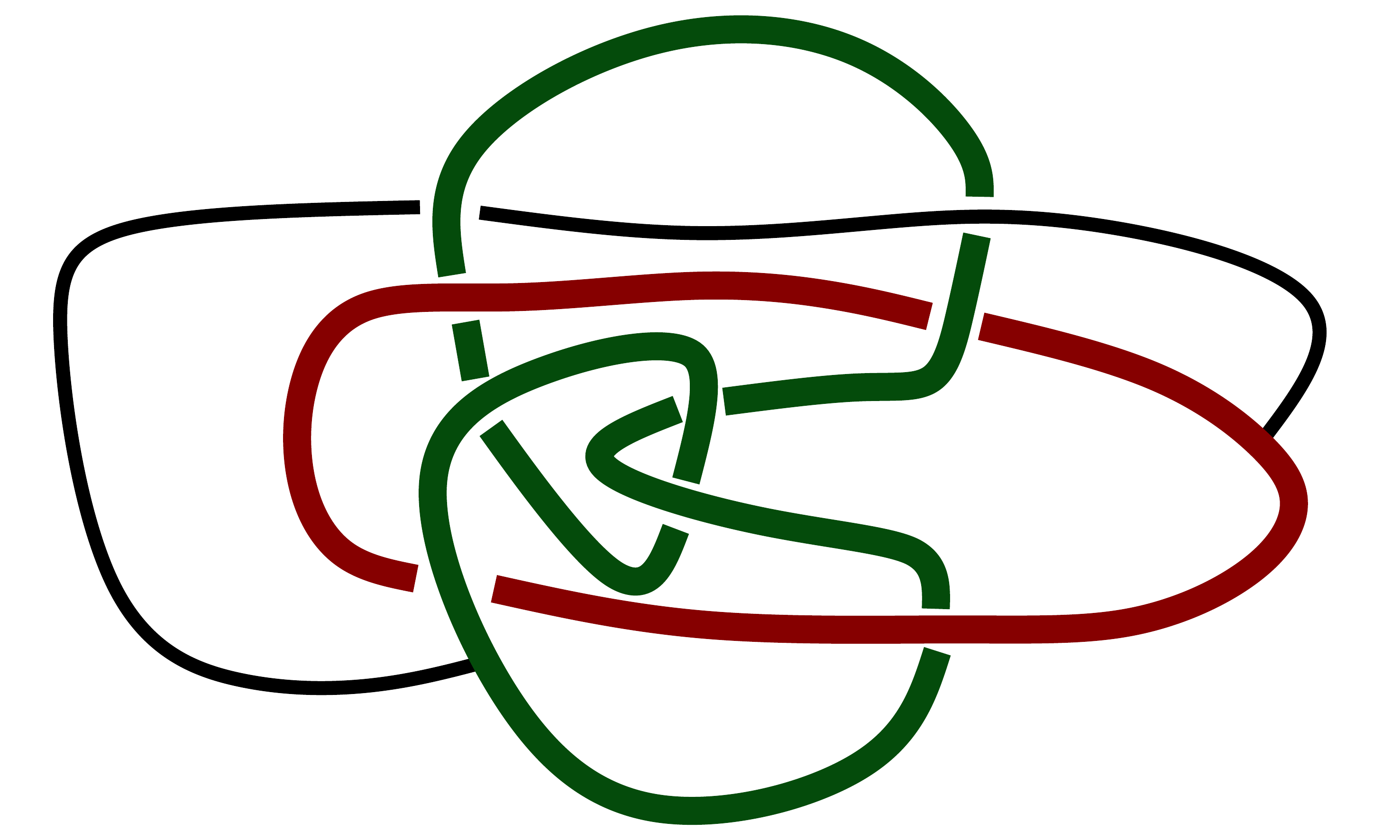}} }%
    \qquad
    \subfloat{ {\includegraphics[scale=.1]{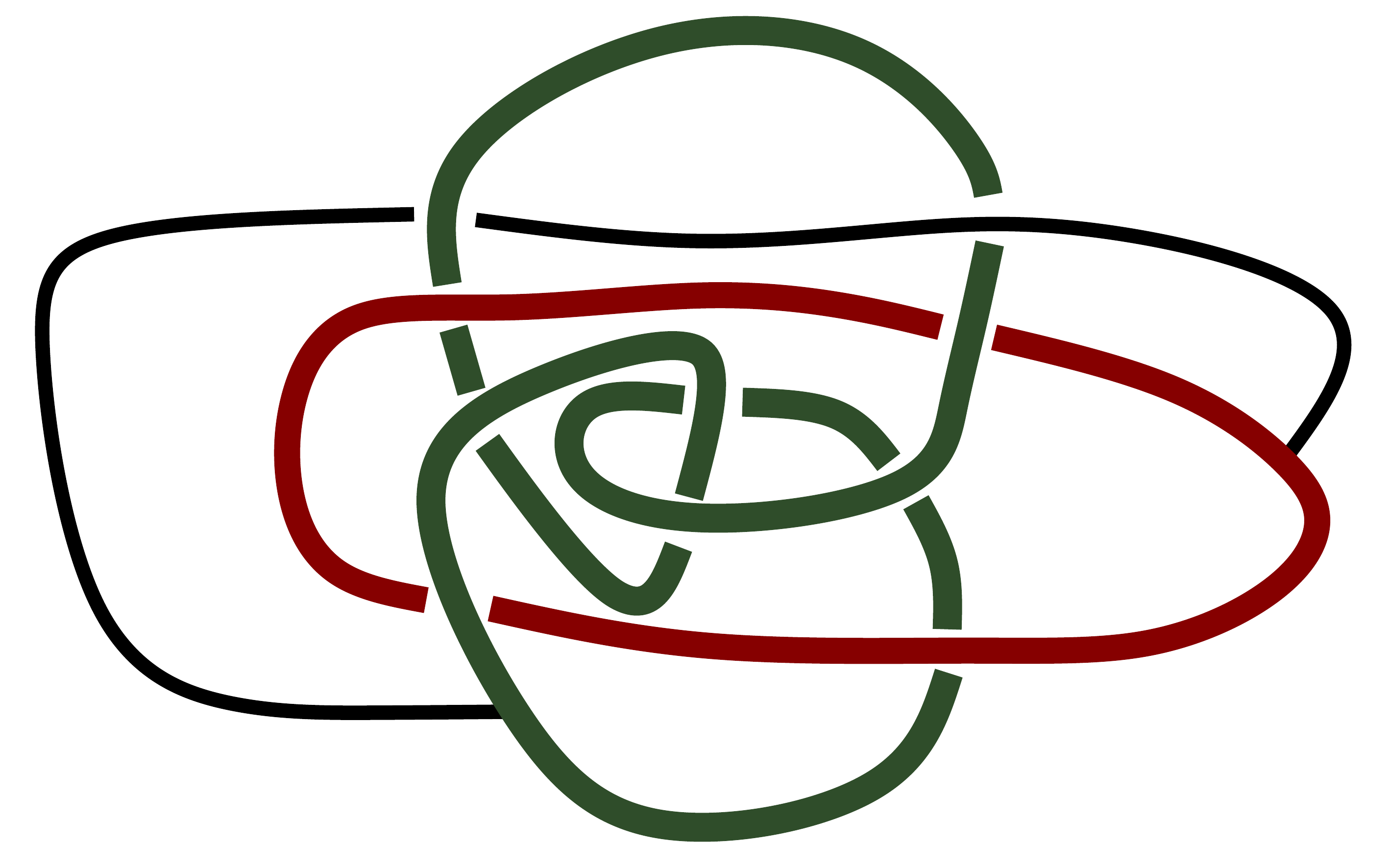}} }%
    \caption{left: $-A_{1}$-twisted $\operatorname{HK}6_2$; right: $+A_1$-twisted 
    $\operatorname{HK}6_2$.}%
    \label{twistedHK62_A1}%
\end{figure}
%
% 
%
%\begin{center}
%\begin{figure}[ht]
%\includegraphics[scale=.12]{HK6_2_A2.pdf}
%\caption{Annulus $A_{2}$}
%\label{HK6_2_A2}
%\end{figure}
%\end{center}
%
%o\noindent
and get the corresponding $-A_{2}$-twisted $\operatorname{HK}6_2$ (Fig. \ref{twistedHK62_A2}, left) and $+A_{2}$-twisted $\operatorname{HK}6_2$ 
(Fig. \ref{twistedHK62_A2}, right).

\begin{figure}[ht]
    \centering
    \subfloat{ {\includegraphics[scale=.1]{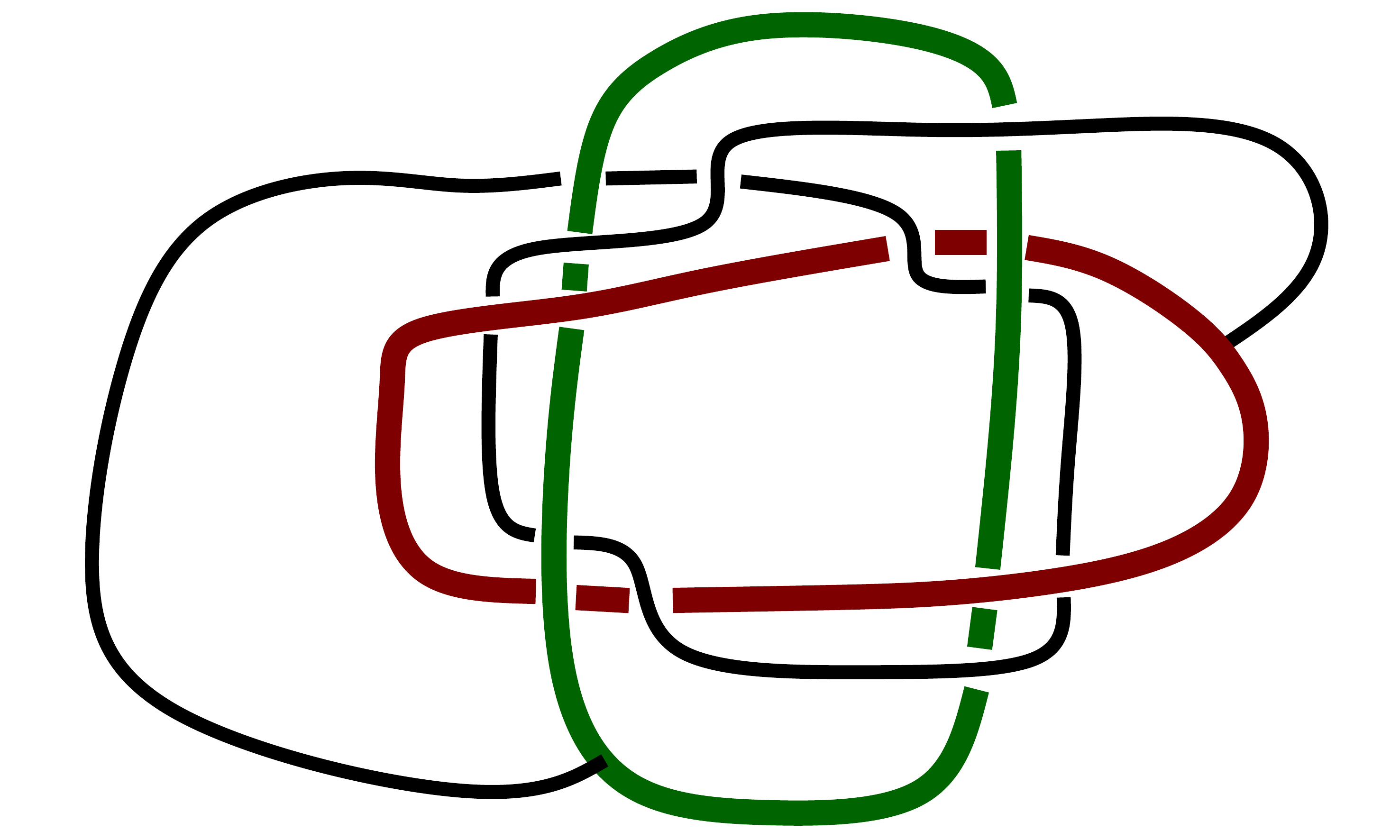}} }%
    \qquad
    \subfloat{ {\includegraphics[scale=.1]{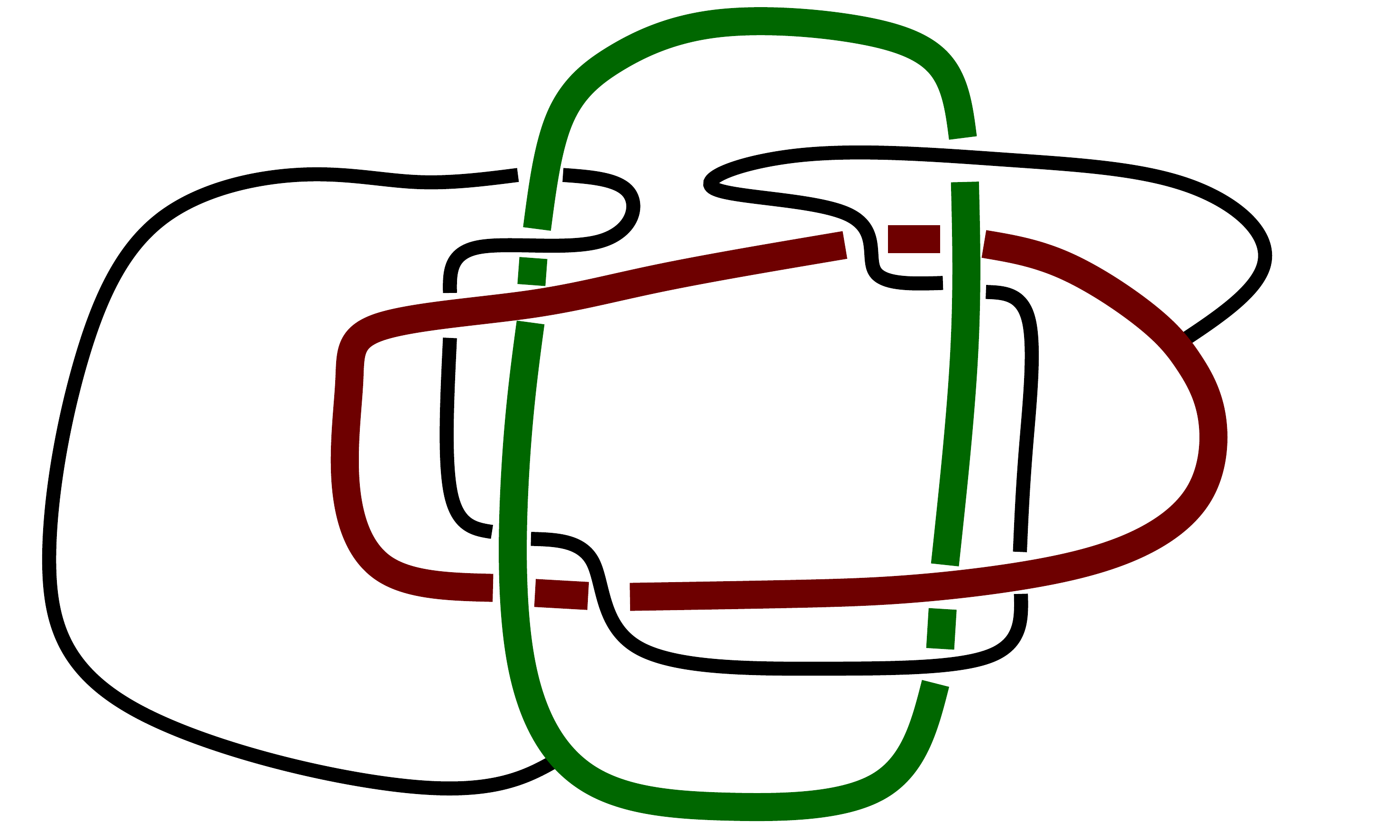}} }%
    \caption{left: $- A_{2}$-twisted $\operatorname{HK}6_2$; right: $+A_2$-twisted
    $\operatorname{HK}6_2$.}%
    \label{twistedHK62_A2}%
\end{figure}

\end{example}

In Subsection \ref{subsec:the_g_image_of_meridians} (Tables \ref{The_A_5_image_HK5_1}, \ref{The_A_5_image_HK6_2}),
we shall prove the following results by computing invariants derived from Theorem \ref{teo:complete_invariant}.

\begin{proposition}\label{Inequivalent_HK_with_homeo_complement}
The following holds.\\
$\bullet$ $\operatorname{HK}5_1$, $-A_1$-twisted $\operatorname{HK}5_1$, $+A_1$-twisted $\operatorname{HK}5_1$ and $+A_{2}$-twisted $\operatorname{HK}5_1$ 
%have different knot types ({\it i.e.}, they 
are not ambient isotopic.

\noindent
$\bullet$
$-A_{2}$-twisted $\operatorname{HK}5_1$ is 
ambient isotopic to $\operatorname{HK}5_1$.

\noindent  
$\bullet$ $+A_{1}$-twisted $\operatorname{HK}5_1$ has crossing number $=7$.

\noindent
$\bullet$ 
Among $\operatorname{HK}6_2$, $\mp A_{1}$-twisted $\operatorname{HK}6_2$ and $\mp A_{2}$-twisted $\operatorname{HK}6_2$, 
there are at least three inequivalent handlebody knots.

\noindent
$\bullet$ $-A_{1}$-twisted $\operatorname{HK}6_2$ has crossing number $=7$.
\end{proposition}

\subsection{Unswappable scenes}\label{subsec:Unswappable_scenes}
In this short subsection we present a construction of unswappable scenes of 
genus $3$ and prove Theorem \ref{Unswappable_scenes_of_genus_3}.

Let $(\inside,\Sigma,\outside)$ be a handlebody knot, and suppose there exists a loop $l$ on
$\Sigma$ which intersects with only one meridian $m$ in a complete system
of meridians of $\inside$ and bounds a properly embedded disk in $\sphere\setminus H_1$,
where $H_1\subset \inside$ is the solid torus induced by the loop $l$ and the meridian disk 
bounded by $m$. Then the complement of such a handlebody knot 
can be expressed as a solid torus with some tunnels in it.
For instance, the handlebody knot $\operatorname{HK}5_{1}$ has such a loop:
\begin{center}
\begin{figure}[ht]
\def\svgwidth{0.95\columnwidth}
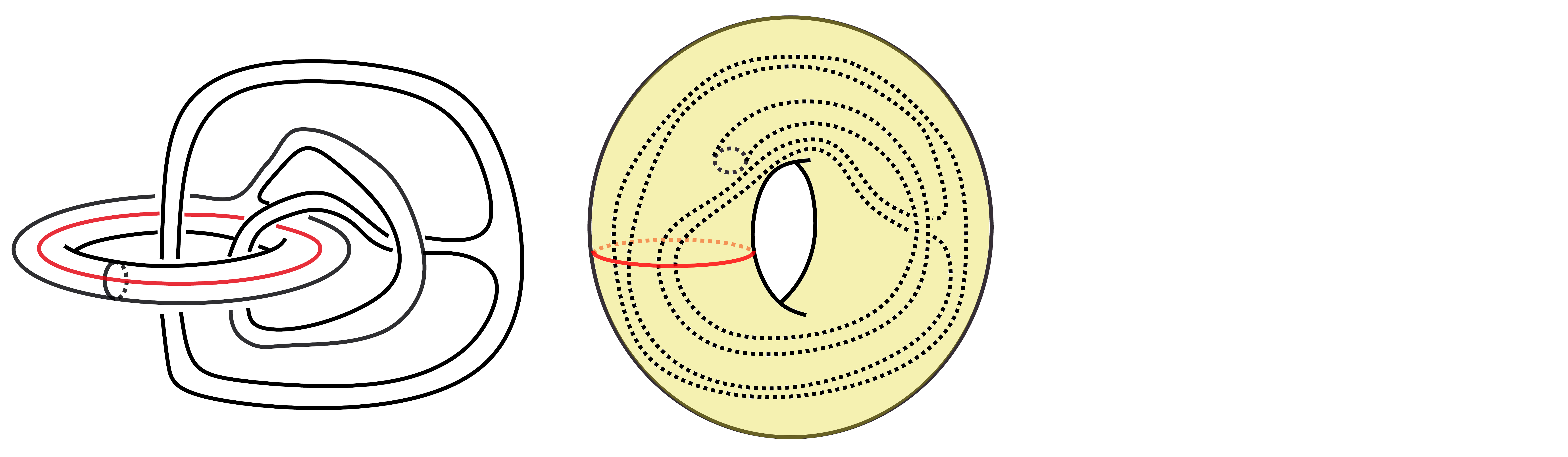
\caption{The complement of $\operatorname{HK}5_1$ as a solid torus with tunnels in it (middle); 
the tunnel expression for the complement of $+A_1$-twisted $\operatorname{HK}5_1$ (right).}
\label{HK5_1_tunnel_view}
\end{figure}
\end{center}
the annulus $A_1$ in Fig. \ref{HK5_1_A1} is bounded by $l$, and hence
$+A_{1}$-twisted $\operatorname{HK}5_1$ can be obtained by twisting the 
two tubes encircled by $l$; its complement as a solid torus with tunnels is 
depicted in Fig. \ref{HK5_1_tunnel_view} (right).

\begin{example}[\textbf{Unswappable prime scenes of genus $3$}]
To construct an unswappable prime scene of genus $3$, 
we start with a trivial scene of genus $1$ 
(Fig. \ref{unswappable_prime_scene_step1}, left).
Next, we grow a solid-torus-shaped tree 
such that the resulting object is the handlebody 
knot $\operatorname{HK}5_{1}$ (Fig. \ref{unswappable_prime_scene_step1}, right).
\begin{center}
\begin{figure}[ht]
\includegraphics[scale=.08]{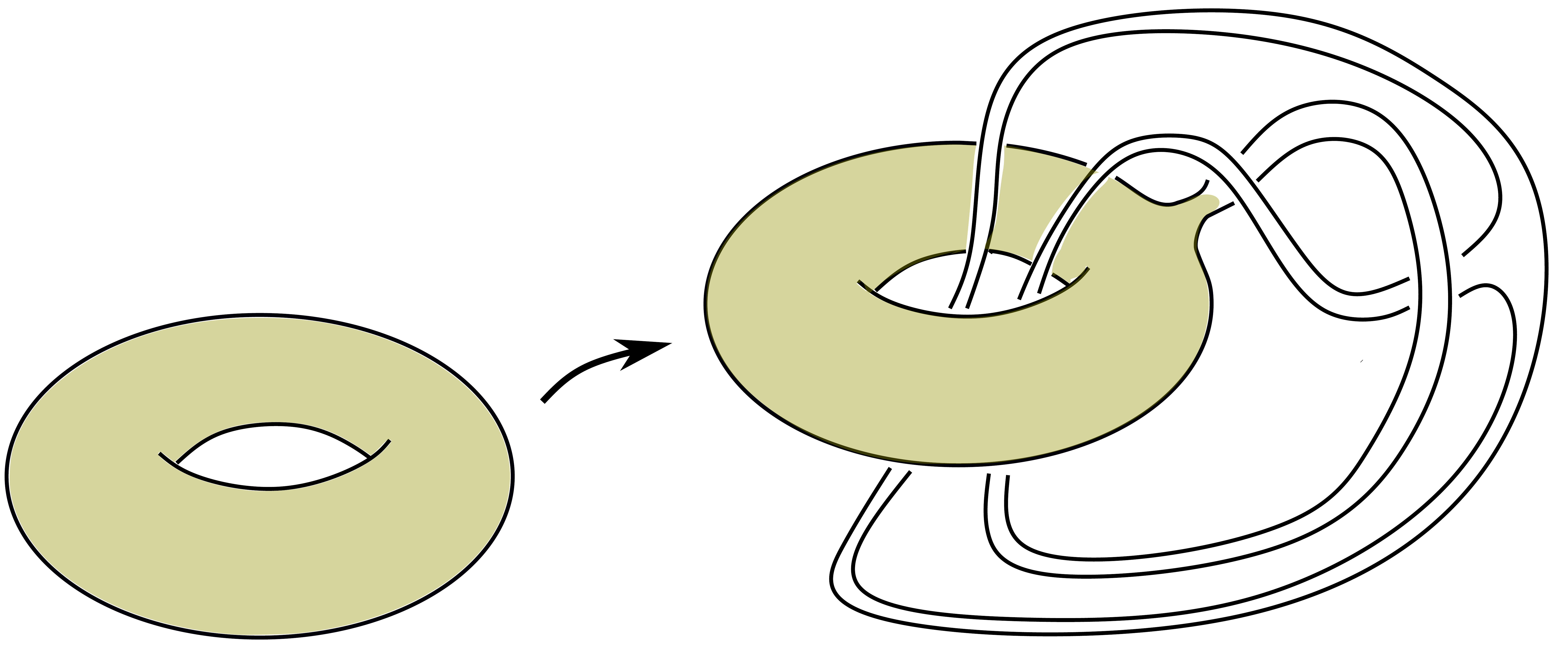}
\caption{$\operatorname{HK}5_1$ as a tree on a solid torus.}
\label{unswappable_prime_scene_step1}
\end{figure}
\end{center}

\noindent
Then, we dig a tunnel (Fig. \ref{unswappable_prime_scene_step2}) into the original 
solid torus in such a way that, 
without the tree, the resulting object
in $\sphere$ is the tunnel expression for the 
complement of $+A_{1}$-$\operatorname{HK}5_1$.
\begin{center}
\begin{figure}[ht]
\includegraphics[scale=.08]{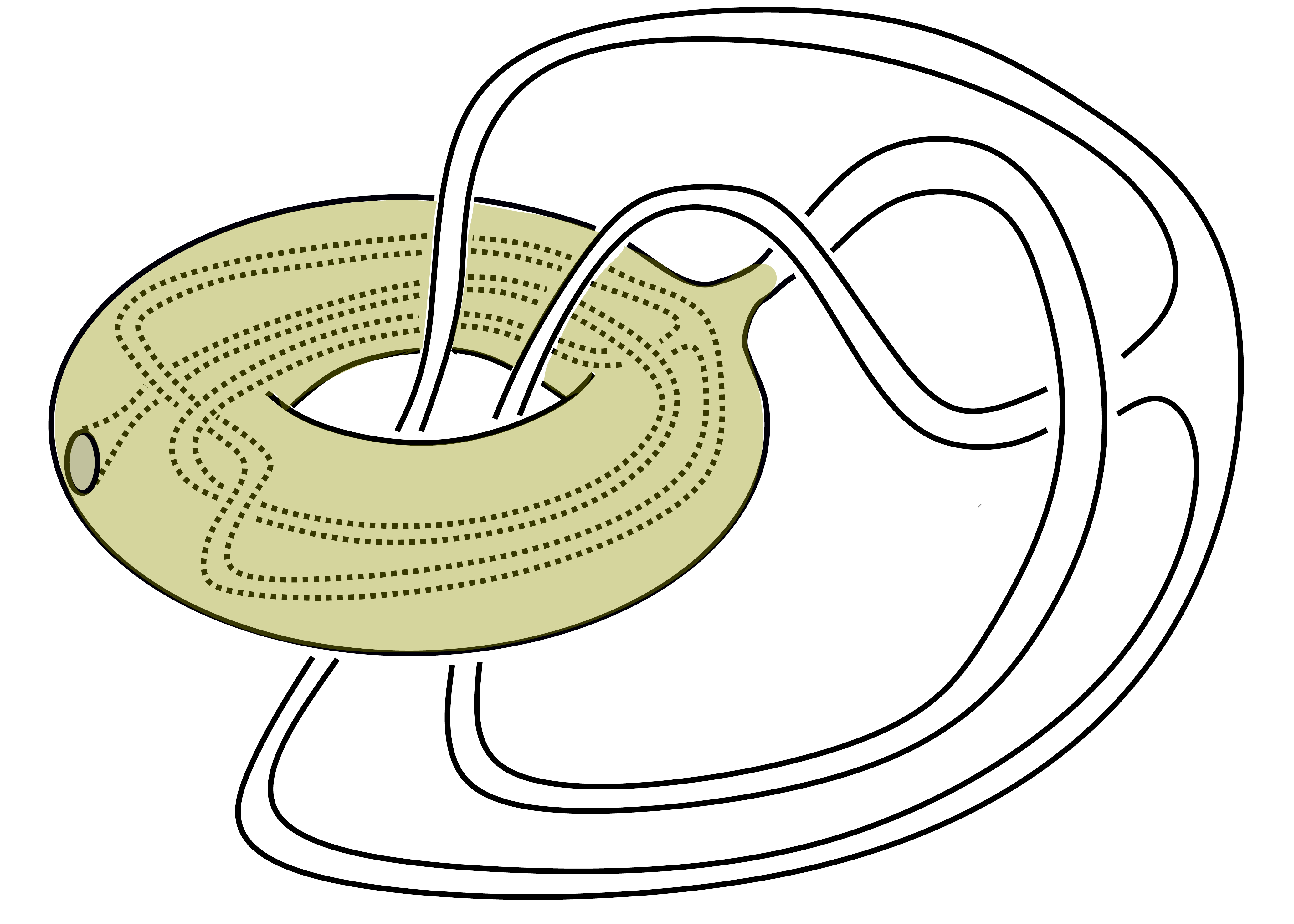}
\caption{Unswappable prime scene of genus $3$.}
\label{unswappable_prime_scene_step2}
\end{figure}
\end{center}

\noindent
Denote the resulting connected scene by $\SSS=(V,\Sigma,W)$.
From the construction, it is clear that both $V$ and $W$ are homeomorphic
to the connected sum of a solid torus and the complement of $\operatorname{HK}5_1$.
Hence, it is a symmetric scene. 

To see it is not swappable, we note that
any diffeomorphism between $V$ and $W$ sends the meridian of $V$ to 
the meridian of $W$ \cite[Corollary $3.6$]{Suz:75}. 
Removing these meridians from $W$ and $V$ in $\SSS$, one gets
a $\operatorname{HK}5_{1}$ and a swapped $+A_1$-$\operatorname{HK}5_1$,
respectively. In particular, if there is an equivalence between
$(V,\Sigma,W)$ and $(W,-\Sigma,V)$, then it induces an equivalence between
$\operatorname{HK}5_1$ and $+A_1$-$\operatorname{HK}5_1$, which contradicts
Proposition \ref{Inequivalent_HK_with_homeo_complement}.

Now, suppose $\SSS$ is not prime and there 
exists a decomposition $\SSS=\SSS_1\#\SSS_2$ with $\SSS_i$ a non-trivial scenes of genus $i$, 
$i=1,2$. Let $\Sbb^2\subset\sphere$ be the separating $2$-sphere of 
the decomposition. Then the disk $\Sbb^2\cap V$ separates $V$ into a solid torus
and the complement of $+A_1$-$\operatorname{HK}5_1$ and the disk $\Sbb^2\cap W$
separates $W$ into a solid torus and the complement of $\operatorname{HK}5_1$. 
So, $\SSS_2=(V_2,\Sigma_2,W_2)$ is such that both 
$V_2$ and $W_2$ are $\partial$-irreducible
manifolds, which contradicts to Fox's Theorem \cite[p.462 (2)]{Fox:48} 
(see \cite[Proposition $2.5$]{Suz:75}). Therefore, $\SSS$ is prime.

\end{example}

This construction works not only for $\operatorname{HK}5_1$ but also
for other handlebody knots admitting the loop $l$ described at the beginning 
of the subsection. Combining Motto's or Lee-Lee's results with the construction, 
we see there are infinitely many unswappable prime scenes of genus $3$; 
this proves Theorem \ref{Unswappable_scenes_of_genus_3}.

\subsection{Bi-knotted scenes}\label{subsec:biknotted}

The next examples concern bi-knotted scenes---namely, 
both ``inside" and ``outside" are not handlebodies.
To construct such examples, we introduce a satellite 
construction for handlebody knots.
\begin{definition}[\textbf{Transverse disks}]
\label{def:potentially_explosive_disk}
Given a handlebody knot $(\inside,\Sigma,\outside)$, that is, $\inside$ 
being a handlbody, a transverse 
disk $D$ of $(\inside,\Sigma,\outside)$ is a disk in
$\sphere$ such that 
$\mathring{D} \cap \inside$ is the union 
%$\coprod_{i=1}^{n}D_{i}$ 
of $n$ disjoint disks $D_1,\cdots,D_n$
properly embedded in $\inside$, and 
$D\setminus \bigcup_{i=1}^{n}\mathring{D}_{i}$
an $n$-times punctured disk properly embedded in $\outside$.  
\end{definition}

\begin{definition}[\textbf{Neighborhood}]
A neighborhood of a transverse disk $D$ of 
a handlebody knot $\triplet$
is a closed $3$-ball $B$ in $\sphere$ 
containing $D$ in its interior 
such that $\mathfrak{N}_1=B\cap \outside$ 
is a tubular neighborhood of 
$D\setminus \bigcup_{i=1}^{n}\mathring{D}_{i}$ in $\outside$, 
and $B\cap \inside$ is the union of 
a tubular neighborhood 
$\mathfrak{N}_2$ of $\bigcup_{i=1}^{n}D_{i}$ and a 
tubular neighborhood $\mathfrak{N}_3$ of $\partial D$ 
in $\inside$ (see Fig.\ \ref{blow_up_construction1}). 
\end{definition}

\begin{center}
\begin{figure}[ht]
\includegraphics[scale=.09]{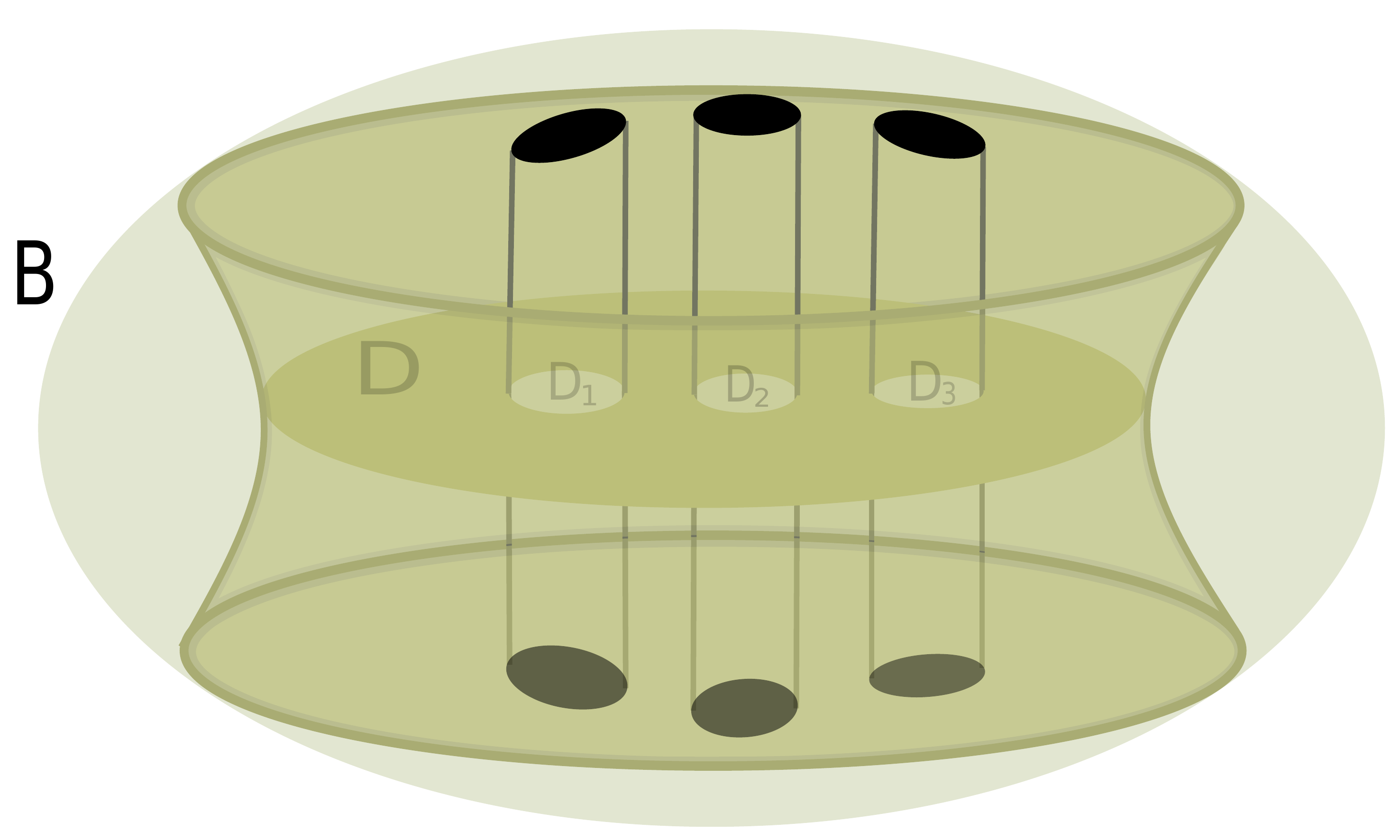}
\caption{A neighborhood $B$ of a transverse 
disk $D$ with $n=3$.}
\label{blow_up_construction1}
\end{figure}
\end{center}

\begin{construction}[The satellite construction.]
Identifying $B$ with the unit $3$-ball $B^{3}_u\subset \mathbb{R}^3$
via a homeomorphism  
\[
f:(B^3_u,D^2_\half\times I_,\bigcup_{i=1}^n D^2_i \times I)
\rightarrow (B,\mathfrak{N}_1\cup \mathfrak{N}_2,\mathfrak{N}_2),
\]
where $I=[-1,1]$, and $D^2_i$, $i=1,\cdots,n$,  
are disjoint disks in the interior of the disk 
$D_\half^2$ with radius $\frac{1}{2}$. 
By convention, we identify $D^2_\half\times I$ with 
the subspace  
\[C:=\{(x,y,z)\in B^3_u\mid x^2+y^2\leq \frac{1}{2}\}\subset B^3_u\]
via the homeomorphism
\begin{align*}
D^2_\half\times [-1,1]&\rightarrow C\\
(x,y,t)&\mapsto (x,y,t\sqrt{1-x^2-y^2}).
\end{align*}

Now, given an oriented knot 
$K:\Sbb^{1}\rightarrow \sphere$, we 
choose a basepoint $\ast \in \Sbb^{1}$ 
and observe that $K$ induces an embedding of arc
\[\Sbb^{1}\setminus \mathfrak{N}(\ast)\xrightarrow{K} \sphere\setminus \mathfrak{N}(K(\ast)),\]
where $\mathfrak{N}(\ast)$, $\mathfrak{N}(K(\ast))$ 
are open tubular neighborhoods of $\ast$ and $K(\ast)$ 
in $\Sbb^{1}$ and $\sphere$, respectively.
Identify $S^3\setminus \mathring{\mathfrak{N}}(K(\ast))$ with $B^3_u$ via 
a homeomorphism $h$, and 
$S^1\setminus \mathring{\mathfrak{N}}(\ast)$ with $I$ via a homeomorphism 
$l$ such that their composition with $K$ induce an embedding of the interval $I$: 
\[K_{h,l}:I\hookrightarrow B^3_u\]
going from the north pole of $B^3_u$ to the south pole (see Fig.\ \ref{blow_up_construction2}).   

Let $\mathfrak{N}(K_{h,l}(I))$ be a tubular 
neighborhood of the embedding arc $K_{h,l}(I)$ in $B_u^3$
such that $\mathfrak{N}(K_{h,l}(I))\cap \partial B^3_u=D^2_\half \times \partial I$. Then $K_{h,l}$ extends to an embedding
\[
\tilde{K}_{h,l}:D^2_\half \times I \rightarrow \mathfrak{N}(K_{h,l}(I))
\subset B^3_u
\] 
with $\tilde{K}_{h,l}\vert_{D^2_\half\times \partial I}$
being the inclusion, and for any given $(x,y)\in D^2_\half$, 
the arc $\tilde{K}_{h,l}((x,y),t), t\in I$, is parallel
\footnote{
namely, with zero linking number.
The linking number of the two strings is defined by
first gluing a copy of $B^3_u$ with opposite orientation 
to $B^3_u$ via the identity map
on $\partial B^3_u$, and complete the strings by vertical
lines through $(x,y)$ and $(0,0)$ in the copy of $B^3_u$.
} 
to $\tilde{K}_{h,l}((0,0),t):=K_{h,l}(t), t\in I$.

\begin{center}
\begin{figure}[ht]
\def\svgwidth{0.68\columnwidth}
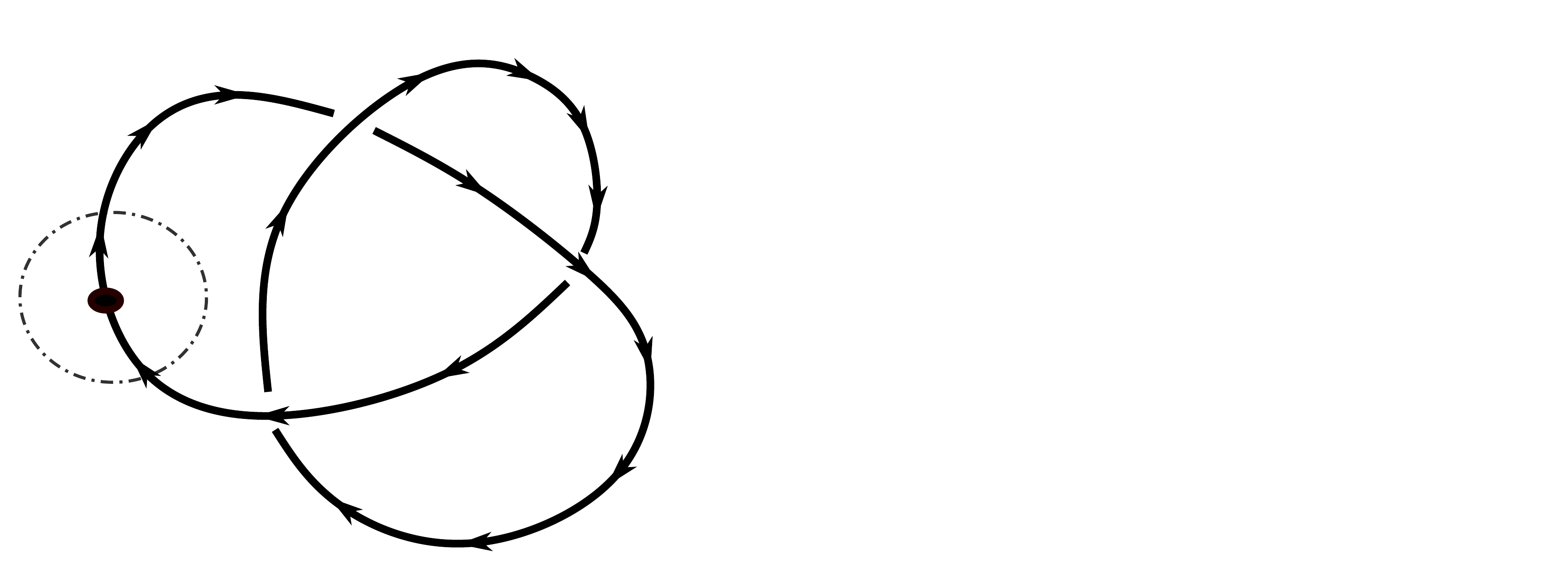
\caption{A proper arc in a ball.}
\label{blow_up_construction2}
\end{figure}
\end{center}

Consider the filtration given by 
\begin{equation}\label{BlowupatK} 
\bigcup_{i=1}^{n}D_{i}\times I\subset D^{2}_\half\times I 
\xhookrightarrow{\tilde{K}_{h,l}}B^{3}_u\xrightarrow{f}B.
\end{equation}
Then 
the new scene $\SSS^{D,K}=\tripletDK$ (Fig.\ \ref{blow_up_construction3})
obtained by performing the satellite construction
along $K$ w.r.t. $D$ on $\SSS$ is given by
\begin{align*}
\Sigma^{D,K}&:=f\circ \tilde{K}_{h,l}\big(
\partial D_\half^2\times I\cup_i \partial
D_i^2\times I
\big)
\bigcup 
\overline{\Sigma\setminus B}\\
\inside^{D,K}&:=f\Big(
\overline{B^3_u\setminus \tilde{K}_{h,l}(D_\half^2\times I)}
\cup_i 
\tilde{K}_{h,l}(D^2_i\times I)\Big) 
\bigcup \overline{\inside\setminus B}\\
\outside^{D,K}&:=f\circ \tilde{K}_{h,l} 
\big(
(D_\half^2\setminus \cup_i D_i^2)\times I
\big) 
\bigcup \overline{\outside\setminus B}.
\end{align*}

\begin{center}
\begin{figure}[ht]
\def\svgwidth{0.68\columnwidth}
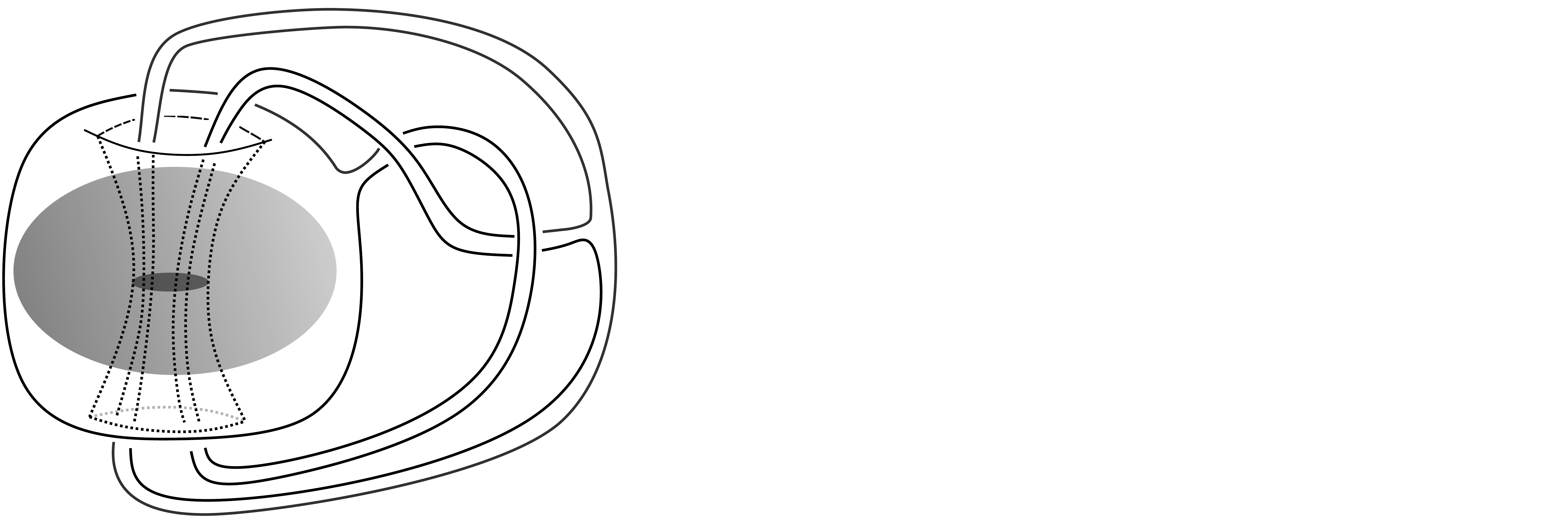
\caption{The new scene $\SSS^{D,K}$ (right) obtained by performing the satellite construction along $K$ w.r.t.\ $D$ on $\SSS$ (left).}
\label{blow_up_construction3}
\end{figure}
\end{center}
\end{construction}
%%%%%%%%%%%%%%%%%%%%%%%%%%%%%%%%%%%%%
%%%%%%%%%%%%%%%%%%%%%%%%%%%%%%%%%%%%%
%%%%%%%%%%%%%%%%%%%%%%%%%%%%%%%%%%%%%
 
To see that $\tripletDK$ is independent of all choices involved,
we first note that, in view of 
the tubular neighborhood theorem, it suffices
to consider another embedding
\[K':\Sbb^1\hookrightarrow \sphere\]
that has the same knot type as $K$, and another 
identifications
\begin{align*}
f':(B^3_u,D^2_\half\times I_,\bigcup_{i=1}^n D^2_i \times I)
&\rightarrow (B,\mathfrak{N}_1\cup \mathfrak{N}_2, \mathfrak{N}_2) \\
h':B^3_u  & \rightarrow \sphere\setminus \mathfrak{N}(K'(\ast))\\
l':I  & \rightarrow \Sbb^1\setminus  \mathfrak{N}(\ast). 
\end{align*}
As with $h,l,K$, the homeomorphisms
$h',l',K'$ together induce an embedding 
\[
\tilde{K}'_{h',l'}:
D^2_\half \times I \rightarrow \mathfrak{N}(K_{h,l}(I))
\subset B^3_u
\]
with $\tilde{K}'_{h',l'}\vert_{D^2_\half\times \partial I}$ being
the inclusion, and for any given $(x,y)\in D_\half^2$, 
the arc $\tilde{K}'_{h',l'}((x,y),t), t\in I$, is parallel
to $\tilde{K}'_{h',l'}((0,0),t):=K'_{h',l'}(t), t\in I$.
 
Since $K,K'$ have the same knot type,
there exists $g:B^3_u\rightarrow B^3_u$
with 
\[g\vert_{\mathfrak{N}(K_{h,l}(I))\cap\partial B^3_u}=\id\] 
such that the diagram commutes
\begin{equation}\label{diag:choices_in_K}
\begin{tikzpicture}
\node (D) at (0,1.5){$D^2_\half\times I$};
\node (Dprime) at (0,0) {$D^2_\half\times I$};
\node (nbhd) at (3,1.5) {$\mathfrak{N}(K_{h,l}(I))$};
\node (nbhdprime) at (3,0) {$\mathfrak{N}(K_{h',l'}'(I))$};
\node (Bu) at (4.5,1.5) {$B_u^3$};
\node (Buprime) at (4.5,0) {$B^3_u$};
\node   at (4.1,1.5) {$\subset$};
\node  at (4.1,0) {$\subset$};
\node (Commute) at (1.5,.8){$\circlearrowright$};
\draw[->] (D) to node[right]{$\vert$}(Dprime);
\draw[->] (nbhd) to node[right]{$g$}(nbhdprime);
\draw[->] (Bu) to node[right]{$g$} (Buprime);
\draw[->] (D) to node[above]{$\tilde{K}_{h,l}$}(nbhd);
\draw[->] (Dprime) to node[above]{$\tilde{K}'_{h',l'}$}(nbhdprime);
\end{tikzpicture}
\end{equation}    
In view of \eqref{diag:choices_in_K}, we may further assume 
that $g\vert_{\partial B^3_u}=\id$.  

Now observe that if there is an isotopy 
\[\Phi_t:
(B^3_u,D^2_\half\times I_,\cup_{i=1}^n D^2_i \times I)
\rightarrow 
(B,\mathfrak{N}_1\cup \mathfrak{N}_2, \mathfrak{N}_2) \]
between $f,f'$, then via the collar of $\partial B\subset \sphere$,   
the scenes obtained by 
$(f,h,l,K,\mathfrak{N}(K))$ 
and 
$(f',h',l',K',\mathfrak{N}(K'))$
are equivalent. Thus in general, it may be assumed that 
${f'}^{-1}f$ restricts to
the identity on 
$\partial B^3_u\setminus (\mathring{D}^2_\half\times I)\cup \partial D_\half^2\times I$.
Define $F$ to be 
\begin{equation}
F(x):= 
\begin{cases}
\tilde{K}_{h,l}{f'}^{-1}f\tilde{K}_{h,l}^{-1}(x), & x\in \mathfrak{N}(K_{h,l}(I))\\
g(x),& x\in \overline{B^3_u\setminus \mathfrak{N}(K_{h,l}(I))}.
\end{cases}
\end{equation}
$F$ is well-defined since $f^{'-1}f\vert_{\partial D^2_\half\times I}=\id$
and \eqref{diag:choices_in_K}.
Then we have a commutative diagram
\begin{equation} 
\begin{tikzpicture}
\node (D) at (0,1.5){$D^2_\half\times I$};
\node (Dprime) at (0,0) {$D^2_\half\times I$};
\node (nbhd) at (3,1.5) {$\mathfrak{N}(K_{h,l}(I))$};
\node (nbhdprime) at (3,0) {$\mathfrak{N}(K_{h',l'}'(I))$};
\node (Bu) at (4.5,1.5) {$B_u^3$};
\node (Buprime) at (4.5,0) {$B^3_u$};
\node (B) at (6,1.5) {$B$};
\node (Bprime) at (6,0) {$B$};

\node   at (4.1,1.5) {$\subset$};
\node  at (4.1,0) {$\subset$};
%\node (B) at (6,0)
\draw[->] (D) to node[left]{\tiny ${f'}^{-1}f$}(Dprime);
\draw[->] (nbhd) to node[left]{\tiny $\tilde{K}_{h',l'}{f'}^{-1}f\tilde{K}_{h,l}^{-1}$}(nbhdprime);
\draw[->] (D) to node[above]{\tiny $\tilde{K}_{h,l}$}(nbhd);
\draw[->] (Dprime) to node[below]{\tiny $\tilde{K}'_{h',l'}$}(nbhdprime);
\draw[->] (Bu) to node[right]{\tiny $F$} (Buprime);
\draw[->] (Bu) to node[above]{\tiny $f$}(B);
\draw[->] (Buprime) to node[below]{\tiny $f'$}(Bprime); 
\end{tikzpicture}
\end{equation}
Since $f^{'-1}Ff$ restricts to the identity on $\partial B$, 
together with the identity map on $\sphere\setminus \mathring{B}$,
it induces an equivalence between scenes obtained by 
$(f,h,l,K,\mathfrak{N}(K))$ 
and 
$(f',h',l',K',\mathfrak{N}(K'))$.
  
%%%isotopic f
%%%isotope ff' such that it is identity on the outside of the D2 times I
%%%another diagram (combining ff' and g)

We note that by the construction of $\tripletDK$, 
there is a natural homeomorphism 
$\iota:\outside\rightarrow \outside^{D,K}$ given by
\begin{align}\label{homeo_F_to_Fdk}
\iota_{\vert_{\outside\setminus B}}&=\id
\nonumber\\ 
{\iota_{K} }_{\vert_{\outside\cap B }}&:
\outside\cap B \xrightarrow{f^{-1}} (D \setminus \bigcup_{i=1}^{n}\mathring{D}_{i} )\times I 
\xrightarrow{f\circ\tilde{K}_{h,l}}
\outside^{D,K}\cap B .
\end{align}
On the other hand, the topology of $\inside^{D,K}$ might change.
Whether $\SSS^{D,K}$ and $\SSS$ are equivalent
depends solely on whether $\partial D\subset\inside$ is essential. 
%%%%%%%%%%%%%%%%%%%%%%%%%%%%%%%%%%%%%%%%%%%%%
%%%%%%%%%%%%%%%%%%%%%%%%%%%%%%%%%%%%%%%%%%%%%
%%%%%%%%%%%%%%%%%%%%%%%%%%%%%%%%%%%%%%%%%%%%%
%%%%%%%%%%%%%%%%%%%%%%%%%%%%%%%%have to change the name
%%%%%%%%%%%%%%%%%%%%%%%%%%%%%%%%%%%%%%%%%%%%%
\begin{lemma}\label{lem:explosive_equivalent}
Given a handlebody knot $\SSS=(\inside,\Sigma,\outside)$ with
a transverse disk $D$, 
if $\partial D$ bounds a disk in $\inside$, then $\SSS^{D,K}$
is equivalent to $\SSS$, for any knot $K$. 
\end{lemma}
\begin{proof}
Suppose $\partial D$ bounds a proper disk $D^{\prime}$ in $\inside$. 
Let $\mathfrak{N}(D^{\prime})$ 
be a tubular neighborhood of $D^{\prime}$ in $\inside$.
Then the closure complement of $\mathfrak{N}(D^{\prime})$ 
is a $3$-ball 
$\tilde{B}\subset \sphere$ containing $F$ with $D$ properly embedded in $\tilde{B}$. 
Take a tubular neighborhood $\mathfrak{N}(D)$ of $D$ in $\tilde{B}$ such that $\mathfrak{N}(D)\cap F$ is a tubular neighborhood of 
$D\setminus \bigcup_{i=1}^{n}\mathring{D}_{i}$ in $F$, and
denote by $B_1,B_2$ components of
$\tilde{B}\setminus \mathring{\mathfrak{N}}(D)$.
Note that, properly choosing
a tubular neighborhood $\mathfrak{N}(\partial D)$
of $\partial D$ in $\mathfrak{N}(D')$ gives us a neighborhood 
$B:=\mathfrak{N}(D)\cup\mathfrak{N}(\partial D)$ 
of $D$ with which we can perform 
the satellite construction \eqref{BlowupatK} (see Fig.\ \ref{Not_truly_knottable_disk}). 

Let $(f,h,l,\mathfrak{N}(K))$
be a choice of homeomorphisms and tubular neighborhood of an oriented knot $K$.
Then performing
the satellite construction along $K$ w.r.t.\
$D$
%, namely, reembedding $N(D)\simeq D^2\times \boundedclosedinterval$ 
%into $B\simeq B^3$ (Fig. \ref{blow_up_construction1}), 
amounts to reembedding $F$ via the composition 
\begin{equation}\label{Reembed_the_outside}
\outside\subset \tilde{B}=B_1\cup\mathfrak{N}(D)\cup B_2\xrightarrow{\mathbf{K}}\sphere, 
\end{equation}
where $\mathbf{K}$ is the identity on $B_1\cup B_2$, and is
the composition
\[\mathfrak{N}(D)\xrightarrow{f^{-1}}D^2_\half\times I \xrightarrow{\tilde{K}_{h,l}}
B^3_u\xrightarrow{f}\sphere\]
on $\mathfrak{N}(D)$ (Fig.\ \ref{BlowupatK}).
Since any two embeddings of the $3$-ball $\tilde{B}$ 
in $\sphere$ are ambient isotopic, 
$F^{D,K}$ is isotopic to $F$ in $\sphere$. 
\end{proof}

%Add a figure here
\begin{center}
\begin{figure}[ht]
\includegraphics[scale=.13]{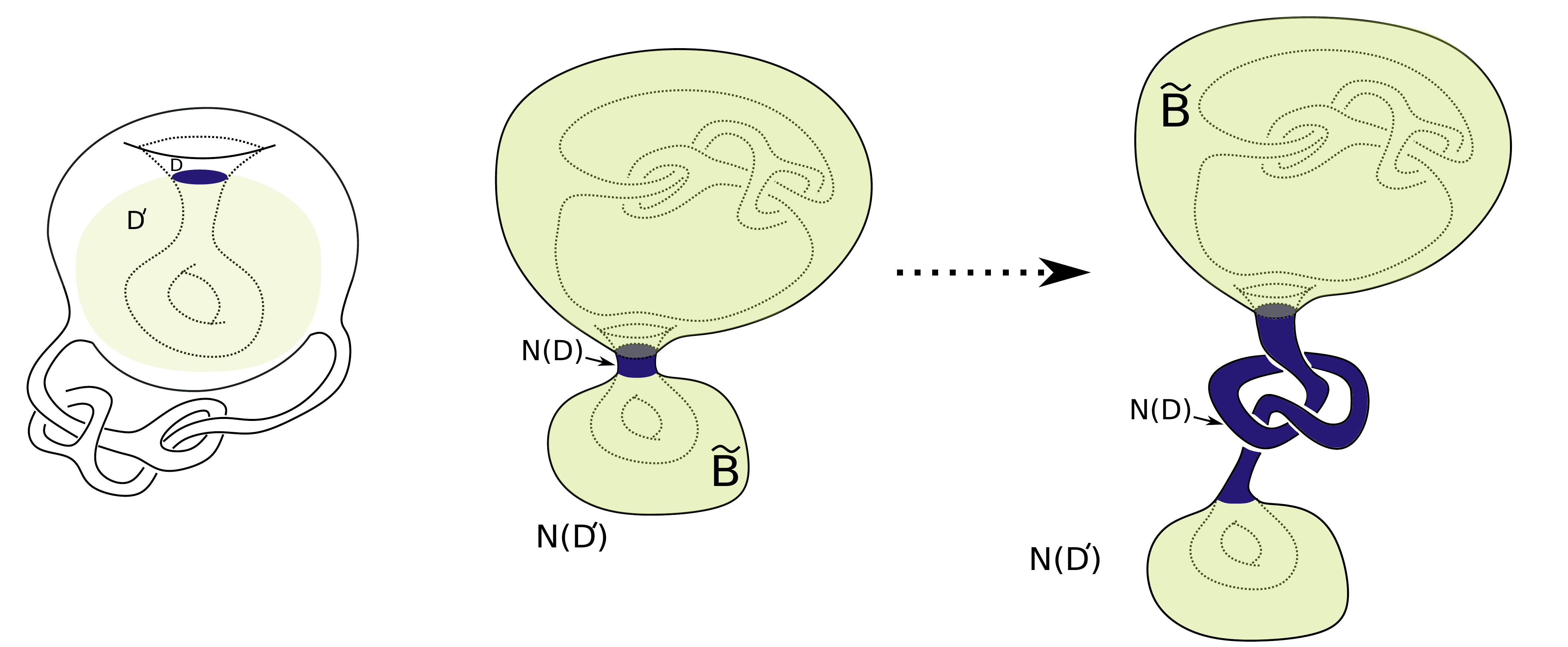}
\caption{Inessential transverse disk.} 
\label{Not_truly_knottable_disk} 
\end{figure}
\end{center}

Fig.\ \ref{Not_truly_knottable_disk} illustrates the situation 
in Lemma \ref{lem:explosive_equivalent}. 
Starting with a solid sphere with an unknotted solid torus 
removed from the inside,
then drilling a cylindrical hole connecting the north pole of the sphere with the northest point of the internal cavity and thus connecting 
the cavity with the outside, we get a solid set 
which is topologically equivalent to a solid torus. 
Attaching some knotted handles to the solids, we further obtain 
a handlebody knot in $\sphere$.
Now, take a cross-section of the
gallery entering the cave as our transverse disk $D$, 
whose boundary clearly bounds a disk $D^\prime$ in the solid. 
Then, performing the satellite construction along a knot 
w.r.t.\ the disk $D$ is equivalent to substituting the straight gallery by
a knotted one. However, it does not change the embedding,
since the knotted gallery can be easily untangled by isotoping
$D^\prime_{0}$, the bottom disk in $\mathfrak{N}(D^\prime)\simeq D^\prime\times \closedintervalzeroone$.

\begin{definition}[\textbf{Essential Transverse Disk}]\label{def:truly_explosive_disk}
A transverse disk $D$ of a handlebody knot 
$\SSS=(\inside,\Sigma,\outside)$ is 
essential if $\partial D$ bounds no disk in $\inside$. 
\end{definition}

%The following lemma justifies the definition.
\begin{lemma}
Let $D$ be an essential transverse disk of a handlebody knot
$\SSS=(\inside,\Sigma,\outside)$ and $K$ a non-trivial oriented knot. 
Then $\SSS^{D,K}$ and $\SSS$ are
inequivalent.
\end{lemma}
\begin{proof}
Let $B$ be a neighborhood of $D$.  
Then $\inside \cap \partial B$ is a collection of disks and an annulus $A$.
Note that performing the satellite construction along $K$ 
w.r.t.\ the disk $D$ does not change the boundary of $B$,
but $B\cap \inside^{D,K}$ is now a union of some tubes and 
the complement $\mathfrak{C}(K)$ of an open tubular neighborhood of $K$.
%, $B\setminus N(K)$; 
%here
%$N(K)$ can be identified with the image of $\mathring{D}^{2}\times %\boundedclosedinterval$ 
%by the map $\tilde{K}$ in \eqref{BlowupatK}.

Choose a base point $\ast\in A$, and apply van Kampen's Theorem
to the decomposition 
\[\inside^{D,K}= 
\mathfrak{C}(K) 
\cup \big(\inside^{D,K}\setminus  \mathring{\mathfrak{C}}(K) \big).
\]  
Then the following 
pushout diagram computes the fundamental group of $\inside^{D,K}$. 
\begin{center}
\begin{tikzpicture}
\node(Lm) at (0,0){$\pi_{1}(A,\ast)$};
\node(Rm) at (8,0){$\pi_{1}(\inside^{D.K},\ast)$};  
\node(Ml) at (4,-1) {$\pi_{1}(\inside^{D,K}\setminus \mathring{\mathfrak{C}}(K),\ast)$};
\node(Mu) at (4,1){$\pi_{1}(\mathfrak{C}(K),\ast)$};
\path[->, font=\scriptsize,>=angle 90] 

(Lm) edge node [above]{$\iota$}(Mu)  
(Lm) edge node [below]{$\vartheta$}(Ml) 
(Mu) edge (Rm)
(Ml) edge (Rm);
\end{tikzpicture}
\end{center}
where the homomorphisms $\iota,\vartheta$ 
are induced by the inclusions.  
 
The assumption that $\partial D$ does not bound a disk in $\inside$ 
implies the homomorphism $\vartheta$ is injective 
in view of the identification 
\[
\inside^{D,K}\setminus \mathring{\mathfrak{C}}(K)\simeq 
\inside\setminus \mathring{\mathfrak{N}}(D).\] 
On the other hand, $\iota$ sends the generator to 
the meridian of $K$, and hence is also injective.

The injectivity of $\vartheta$ and $\iota$ implies 
the other two homomorphisms are injective. 
In particular, $\pi_{1}(\inside^{D,K},\ast)$ contains 
a non-free group $\pi_{1}(\mathfrak{C}(K),\ast)$ 
since $K$ is non-trivial. 
In view of the Nielsen--Schreier theorem \cite{Sch:27}, 
$\inside^{D,K}$ cannot be a handlebody
and therefore not diffeomorphic to $\inside$.  

\end{proof}

From now on, we restrict our attention to  
handledbody knots of genus $2$.
%and we shall show that, if $\SSS$ is prime, 
%the boundary of a knottable disk in $\SSS$
%cannot separate $\Sigma$ 
%(Type \textit{II} in \ref{Type_of_knottable_disks}). 
%This motivates the following definition.

\begin{definition}[\textbf{Type I and II}]\label{Type_of_knottable_disks}
Let $\SSS=(\inside,\Sigma,\outside)$ 
be a handlebody knot of genus $2$. 
An essential transverse disk $D$ of $\SSS$
is said to be of type \textit{I} 
if $\partial D$ separates $\Sigma$, and of type {\it II} otherwise. 
%if $\partial D$ does not separate $\Sigma$.
\end{definition}

\begin{lemma}[\textbf{Primeness}]
Let $\SSS=(\inside,\Sigma,\outside)$ 
be a handlebody knot of genus $2$ admitting 
an essential \knottable disk $D$ of type $\textit{I}$. 
Then $\SSS$ is not prime, and furthermore if 
$\SSS=\SSS_{1}\sharp \SSS_{2}$
is the prime decomposition of $\SSS$, 
then either $\SSS_1$ or $\SSS_2$ is trivial,
where $\SSS_i$ is a connected scene of genus one, for $i=1,2$.  
\end{lemma}
 
\nada{
In other words, 
$\SSS$ is reducible (in the sense of \cite{IshKisMorSuz:12})
and has at least one trivial factor. }

\begin{proof}
Observe first that, since 
$\partial D$ separates $\Sigma$ and does not bound any disks in $\inside$, 
given a system of meridian disks  
of $\inside$, namely two disjoint disks $M_1,M_2$
such that $\inside \setminus \mathring{\mathfrak{N}}(M_1\cup M_2)$
is a $3$-ball (compare with \cite[p.864]{Och:91}, \cite[$2.16$]{Suz:75}),
then $\partial D$ intersects both $\partial M_1$ and $\partial M_2$.

Now let $\mathring{D}\cap \outside$ be the union 
of disjoint disks $D_1,\cdots, D_n$. If $D_{i}$ is essential in $\inside$, 
it induces a system of meridian disks of $\inside$, and
hence $\partial D\cap \partial D_i\neq \emptyset$, contradicting 
the fact that 
$D$ is an embedding in $\sphere$. On the other hand,
if $D_i$ is inessential, we can isotope $D$ away from $D_i$
with other components in $\inside \cap \mathring{D}$ intact.
Therefore, it may be assumed that 
$\mathring{D}\cap \outside=\emptyset$. Thus $D$
is properly embedded in $\outside$, and therefore 
$\SSS$ is not prime
by \cite[Prop. $2.15$]{Suz:75} and \cite[Theorem $1$]{Tsu:75}.

Let $\SSS=\SSS_{1}\sharp \SSS_{2}$ be the prime decomposition of $\SSS$, and 
suppose both $\SSS_{1}$ and $\SSS_{2}$ are non-trivial. 
Denote by $\outside_{1}$ and $\outside_{2}$ the outsides of $\SSS_1,\SSS_2$,
respectively---that is, the closures of the complement 
of the two non-trivial knots, and by 
$D_{s}$ the intersection of $\outside$ and 
a separating sphere $S$ of the decomposition $\SSS_{1}\sharp\SSS_{2}$.
Then the kernel of the induced homomorphism 
\begin{equation}\label{Lemma_separating_D}
\pi_{1}(\Sigma,\ast)\rightarrow \pi_{1}(\outside,\ast)\qquad \ast\in\partial D_s\subset\Sigma
\end{equation}
is the normal closure of the homotopy class $[\partial D_{s}]\in\pi_1(\Sigma,\ast)$.  
Since $D$ is properly embedded in $\outside$, 
$[\partial D]$ is also in the kernel of  
\eqref{Lemma_separating_D} and hence is 
a product of conjugates of $[\partial D_{s}]$.

On the other hand, $[\partial D_{s}]$ is also in the kernel of 
\begin{equation}\label{Kernel_Sigam_to_inside}
\pi_{1}(\Sigma,\ast)\rightarrow \pi_{1}(\inside,\ast),
\end{equation}  
and hence $[\partial D]$ is in the kernel of \eqref{Kernel_Sigam_to_inside} as well; this, however, contradicts the essentiality of $D$ as a transverse disk.
Thus one of $\SSS_1, \SSS_2$ has to be trivial.
\end{proof}

%%change the figure
The figure below is an essential transverse disk 
of \textit{type I} in a trivial scene.
Performing the satellite construction along a trefoil knot 
w.r.t.\ the disk produces the complement of the handlebody knot 
$\operatorname{HK}5_4$ ($5_4$ in Ishii et al.'s handlebody knot table).
\begin{center}
\begin{figure}[ht]
\includegraphics[scale=.13]{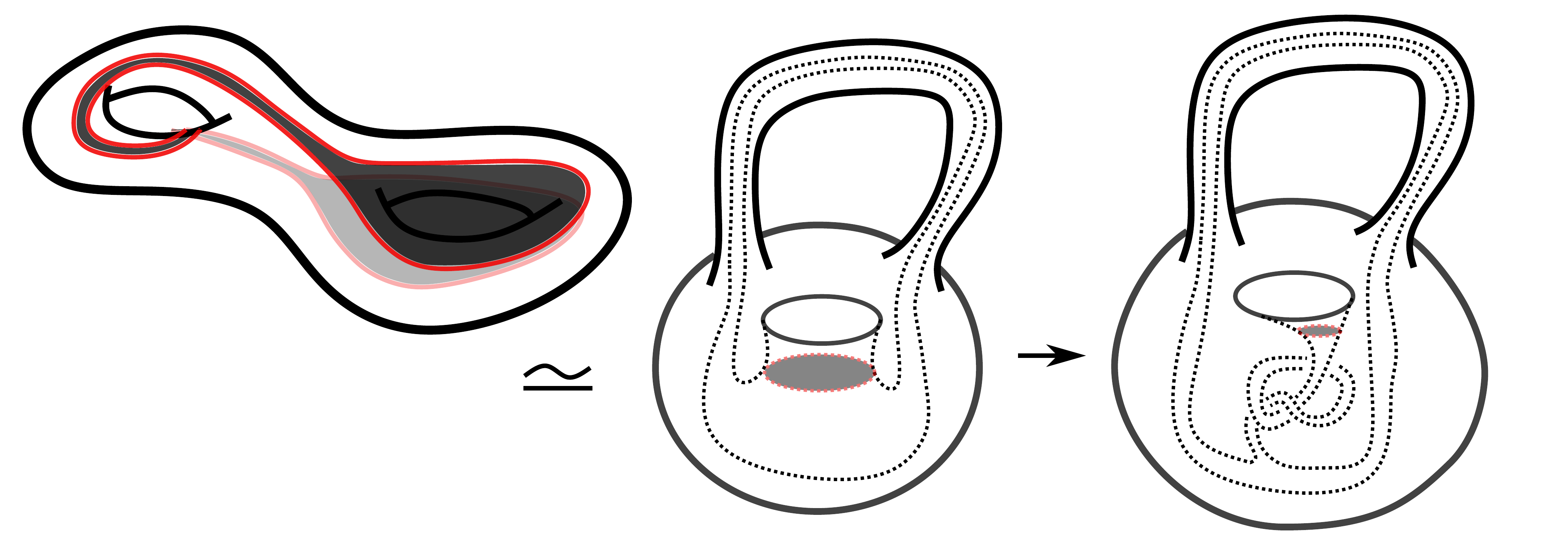}
\caption{Essential transverse disk of \textit{type I}.}  
\end{figure}
\end{center}

\begin{lemma}[\textbf{Meridian with respect to a \knottable disk}]\label{Meridians_w_r_t_explo_disk1}
Given a prime handlebody knot $\SSS=(\inside,\Sigma,\outside)$ of
genus $2$ with an essential \knottable disk $D$.
Then there exists a system of meridian disks $\{M_{1},M_{2}\}$ of $E$ 
such that $\partial D$ intersects $M_1$
at a single point and $M_2 \cap \partial D=\emptyset$.
%\draftGGG{is here $m_2$?}. 
Furthermore, if $\{M_1^{\prime},M_2^{\prime}\}$ is another 
system of meridian disks with $M_1' \cap \partial D$ a single point and
$M_2' \cap D =\emptyset$,
then $M_2,M_2'$ are isotopic in $E$.
\end{lemma}
\begin{proof}
Since $(\inside,\Sigma,\outside)$ is prime,
$\mathring{D}\cap \inside\neq \emptyset$. 
If one of the disks, say $D_1$ in $\mathring{D}\cap \inside=\bigcup_{i=1}^n D_i$ is a meridian disk, then we define $M_2$ to be $D_{1}$.
The complement $\tilde{\inside}$ of 
an open tubular neighborhood of $\bigcup_{i=1}^{n}D_{i}$ in $\inside$ contains exactly one component homeomorphic to a solid torus with 
$\partial D$ on the boundary of the solid torus since  
$\partial D$ does not bound any disk in $\inside$. 
In particular, the solid torus is unknotted in $\sphere$.
We define $M_1$ to be a meridian disk of the solid torus.

On the other hand, if none of $D_i's$ is a meridian disk of $\inside$, namely $D_i's$ all separating $\inside$, then 
\[\tilde{\inside}:=E\setminus \mathring{\mathfrak{N}}(\bigcup_{i=1}^n D_i)\] 
contains two solid tori. One 
of them contains $\partial D$, while the other does not intersect $D$. 
We define $M_1$ to be the meridian disk of 
the former, and $M_2$ the meridian disk of the latter.

Suppose there is another system of meridian disks
$\{M_1^{\prime},M_2^{\prime}\}$ 
satisfying $M_1^{\prime}\cap   D$ is a point and
$M_2'\cap   D=\emptyset$. Then we attach a $2$-handle $h^2$
to $\inside$ along a tubular neighborhood of $\partial D$ in $\partial \inside$. The resulting manifold $\hat{\inside}:=\inside\cup h^2$ is a solid torus with both $M_2,M_2'$ meridian disks of $\hat{\inside}$.

If $\partial M_2\cap \partial M_2'\neq \emptyset$, 
then there are at least two arcs of $\partial M_2\cap \partial M_2'$ 
innermost in $M_2$, 
at least one of which cuts off disks $Q,Q'$ from $M_2,M_2'$, 
respectively, such that $Q'\cap M_2=M_2'\cap M_2$,  
$D'':=Q \cup Q'$ inessential in $\hat{\inside}$,
and the $3$-ball component of $\hat{\inside}\setminus \mathring{\mathfrak{N}}(D'')$
not containing $h^2$. Therefore, it may be assumed that 
$M_2,M_2'$ are disjoint, and the $3$-ball component 
in $\hat{\inside}\setminus \mathfrak{N}(M_2\cup M_2')$
not containing $h^2$ gives an isotopy 
between $M_2$ and $M_2'$ in $\inside$.  
\end{proof}
 
\begin{definition}[\textbf{Associated meridian}]
We call the boundary $m$ of $M_2$ in Lemma \ref{Meridians_w_r_t_explo_disk1}  
the associated meridian w.r.t.\ the essential 
\knottable disk $D$. 
\end{definition}

%In view of Lemma \ref{Meridians_w_r_t_explo_disk1}, 
Given an essential \knottable disk $D$ of 
a prime handlebody knot $(\inside,\Sigma,\outside)$,
we can identify $\inside$ with the complement of
a trivial scene of genus $2$, $H_2\subset\sphere$, 
and $\partial D$ with one of its meridians.

\begin{figure}[ht]
    \centering
    \subfloat{ {
    \def\svgwidth{0.42\columnwidth}
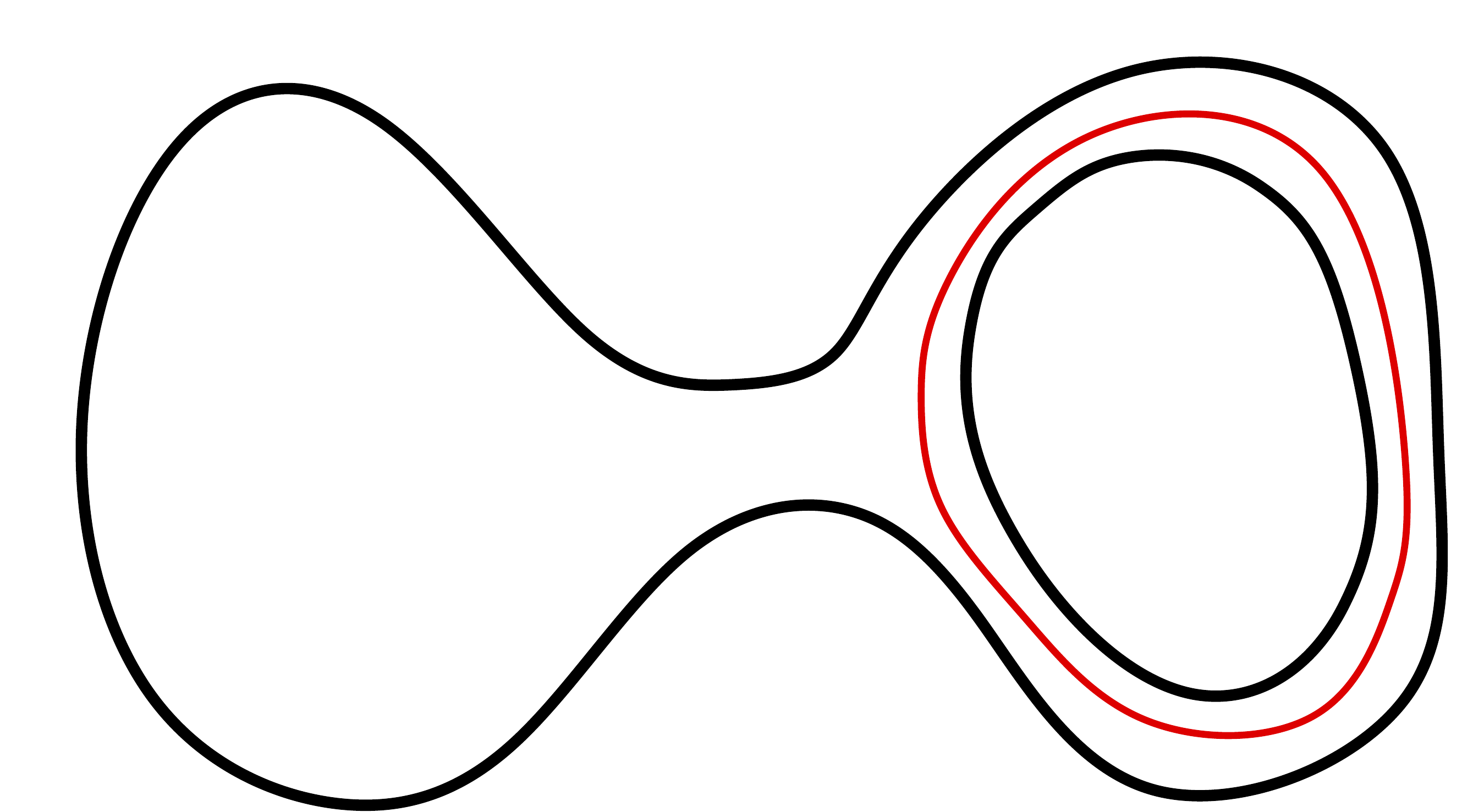
} }%
    \qquad
    \subfloat{ {
\def\svgwidth{0.42\columnwidth}
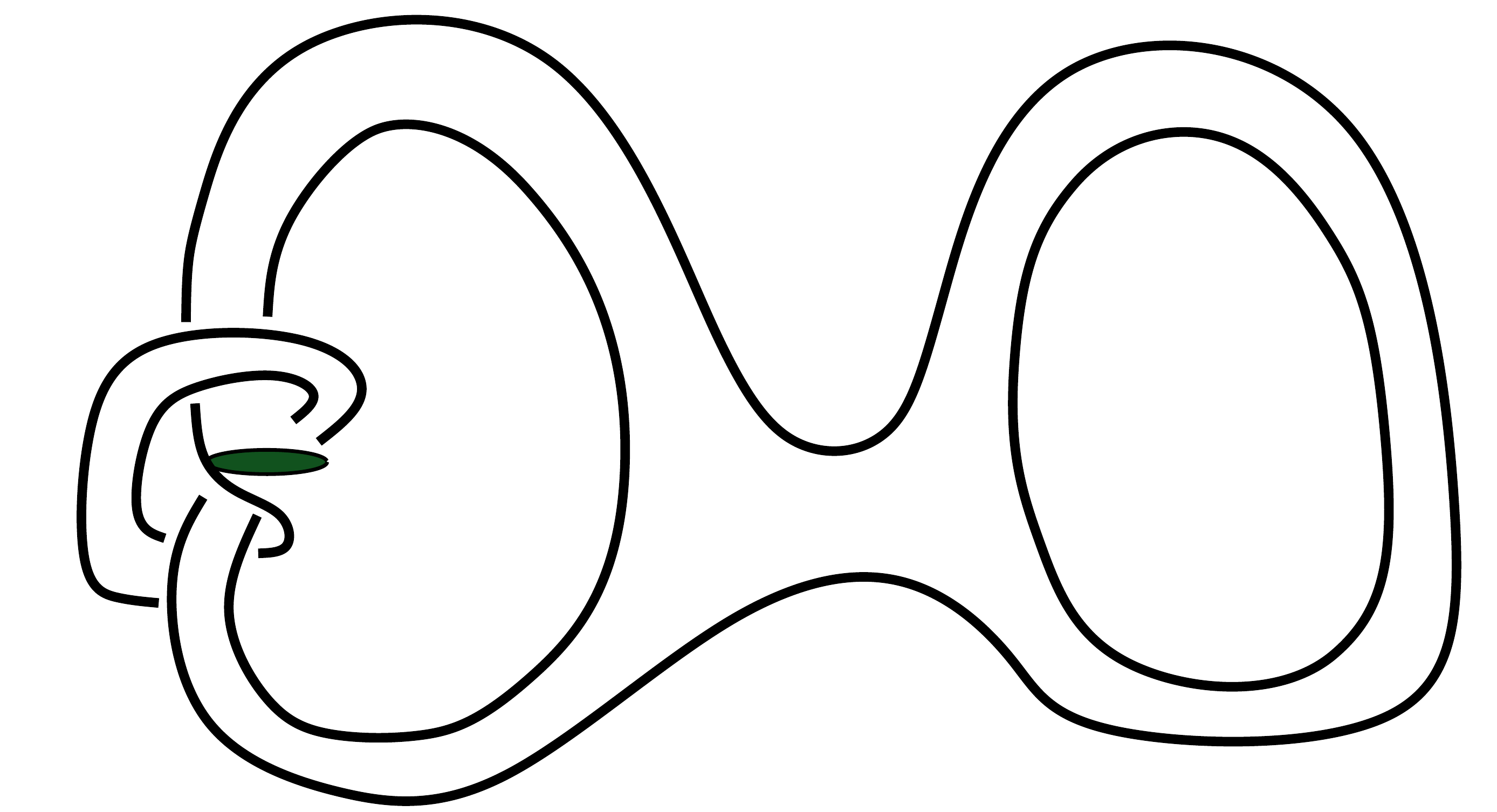    
    %{\includegraphics[scale=.13]{Realization_of_E2.pdf}
    } }%
    \caption{Realizing \inside \
    as the complement of a trivial scene.}
    \label{Realization_of_E}
\end{figure}

Then $\inside^{D,K}$ is homeomorphic to 
the outside of the scene obtained by performing the satellite construction
along $K$ w.r.t.\ a transverse disk bounded by $\partial D$ in $H_2$.

\begin{lemma}[\textbf{Meridians into meridians}]\label{Two_truely_explosive_disks}
Let 
$\SSS=(\inside,\Sigma,\outside)$,
$\SSS'=(\inside',\Sigma',\outside')$ 
be connected prime handlebody knots of genus two, and $D$,
$D^{\prime}$ essential \knottable disks in $\SSS$ 
and $\SSS'$, and 
$m$, $m^{\prime}$ the associated meridians w.r.t.\ $D, D'$, respectively.
Then any equivalence between $\SSS^{D,K}$ and $\SSS^{'D^{\prime},K}$ can be isotoped such that it sends $m$ to $m^{\prime}$.  
\end{lemma}
\begin{proof}
Observe first that $\inside^{D,K}\simeq \inside^{'D',K}$ 
have the prime decomposition 
\[\inside^{D,K}\simeq V\#  \mathfrak{C}(K)\simeq \inside^{'D',K},\]
where $\mathfrak{C}(K)$ is the complement of an open tubular neighborhood of 
$K$ (see Fig.\ \ref{Realization_of_E}),
and $m$ and $m^{\prime}$ can be identified with 
meridians of the solid torus $V$. 
By \cite[$3.4;3.6$]{Suz:75}, up to isotopy any diffeomorphism between 
$\inside^{D,K}$ and $\inside^{'D',K}$ sends $m$ to $m^{\prime}$.
On the other hand, any equivalence between $\SSS^{D,K}$ and $\SSS^{'D',K}$
induces a diffeomorphism between $\inside^{D,K}$ and $\inside^{'D',K}$, 
and hence can be isotoped such that $m$ is sent to $m^{\prime}$.  
\end{proof}

%\begin{lemma}\label{Meridians_w_r_t_explo_disk2}
%Given two truly explosive disks $D_{1}$ and $D_{2}$ in a handlebody scene $(E,\Sigma,F)$, 
%suppose $\partial D_{1}\cap \partial D_{2}=\emptyset$ and 
%there exists a meridian system $\{m_{1},m_{2}\}$ such that
%$\partial D_{i}\cap m_{j}=\emptyset$ if $i\neq j$, and
%$\partial D_{i}\cap m_{j}$ is a point if $i=j$. 
%Then there is a homeomorphism between $F$ and the complement of a trivial connected scene
%such that $\partial D_{i}$, $i=1,2$ are sent to the meridian of the trivial scene. 
%\end{lemma} 
%\begin{proof}
%Observe first that the resulting manifold obtained by gluing two disks to $F$ along 
%$\partial D_{i}$, $i=1,2$, is a $3-ball$. Then the trivial connected scene can be obtained 
%by gluing another $3$-ball to the $3$-ball along their boundaries. 
%\end{proof} 

%%%Lemma 6.4 and 6.5 perhaps are not needed for the present paper

\begin{example}[\textbf{Handcuff}] 
Consider a handcuff graph $\mathsf{G}$ (left in \eqref{eq:handcuff_transverse_disks}). Attach two $2$-cells (right in \eqref{eq:handcuff_transverse_disks}) to its two circles, respectively.
\begin{equation}\label{eq:handcuff_transverse_disks}
\begin{tikzpicture}
\draw [ultra thick] (3,0.5) circle [radius=0.5];
\draw [ultra thick] (1,0.5) circle [radius=0.5];
\draw [ultra thick] (1.5,0.5) -- (2.5,0.5);
\draw [fill=gray, ultra thick] (7,0.5) circle [radius=0.5];
\node  at (7,.5){$\mathsf{D}^\prime$};
\draw [fill=gray, ultra thick] (5,0.5) circle [radius=0.5];
\node  at (5,.5) {$\mathsf{D}$};
\draw [ultra thick] (5.5,0.5) -- (6.5,0.5);
\end{tikzpicture}
\end{equation}
The resulting space is a $2$-dimensional $CW$-complex $\mathsf{X}$.
Let $\iota:\mathsf{X}\rightarrow \sphere$ be a map 
such that its restrictions 
$\iota\vert_{\mathsf G}$, $\iota\vert_{\mathsf D}$, and 
$\iota\vert_{\mathsf{D}'}$ are embeddings, and 
intersections between the two $2$-cells $D:=\iota(\mathsf{D})$ and 
$D':=\iota(\mathsf{D}')$, and 
between each of them and $G:=\iota(\mathsf{G})$ are transversal.

A tubular neighborhood $\inside_G$ of $G$ induces 
a handlebody knot 
\[\SSS_{G}=(\inside_{G},\partial \inside_{G},\outside_{G}),\]
where $\outside_{G}$ 
is the closure of the complement of $\inside_G$ in $\sphere$, and 
$D,D'$ give 
two \knottable disks of $\SSS_{G}$.
It is not always possible to perform the satellite construction w.r.t.\ 
$D,D'$ at the same time in general
as $D\cap D'$ might not be empty. 

Denote by 
$\SSS_{G}^{K}=(\inside_{G}^{K},\partial \inside_{G}^{K}, \outside_{G}^{K})$ and 
$\SSS_{G}^{\prime,K}=(\inside_{G}^{\prime,K},\partial \inside_{G}^{\prime,K}, \outside_{G}^{\prime,K})$
the connected scenes obtained by performing the satellite construction 
along an oriented knot $K$ w.r.t.\ $D$ and $D^{\prime}$, respectively. Then the $3$-manifolds $\inside^{K}_{G}$ and  $\inside^{\prime,K}_G$ (resp.\ 
$\outside^{K}_{G}$ and $\outside^{\prime,K}_{G}$) are homeomorphic, 
and hence the connected scenes 
$\SSS_{G}^{K}$ and $\SSS_{G}^{\prime,K}$
are bi-knotted scenes with homeomorphic components. 
Whether or not this pair are equivalent depends on
the symmetry of the spatial graph $G\subset\sphere$.
\end{example}

For instance, for any of the handcuff diagrams $\operatorname{HK}4_{1}$, $\operatorname{HK}5_{4}$, 
$\operatorname{HK}6_{5}$, $\operatorname{HK}6_{7}$, $\operatorname{HK}6_{10}$ 
and $\operatorname{HK}6_{16}$ in \cite{IshKisMorSuz:12}, there is an 
obvious isotopy sending the graph to itself with two circles exchanged. 
Hence, performing the satellite construction w.r.t.\ $D$ and $D^{\prime}$ results in equivalent scenes.
On the other hand, by Theorem \ref{Two_truely_explosive_disks}, we have a sufficient condition for $\SSS_{G}^{K},\SSS_{G}^{\prime,K}$ to be inequivalent. 

%\draftGGG{there should be a latex error, corollary and theorem should be in bold}

\begin{corollary}[\textbf{Inequivalence criterion}]\label{cor:inequivalent}
Suppose $\SSS_{G}$ is prime, and  
$m$ and $m^{\prime}$ be the associated meridians with respect to $D$ and $D^{\prime}$, respectively.  
If there is no self-homeomorphism of $\outside_{G}$ sending $m$ to $m^{\prime}$,
then $\SSS_{G}^{K},\SSS_{G}^{\prime,K}$ are inequivalent, for any non-trivial $K$.
\end{corollary} 

\begin{proof}
Via the homeoomorphism in \eqref{homeo_F_to_Fdk}, there are identifications 
\[\outside_{G}^{K}\simeq \outside_{G},\quad \outside_{G}^{\prime,K}\simeq \outside_{G};\]
the first homeomorphism preserves $m$, while the second preserves $m^{\prime}$.
Thus, by Lemma \ref{Two_truely_explosive_disks}, 
any equivalence between 
$\SSS_{G}^{K}$ and $\SSS_{G}^{\prime,K}$ 
can be isotoped such that the induced 
self-homeomorphism of $\outside_{G}$ sending $m$ to $m^{\prime}$. 
\end{proof}

Corollary \ref{cor:inequivalent} 
and the invariant derived from the fundamental span in Section \ref{sec:representations} imply the following: 
\begin{theorem}[\textbf{Inequivalence of Ishii et al's handcuffs}]\label{Asymmetric_handcuff_graph}
The connected scenes obtained by performing the satellite construction w.r.t.\ 
$D$ and $D^{\prime}$ along a non-trivial $K$ on
any of the handcuff graph diagrams $\operatorname{HK}5_{1}$, $\operatorname{HK}6_{1}$, 
and $\operatorname{HK}6_{11}$ in \cite{IshKisMorSuz:12} are inequivalent.
\end{theorem}

In particular, taking one of handcuff graph diagrams in Theorem \ref{Asymmetric_handcuff_graph}, 
and performing the satellite construction along inequivalent oriented knots 
w.r.t.\ $D;D'$, we get a an infinite family of pairs 
of inequivalent connected scenes with homeomorphic components.

The only handcuff diagram in Ishii et al.'s handlebody knot table 
not mentioned yet is $\operatorname{HK}6_{2}$.
There is a less obvious diffeomorphism 
sending $m$ to $m'$. The moves in Fig.\ \ref{fig:symmetry_hk6_2} 
induces a diffeomorphism swapping two circles in the handcuff graph diagram 
$\operatorname{HK}6_2$ of Ishii's handlebody knot table.

\begin{figure}[ht]
    \centering
    \subfloat{ {\includegraphics[scale=.075]{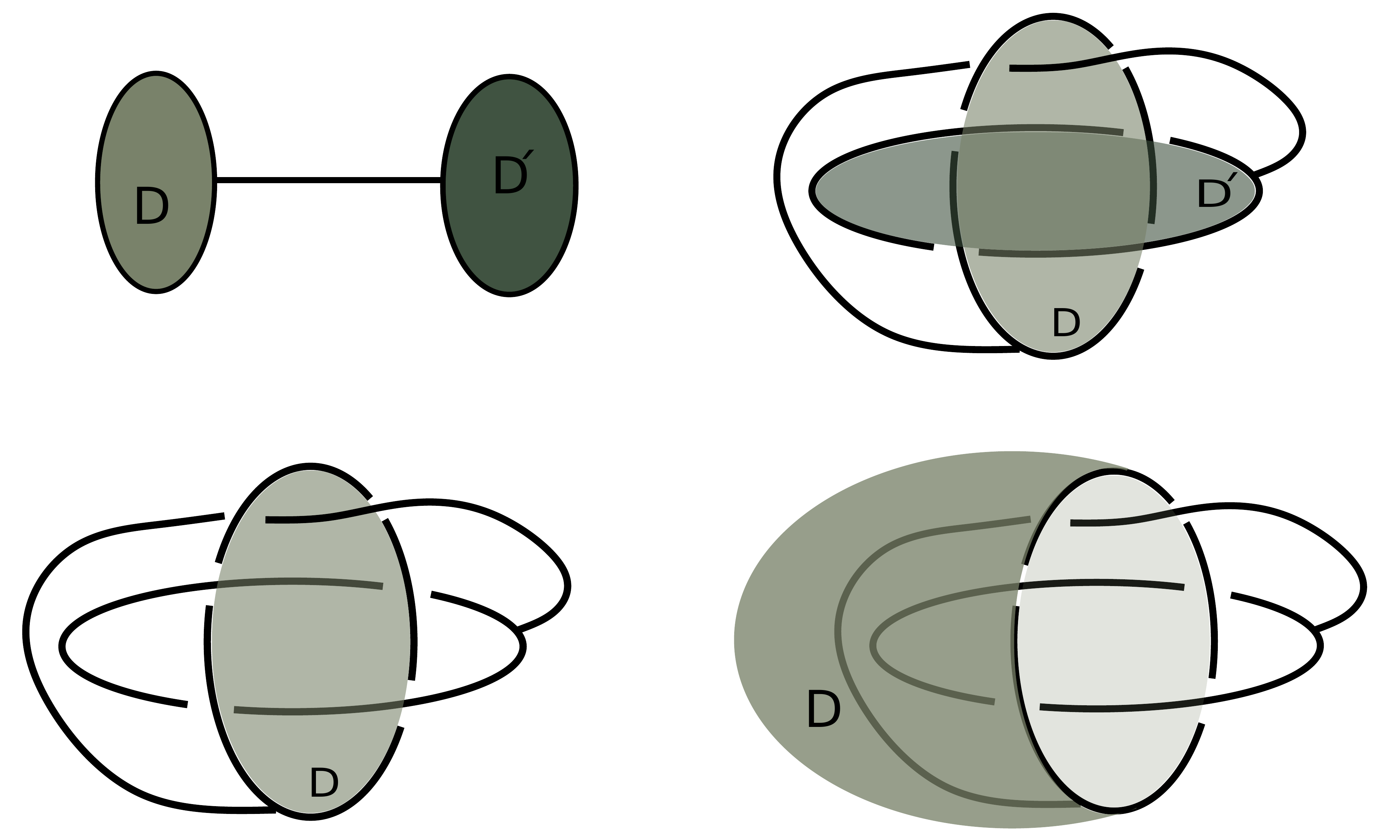} } }%
    \qquad
    \subfloat{ {\includegraphics[scale=.072]{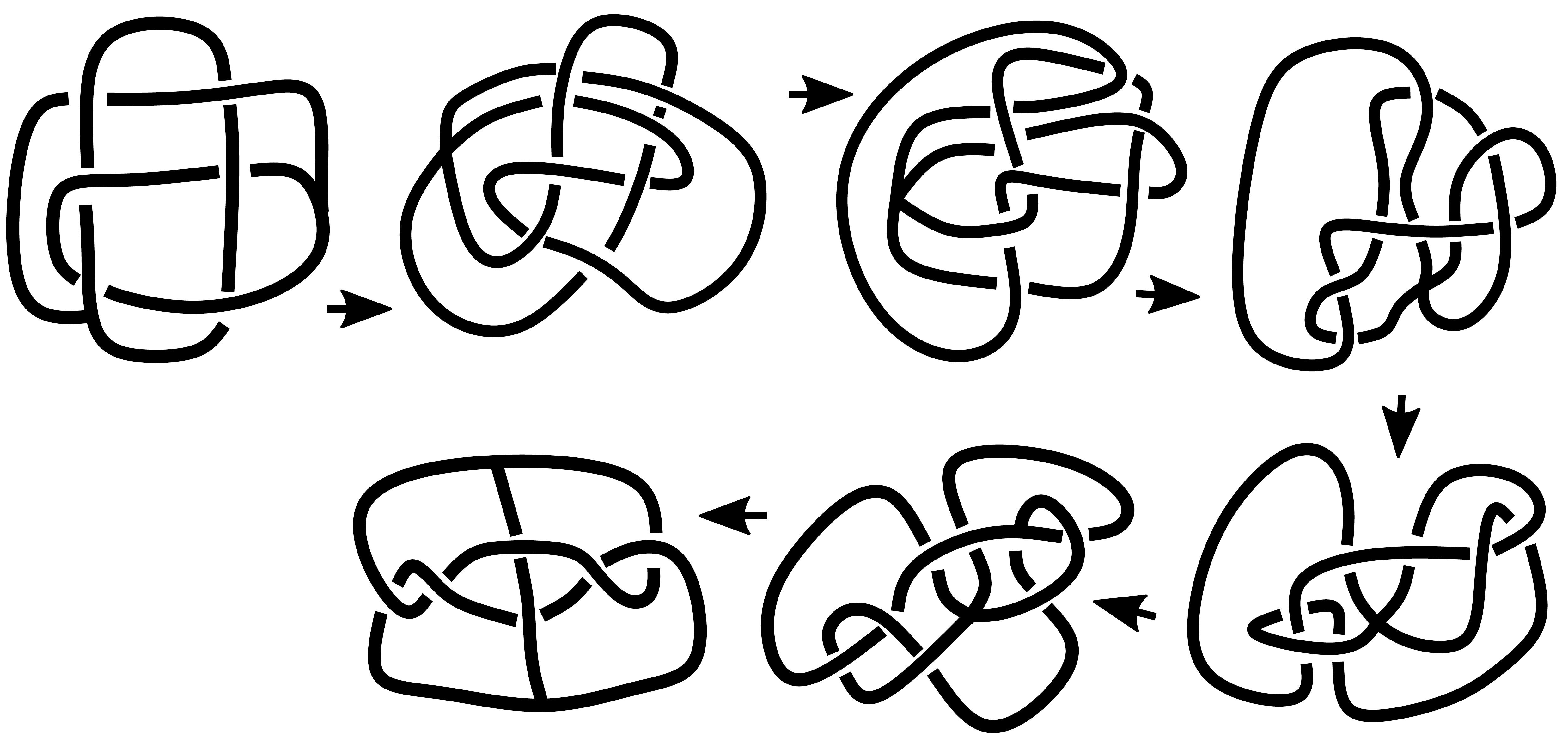}} }%
    \caption{Symmetry in $\operatorname{HK}6_2$.} 
    \label{fig:symmetry_hk6_2}
\end{figure}
%\noindent 
%One can keep track of $D$ and $D^{\prime}$ in the moves, 
%and they are isotopic to the disks bounded by 
%the upper and lower circles in the last (symmetric) handcuff diagram and 
%hence exchangeable. Notice that it can actually be deformed into an isotopy of graphs.

%%Example $2$: irreducible knots
 
%%Example $1$: separating disks (next paper)

%
%https://math.stackexchange.com/questions/2511364/how-did-dehn-prove-that-the-trefoil-is-chiral
%
 
\subsection{Chirality}
Chirality of a connected scene concerns the relation between a connected scene
and its mirror image; the next definition generalizes the notion of chiral knots.

\begin{definition}[\textbf{Mirror image}]\label{def:mirror_image}
Given a connected scene $\SSS=(\inside,\Sigma,\outside)$, 
its mirror image $m\SSS=\{m\inside,m\Sigma,m\outside\}$ 
is the connected scene defined as follows:
$m\inside$ is the image of $\inside$ in $\sphere$ under an orientation-reversing 
self-diffeomorphism of $\sphere$, the orientation of $m\inside$ is 
induced from $\sphere$, $m\Sigma$ is the boundary of $m\inside$, and 
$m\outside:=\overline{\sphere\setminus m\inside}$.

\noindent
A connected scene $\SSS$ is chiral if $\SSS$ and $m\SSS$
are inequivalent connected scenes; otherwise $\SSS$ is
an amphichiral scene.
\end{definition}
In the present paper, we shall restrict our focus on 
the special case of chiral knots and study the chirality
of $9_{42}$ and $10_{71}$ in Rolfsen's knot table; they are
denoted by $\operatorname{K}9_{42}$ and $\operatorname{K}10_{71}$
here to avoid confusion with the notation in Ishii 
et al.'s handlebody knot table (see Fig. \ref{fig:K9_42_and_K10_71}).

\begin{center}
\begin{figure}[ht]
\includegraphics[scale=.15]{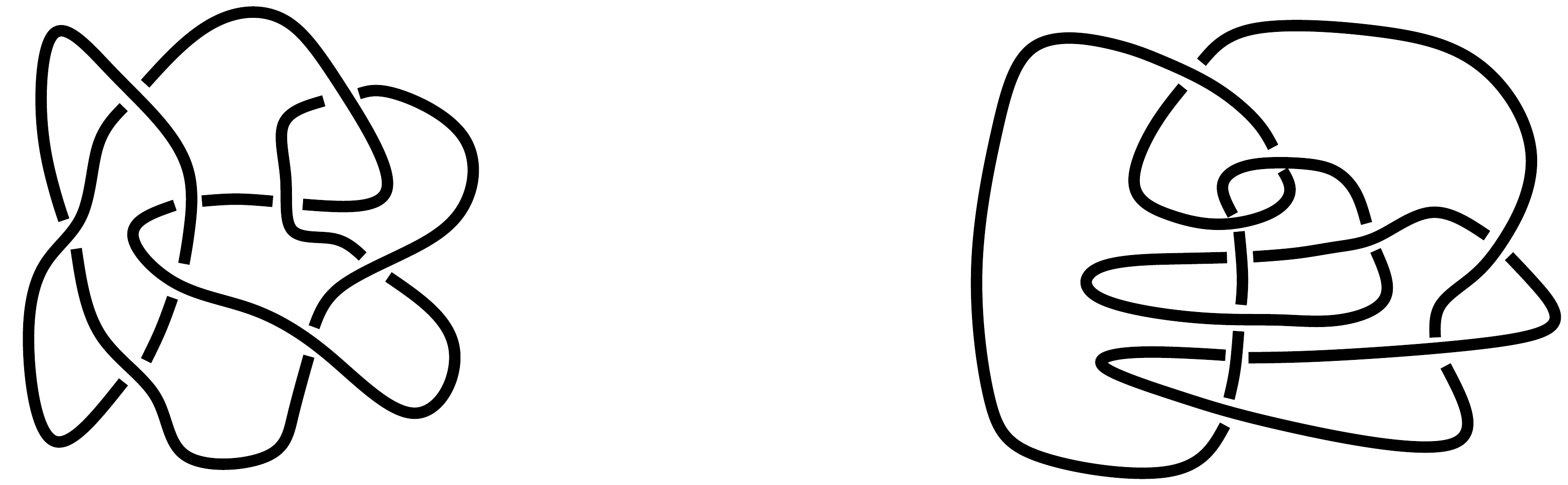}
\caption{$\operatorname{K}9_{42}$ and $\operatorname{K}10_{71}$.} 
\label{fig:K9_42_and_K10_71}
\end{figure}
\end{center}

Their chirality cannot be discerned by knot polynomials, such as   
the Jones polynomial, HOMFLY-PT polynomial and Kauffman 
polynomial. In the next section, we present a simple invariant, 
a reinterpretation of Fox's argument in \cite{Fox:50} in terms
of the group span with pairing, that can detect their chirality.

%[Here we can list a number of examples of scenes that are not distinguished by their
%fundamental groups alone:

%- sphere (genus 0)
%- torus (knots)
%- handlebodies of genus 2
%- biknotted surface

%- at least one example where the scenes are distinguished only by reflection,
%e.g. left/right-handed trefoil knot]

%%%%%%%%%%%%%%%%%%%%%%%%%%%%%%%%%%%%%%%%%%%%%%%%%%%%%%%%%%%%%%%%%%%%%%%%%%%%%

%%%%%%%%%%%%%%%%%%%%%%%%%%%%%%%%%%%%%%%%%%%%%%%%%%%%%%%%%%%%%%%%%%%%%%%%%%%%%
\section{Invariants of an algebraic scene}\label{sec:representations}
%%%Starting from here later... Check baspoint! 
In Section \ref{sec:examples} we use a generalized Motto-Lee-Lee construction
and the satellite construction to produce many inequivalent connected scenes 
with homeomorphic complements. 
The aim of the section is to devise tools to investigate these examples.
The invariants defined here make 
crucial use of homomorphisms from $\pi_1(\outside,\ast)$
to a finite group $G$ and various subgroups of $\pi_1(\outside,\ast)$ 
induced from the fundamental span. 
 
The invariants presented in this section are computable, and the major part of the computation are carried out by the program Appcontour developed by the second 
author \cite{BeBePaPa:15}, \cite{appcontour}. 
The result of our computation is recorded in Section \ref{sec:using}.   

Given a connected scene $\SSS=(\inside,\Sigma,\outside)$ and a base point $\ast\in\Sigma$, 
we consider the set $\operatorname{Hom}(\pi_1(\outside,\ast),G)$ 
of homomorphisms from $\pi_1(\outside,\ast)$ 
to a finite group $G$.
It is clear that this set is independent of the choice of 
a base point---namely, given two base points $\ast$ and $\ast'$ on $\Sigma$,
there is a bijection between  
$\operatorname{Hom}(\pi_1(\outside,\ast),G)$ and $\operatorname{Hom}(\pi_1(\outside,\ast'),G)$.
Furthermore, there is a left action of
$\operatorname{Aut}(G)$ on $\operatorname{Hom}(\pi_1(\outside,\ast),G)$  
given by the composition, where $\operatorname{Aut}(G)$ is the
automorphisms group of $G$. For the sake of simplicity, we denote 
the set of all orbits of 
$\operatorname{Hom}(\pi_1(\outside,\ast),G)$
under the action of $\operatorname{Aut}(G)$ by 
$$
\mathcal{H}(\outside)_G.
$$
In some situations it is more convenient to consider other subgroups 
of $\operatorname{Aut(G)}$, for instance the inner automorphisms
of $G$; the orbit set of $\operatorname{Hom}(\pi_1(\outside,\ast),G)$ 
under the action of the subgroup is also an invariant of $\SSS$.

The cardinality of the orbit set $\mathcal{H}(\outside)_G$ is 
a strong invariant of connected scenes. For example, 
Ishii, Kishimoto, Moriuchi and Suzuki \cite{IshKisMorSuz:12} show 
that most of the handlebody knots up to six crossings
can be distinguished by the number of conjugacy classes 
of $\operatorname{SL}(2,\mathbb{Z}/p\mathbb{Z})$- and 
$\operatorname{SL}(3,\mathbb{Z}/p\mathbb{Z})$-representations
of $\pi_1(\outside,\ast)$.

However, since $\mathcal{H}(\outside)_G$ and its variants depend only on 
the homeomorphism type of $\outside$, they cannot distinguish the examples in Section 
\ref{sec:examples}. A finer invariant taking into account
the interrelation between $\inside$, $\Sigma$, and $\outside$ is required to examine these
examples.

%\draftGGG{ in my opinion  the next sentence should be also
%anticipated, maybe in the introduction, or
%in section 4}  

%It is known that there is no algorithm to decide whether two group
%presentations represent isomorphic groups \cite{Dehn:11}, and so in practice, 
%the fundamental scene is not computable in general. 
% \draftYYY{Is it a consequence of the famous word problem?}
% \draftMMM{It's the group isomorphism problem \url{https://en.wikipedia.org/wiki/Group_isomorphism_problem}, reference added}
%However, one can, by studying group representations of a fundamental group, 
%try to understand geometric contents hidden in a fundamental group. 

%For instance, Kitano and Suzuki 
%\cite{KitSuz:12}
%showed
%that all knots in Rolfsen's knot table can be distinguished 
%by counting the number of conjugacy classes of $\operatorname{SL}(2,\mathbb{Z}/p\mathbb{Z})$-representations 
%of a knot group;
%Similarly, Ishii, Kishmoto, Moriuchi and Suzuki 
%\cite{} 
%proved that most 
%of the handlebody knots up to six crossings
%can be distinguished by the number of conjugacy classes 
%of $\operatorname{SL}(2,\mathbb{Z}/p\mathbb{Z})$ and 
%$\operatorname{SL}(3,\mathbb{Z}/p\mathbb{Z})$-representations of their fundamental groups.  
%However, the number of representations up to conjugacy is determined by the fundamental group,
%and cannot  distinguish connected scenes with homeomorphic complements.

\subsection{The G-image of meridians of a handlebody knot}

%contains not only the information about 
%the topology of $F$ but also the whole span $F\leftarrow\Sigma\rightarrow E$. 
%This invariant is particularly useful in detecting 
%handlebody knots constructed by the twist construction;
%furthermore it is computable and more general than the geometric arguments in the sense that it applies to all %handlebody knots. 
%On the other hand, it is less powerful then geometric arguments when it comes to specific families of handlebody knots like those in \cite{LeeLee:12} and \cite{Mot:90} since it cannot distinguish infinitely many handlebody knots at one time. 
In this subsection we present an invariant of handlebody knots, called the $G$-image
of meridians, which is derived from the fundamental span and
is useful in distinguishing handlebody knots obtained 
by the twist construction in Section \ref{sec:examples}.

\begin{definition}[\textbf{Proper homomorphism}]\label{def:proper_homomorphism}
Let $\SSS=(\inside,\Sigma,\outside)$ be
a handlebody knot.
A surjective homomorphism  
$\phi: \pi_{1}(\outside,\ast) \to G$ is proper 
if the composition 
\[
\phi \circ {i_\outside}_*:
\pi_1(\Sigma,*)\rightarrow \pi_1(\outside,*)\rightarrow G\]
is not onto.  
An element $\alpha$ in $\mathcal{H}(\outside)_G$ is called proper
if $\alpha$ is represented by a proper homomorphism. 
\end{definition}

\begin{definition}[\textbf{$G$-image of meridians}]\label{def:G_image}
Let $\SSS=(\inside,\Sigma,\outside)$ be
a handlebody knot.  Then the $G$-image 
of meridians of $\SSS$ is a set of subgroups of $G$, up to 
automorphisms of $G$, given by
$$
G\operatorname{-im}(\SSS):=\left\{G_\alpha \mid  \alpha\in 
\mathcal{H}(\outside)_G 
 \text{ is proper} 
\right\}.
$$ 
where $G_\alpha$ denotes
the image of the kernel of  
${i_\inside}_*$
under the composition  
$\phi \circ {i_\outside}_*$ under any representative $\phi\in\alpha$.
\end{definition}

\nada{In other words, 
\[G\operatorname{-im}(\SSS):=
\Big\{
\operatorname{Im}(\phi\circ i_{\outside\ast}\vert_{\operatorname{Ker}(\pi_1(\Sigma,\ast)\rightarrow \pi_1(\inside,\ast))}) 
\mid ~ \phi : \pi_1(\outside,\ast) \to G \text{ proper }
\Big\}.
\]}

\noindent
\begin{remark}
Note that $G\operatorname{-im}(\SSS)$ is well-defined only up to automorphism 
of $G$. Also, the $G$-image of meridians of a connected scene 
is independent of the choice of a base point.

The definition of $G$-image of meridians 
applies to any connected scene. 
However, since the kernel of ${i_\inside}_*$
is less manageable for a general $\inside$, 
in the present paper  we restrict our attention 
to the case where $\inside$ is a $3$-handlebody.
In this case, the kernel of ${i_\inside}_*$ 
is the normal closure of meridians of 
the handlebody knot.
\end{remark}

\subsection{An invariant for \knottable disks}
Denote a connected scene $\SSS$ equipped with 
a truly \knottable disk $D$ by $(\SSS,D)$. 
Then, given two such pairs
$(\SSS,D)$ and $(\SSS^{\prime},D^{\prime})$, 
one might want to know whether the connected scenes 
$\SSS^{K}=(\inside^{D,K},\Sigma^{D,K},\outside^{D,K})$ and $\SSS^{\prime,K}=(\inside^{D^{\prime},K},\Sigma^{D^{\prime},K},\outside^{D^{\prime},K})$ 
obtained by 
performing the satellite construction along a knot $K$ 
w.r.t.\ $D$ and $D^{\prime}$, respectively,
are equivalent. 
We define a polynomial invariant for such pairs to investigate the problem. 

\begin{definition}[\textbf{$G$-index}]\label{def:G_index}
Given a prime handlebody knot $\SSS=(\inside,\Sigma,\outside)$ with 
a truly \knottable disk $D$ and a finite group $G$.  
The $G$-index of $(\SSS,D)$ is the polynomial 
\[\operatorname{ind}_{G}[\SSS,D](x)=\sum_{i=1}^{+\infty}n_ix^{i},\]
where $n_{i}$ stands for the number of elements in $\mathcal{H}(\outside)_G$ 
that sends $m$, the associated meridian in $\SSS$ with respect to $D$,
to an element of order $i$ in $G$.  
\end{definition}

Note that the base point might not be on the associated meridian 
so, to evaluate the order of the image of the meridian, one needs 
to connect the meridian with the base point by an arc. 
But changing the connecting arc does not change the conjugacy
class of the image of the meridian in $G$.

By Theorem \ref{Two_truely_explosive_disks}, 
any equivalence between $\SSS^{K}$ and $\SSS^{\prime K}$ 
must send $m$, the associated meridian in $\SSS$ with respect to $D$, 
to $m^{\prime}$, the meridian in $\SSS^{\prime}$ with respect to $D^{\prime}$. 
Hence, we have the following corollary:

\begin{corollary}
Given two irreducible handlebody scenes with truly \knottable disks $
(\SSS,D)$ and $(\SSS^{\prime},D^{\prime})$,
if their $G$-indices are not the same, then the resulting connected scenes 
$\SSS^{K}$ and $\SSS^{\prime,K}$ are not equivalent, for any non-trivial knot $K$.
\end{corollary}

\subsection{An invariant for knot chirality}
Any equivalence between two knots $\SSS=(\inside,\Sigma,\outside)$, 
$\SSS^{\prime}=(\inside^{\prime},\Sigma^{\prime},\outside^{\prime})$
induces an orientation-preserving diffeomorphism from 
$\outside$ to $\outside^{\prime}$, which sends the meridian $[m]$ 
(resp. the preferred longitude $[l]$) in $\pi_1(\outside,\ast)$ to the 
meridian $[m^{\prime}]$ (resp. the preferred longitude 
$[l^{\prime}]$) in $\pi_1(\outside^{\prime},\ast)$.\footnote{We may
choose the base point to be the intersection of the 
meridian and the longitude and any equivalence of connected scenes
can be deformed to one preserving base points.}
Furthermore, 
since it is orientation-preserving, it sends a positively oriented
pair $([m],[l])$ to a positively oriented pair $([m^{\prime}],[l^{\prime}])$.
A meridian--longitude pair $([m],[l])$ is positively oriented if its intersection
number $d([m],[l])$ is $+1$.

Fix a positively oriented pair $([m],[l])$ in $\SSS$ and partition
the set $\mathcal{H}(\outside)_G$ according to the order of the image
of $[m]$ in $G$:
\[\mathcal{H}(\outside)_G=\mathcal{H}(\outside)_1\cup...\cup\mathcal{H}(\outside)_n,\]
where $\mathcal{H}(\outside)_i$ contains those homomorphisms sending 
$[m]$ to an element of order $i$ in $G$.

Now, consider the product $[m][l]$ in $\pi_1(\outside,\ast)$ and observe
that the isomorphism induced by an equivalence between 
$\SSS$ and $\SSS^{\prime}$ sends $[m][l]$ to $[m^{\prime}][l^{\prime}]$.
Thus, we can further partition each $\mathcal{H}(\outside)_i$ according to the order
of the image of $[m][l]$ in $G$:
\[\mathcal{H}(\outside)_i=\mathcal{H}(\outside)_{i,1}\cup...\cup\mathcal{H}(\outside)_{i,n_i},\]
where $\mathcal{H}(\outside)_{i,j}$ contains those homomorphisms in $\mathcal{H}(\outside)_i$ 
that sends $[m][l]$
to an element of order $j$ in $G$; any equivalence 
between $\SSS$ and $\SSS^{\prime}$ induces a bijection
between the sets $\mathcal{H}(\outside)_{i,j}$ and $\mathcal{H}(\outside^{\prime})_{i,j}$. 

In particular, we have the following:

\begin{lemma}
The set $\mathcal{H}(\outside)_{i,j}$ is 
an invariant of the knot $\SSS=(\inside,\Sigma,\outside)$. 
\end{lemma}
   
\begin{corollary}\label{Detecting_chirality}
If a knot $\SSS$ and its mirror image $m\SSS$ 
are ambient isotopic---namely an amphichiral knot, then there is a $1$-$1$ correspondence 
between the sets $\mathcal{H}(\outside)_{i,j}$ and $\mathcal{H}(m\outside)_{i,j}$,
for every $i,j$. 
\end{corollary}

%%%%%%%%%%%%%%%%%%%%%%%%%%%%%%%%%%%%%%%%%%%%%%%%%%%%%%%%%%%%%%%%%%%%%%%%%%%%%

%%%%%%%%%%%%%%%%%%%%%%%%%%%%%%%%%%%%%%%%%%%%%%%%%%%%%%%%%%%%%%%%%%%%%%%%%%%%%
\subsection{Using the invariants in practice}\label{sec:using}
Here we present the result of our computation of the invariants introduced in Section \ref{sec:representations}. Required appcontour commands and how they are used to
obtain the result are recorded in \ref{sec:appcontour}.

\subsubsection{The G-image of meridians}\label{subsec:the_g_image_of_meridians}
Let $G=A_5$, the alternating group of degree $5$. Table \ref{The_A_5_image_HK5_1} 
describes the $A_5$-image of meridians of twisted $\HKfiveone$'s.
\begin{table}[ht]
\caption{The $A_5$-image of meridians of twisted $\HKfiveone$'s.} 
\begin{center}
\begin{tabular}{l|l l l l}
Handlebody knots & The $A_5$-image of meridians\\ \hline
$\HKfiveone$ & \{$A_{4}$, $A_{4}$, $A_{4}$, $D_{10}$, $D_{10}$\}\\ \hline
$-A_1$-$\HKfiveone$ & \{$V_{4}$, $A_{4}$ $A_{4}$, $\mathbb{Z}/5\mathbb{Z}$, $D_{10}$\}\\ \hline
$+A_1$-$\HKfiveone$ & \{$A_{4}$, $A_{4}$, $A_{4}$, $\mathbb{Z}/5\mathbb{Z}$, $D_{10}$\}\\ \hline 
$-A_2$-$\HKfiveone$ & \{$A_{4}$, $A_{4}$, $A_{4}$, $D_{10}$, $D_{10}$\}\\ \hline
$+A_2$-$\HKfiveone$ & \{$V_{4}$, $A_{4}$, $A_{4}$, $D_{10}$, $D_{10}$\}\\ \hline
\end{tabular}
\label{The_A_5_image_HK5_1}
\end{center}
\end{table}
It appears that the $A_5$-image of meridians cannot distinguish $-A_{2}$-$\HKfiveone$ 
and $\HKfiveone$, but in fact, these two handlebody knots are ambient isotopic as shown in Fig: \ref{Equi_btw_minus_A2_HK5_2_and_HK5_2}.
Contrary to the families of handlebody knots in \cite{Mot:90} and \cite{LeeLee:12}, 
the family of handlebody knots constructed by twisting $\HKfiveone$ along $A_2$
contains both equivalent and inequivalent handlebody knots.

\begin{center}
\begin{figure}[ht]
\includegraphics[scale=.15]{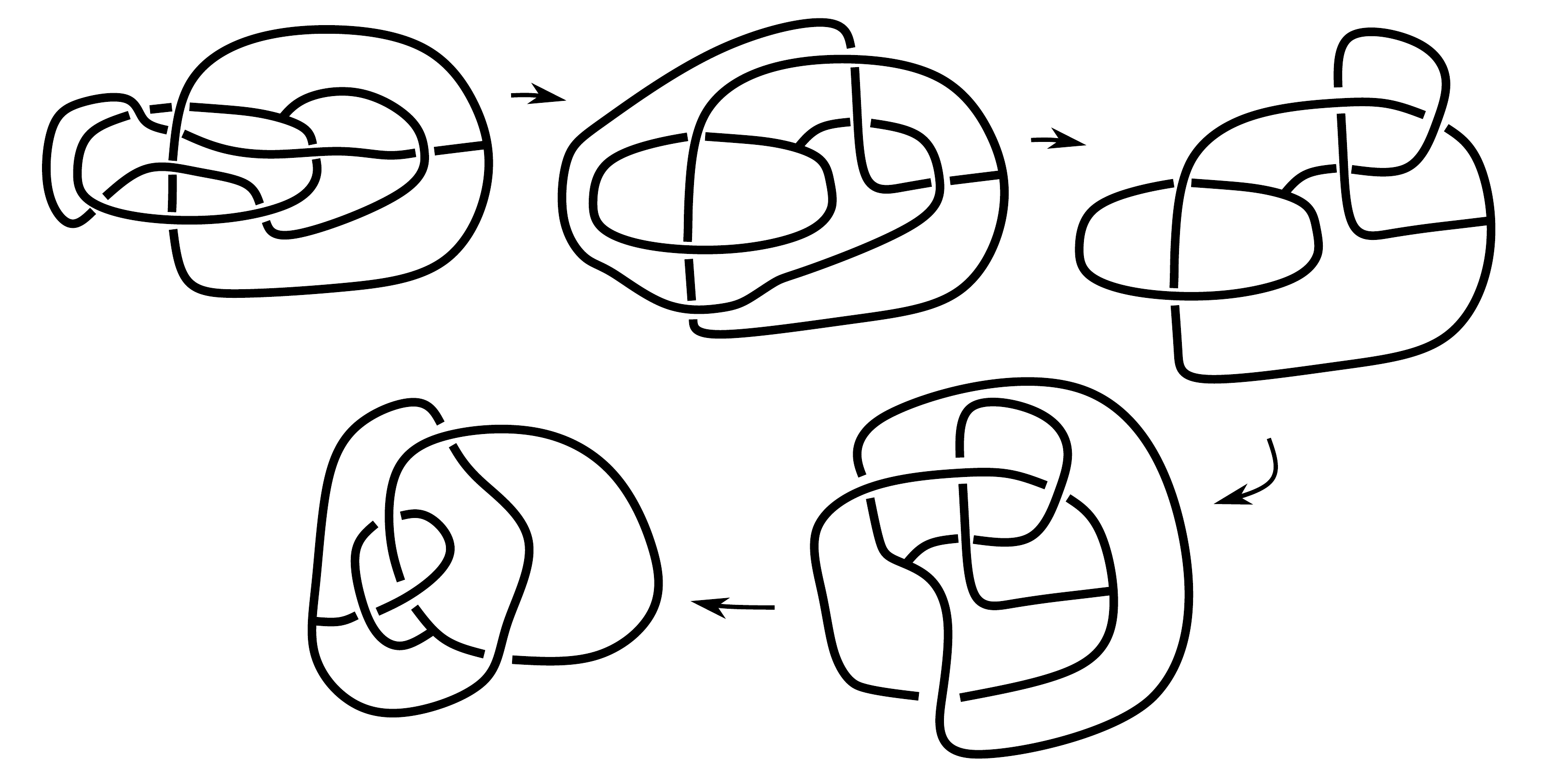}
\caption{Equivalence between $-A_{2}$-$\HKfiveone$ and $\HKfiveone$.}
\label{Equi_btw_minus_A2_HK5_2_and_HK5_2}
\end{figure}
\end{center}

Table \ref{The_A_5_image_HK6_2} presents the $A_5$-image of meridians
of twisted $\HKsixtwo$'s.  
\begin{table}[ht]
\caption{The $A_5$-image of meridians of twisted $\HKsixtwo$'s.} 
\begin{center}
\begin{tabular}{l|l l l l} 
Handlebody knots & The $A_5$-image of meridians\\ \hline
$\HKsixtwo$ & \{$D_{10}$, $D_{10}$, $D_{10}$, $V_{4}$\}\\ \hline
$+A_1$-$\HKsixtwo$ & \{$A_{4}$, $D_{10}$, $D_{10}$, $D_{10}$\}\\ \hline
$-A_1$-$\HKsixtwo$ & \{$A_{4}$, $D_{10}$, $D_{10}$, $D_{10}$\}\\ \hline
$+A_2$-$\HKsixtwo$ & \{$A_{4}$, $D_{10}$, $D_{10}$, $\mathbb{Z}/5\mathbb{Z}$\}\\ \hline
$-A_2$-$\HKsixtwo$ & \{$A_{4}$, $D_{10}$, $D_{10}$, $\mathbb{Z}/5\mathbb{Z}$\}\\ \hline
\end{tabular}
\label{The_A_5_image_HK6_2}
\end{center}
\end{table} 
Among these handlebody knots $+A_1$-$\HKfiveone$ and $-A_1$-$\HKsixtwo$ 
have crossing number $7$ because no other handlebody knot in Ishii et al.'s handlebody knot table, 
except for $-A_1$-$\HKfiveone$, which is $\HKsixfour$, has its complement 
homeomorphic to the complement of $\HKfiveone$; similarly, except for $\HKsixtwo$,
no handlebody knots in the handlebody knot table has its complement homeomorphic to the complement of $-A_1$-$\HKsixtwo$. Therefore, we have proved Proposition \ref{Inequivalent_HK_with_homeo_complement}.

\subsubsection{The G-index of essential transverse disks}

\begin{center}
\begin{tabular}{l|l|l}
Handlebody knot,  & The $S_4$-index & The $A_5$-index\\ 
transverse disk && \\ \hline
($\HKfiveone$,$D^\prime$)   & $5x+20x^2+14x^3+12x^4$ & $4x+14x^2+21x^3+22x^5$\\ \hline 
($\HKfiveone$,$D$) & $5x+22x^2+10x^3+14x^4$ & $4x+15x^2+16x^3+26x^5$\\ \hline
($\HKsixone$,$D^\prime$)   & $5x+20x^2+14x^3+12x^4$ & $4x+24x^2+22x^3+27x^5$\\ \hline
($\HKsixone$,$D$)  & $5x+18x^2+18x^3+10x^4$ & $4x+15x^2+28x^3+30x^5$ \\ \hline
($\HKsixeleven$,$D^\prime$)   & $5x+18x^2+10x^3+10x^4$ & $4x+11x^2+11x^3+16x^5$ \\ \hline 
($\HKsixeleven$,$D$)  & $5x+18x^2+10x^3+10x^4$ & $4x+19x^2+13x^3+6x^5$\\ \hline 
\end{tabular}
\end{center}

%%A figure indicating D and D^\prime
\begin{center}
\begin{figure}[ht]
\includegraphics[scale=.08]{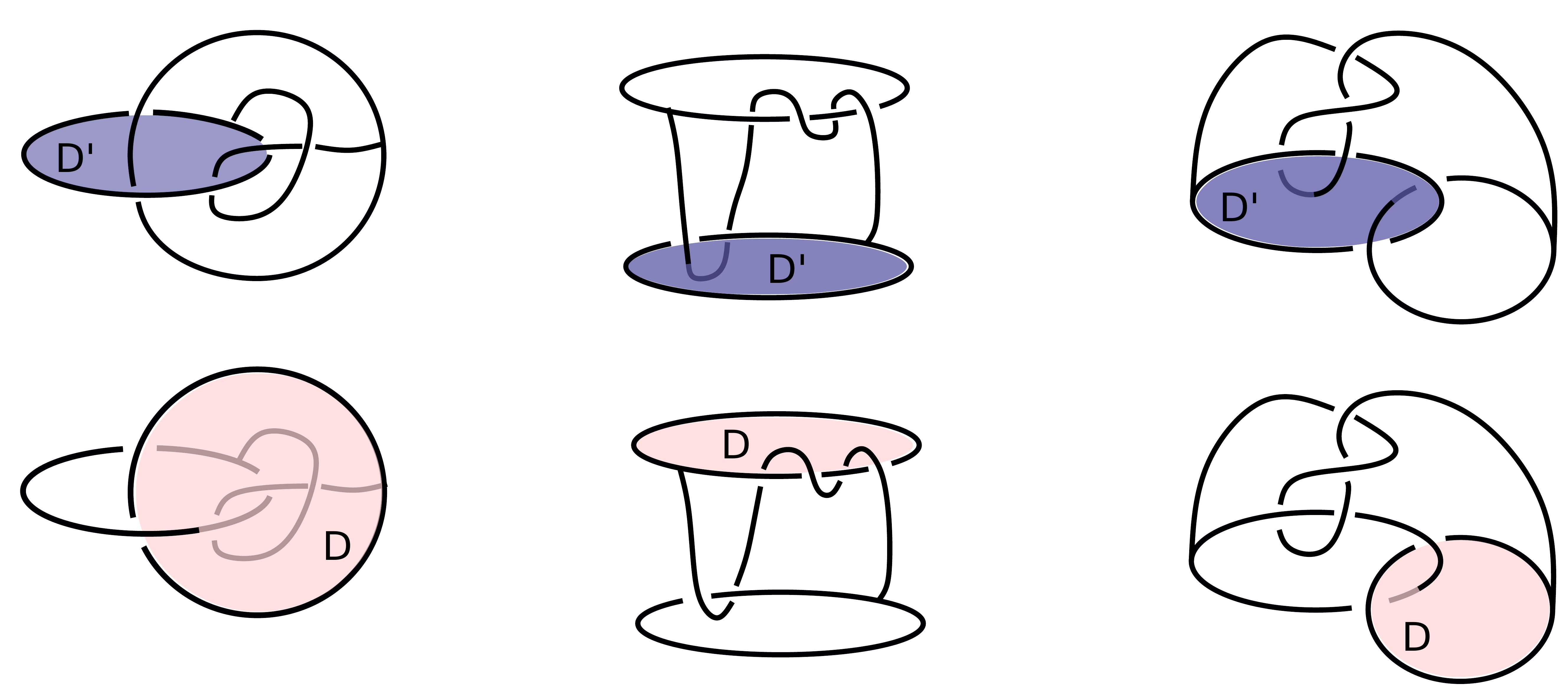}
\caption{$D$ and $D^{\prime}$ in $\HKfiveone$, $\HKsixone$, and $\HKsixeleven$.}
\end{figure}
\end{center}

From the table above we can see that for the handlebody knot diagrams $\HKfiveone$ 
and $\HKsixone$, the $S_4$-index can already distinguish the bi-knotted scenes 
obtained by performing the satellite construction w.r.t.\ $D$ and $D^{\prime}$, 
but for $\HKsixeleven$, we need the $A_5$-index to see that performing the satellite construction w.r.t.\
$D$ and $D^{\prime}$ induce different bi-knotted scenes. 
Note that the $A_4$-index can distinguish none of them.

\subsubsection{Knot chirality}
To illustrate how one can use Corollary \ref{Detecting_chirality} to detect
knot chirality, we consider the right-hand trefoil $K$ and its mirror image $mK$:

\begin{center}
\includegraphics[scale=.18]{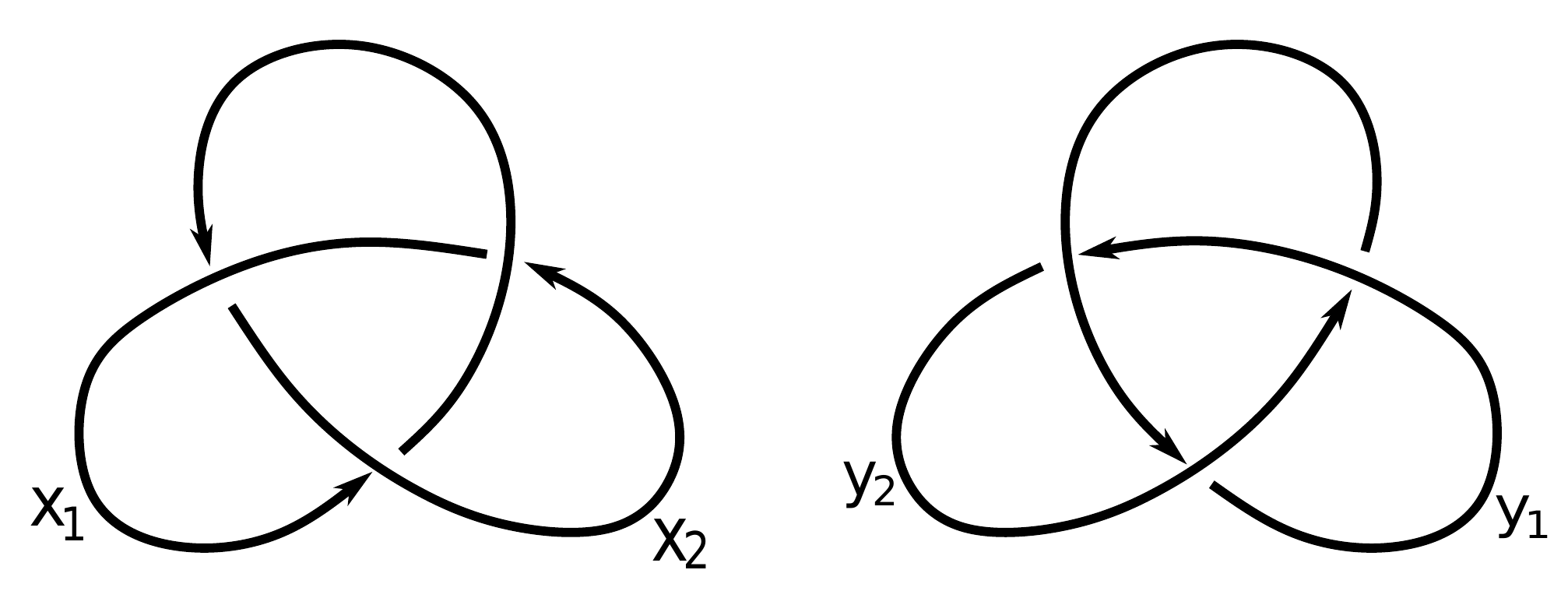}
\end{center}
 
They have isomorphic fundamental groups:
\begin{align*}
\pi_1(K)& =<x_1,x_2\hspace*{1em}\mid \hspace*{1em} x_2x_1x_2=x_1x_2x_1>;\\
\pi_1(mK)& =<y_1,y_2\hspace*{1em}\mid \hspace*{1em} y_2y_1y_2=y_1y_2y_1>.
\end{align*} 

Since different base points are connected by an ambient isotopy, without loss of
generality, we may assume the base points are on the arc labeled with $x_1$ and $y_1$, respectively.
The corresponding meridian-longitude pairs are $(x_1, x_2x_1x_2^{-1}x_1x_2x_1^{-3})$ and 
$(y_1,y_2^{-1}y_1^{-1}y_2y_1^{-1}y_2^{-1}y_1^{3})$.

Now, up to automorphisms of $A_5$, there is only one homomorphism from 
$\pi_1(\outside,\ast)$ to $A_5$ given by
\begin{align*}
x_1 \hspace*{1em}(\text{resp. }y_1)\mapsto (12345);\\
x_2 \hspace*{1em}(\text{resp. }y_2)\mapsto (13542),
\end{align*} 
where $\outside$ is the complement of an open tubular neighborhood of $K$ 
and $\ast\in \partial F$. 
 
In particular, this implies that $\mathcal{H}(\outside)_{5}$ contains only one element.
Computing the image of the product of 
the meridian and longitude in $A_5$ for each of them, we further get
\begin{align*}
x_1x_2x_1x_2^{-1}x_1x_2x_1^{-3} \longmapsto (1);\\
y_1y_2^{-1}y_1^{-1}y_2y_1^{-1}y_2^{-1}y_1^{3} \longmapsto (13524),
\end{align*} 
respectively.
Hence the right-hand trefoil has non-trivial $\mathcal{H}(\outside)_{5,1}$,
whereas the left-hand trefoil has non-trivial $\mathcal{H}(\outside)_{5,5}$,
so they are not equivalent.

\begin{remark}
Fox's proof of inequivalence of the granny knot $K_{g}=(\inside_g,\Sigma_g,\outside_g)$ and 
the square knot $K_{s}=(\inside_s,\Sigma_s,\outside_s)$ \cite[p.39]{Fox:52}
can also be translated in terms of the invariant $\mathcal{H}(\outside)_{i,j}$.
Their inequivalence follows from 
the fact that $\mathcal{H}(\outside_{s})_{5,1}$ contains only 
one element but $\mathcal{H}(\outside_{g})_{5,1}$ contains two.
\end{remark}

In a similar manner, we may compute the orbit set $\mathcal{H}(\outside)_{A_{5}}$ 
for $\Kninefortitwo=(\inside,\Sigma,\outside)$; by the result of computations in Appcontour, 
$\mathcal{H}(\outside)_{A_{5}}$
consists of seven elements, and three of them are in $\mathcal{H}(\outside)_{3}$: 
they send $([m],[l])$, the meridian--longitude pair in $\Kninefortitwo$, 
to $((3,4,5),())$, $((2,3,5),())$, $((1,4,5),(1,4,5))$, respectively. 
So, $\mathcal{H}(\outside)_{3,3}$ contains three elements. The third 
representation corresponds to the representation sending the meridian-longitude pair in 
$m\Kninefortitwo$ to $((1,4,5),(1,5,4))$, and hence 
$\mathcal{H}(m\outside)_{3,3}$ contains only two elements.

In the case of $\Ktenseventyone=(\inside,\Sigma,\outside)$, $G=A_5$ or $S_5$ 
is not large enough to see its chiraliy, and we need to consider $\mathcal{H}(\outside)_{A_{6}}$. 
Computations in Appcontour show there are
three elements in $\mathcal{H}(\outside)_{5}$ and three in $\mathcal{H}(\outside)_{4}$.
These representations induce the following assignments of $([m],[l])$, the meridian--longitude pair of $\Ktenseventyone$: 
\begin{align*}
([m],[l])&\longmapsto ((2 3 4 5 6),())\\
         &\longmapsto ((1 5 2 4 6),(1 5 2 4 6))\\ 
         &\longmapsto ((1 6 4 2 5),(1 6 4 2 5))\\
([m],[l])&\longmapsto  ((1 2)(3 4 5 6),())\\ 
         &\longmapsto  ((1 6)(2 3 4 5),(1 6)(2 3 4 5))\\
         &\longmapsto  ((1 3 6 2)(4 5),(1 3 6 2)(4 5))       
\end{align*}
This implies that $\mathcal{H}(\outside)_{5,5}$ contains three elements and $\mathcal{H}(\outside)_{4,2}$ 
two elements, whereas, in $m\Ktenseventyone$, there is only one element 
in $\mathcal{H}(m\outside)_{5,5}$ and none in $\mathcal{H}(m\outside)_{4,2}$.

%%K10_71 Both H(\outside)_{5} and H(\outside)_{4} can do

%%Shall I mention square knot and granny knot?  

%%%%%%%%%%%%%%%%%%%%%%%%%%%%%%%%%%%%%%%%%%%%%%%%%%%%%%%%%%%%%%%%%%%%%%%%%%%%%

%%%%%%%%%%%%%%%%%%%%%%%%%%%%%%%%%%%%%%%%%%%%%%%%%%%%%%%%%%%%%%%%%%%%%%%%%%%%%
\subsection{Using \texttt{appcontour}}\label{sec:appcontour}
The computer software \texttt{appcontour} \cite{appcontour} is 
a tool originally developed to deal with
``apparent contours'', i.e. drawings that describe smooth solid objects by projecting fold
lines onto a plane.

It was recently extended by adding the capability of computing homomorphisms of groups described by group presentation to a finite group as mentioned in the beginning of Section \ref{sec:representations}.

As an example, we can count the number of representations of handlebody knot $\HKfiveone$ in
the alternating group $A_5$ with the command
\begin{Verbatim}[fontsize=\small,frame=single,xleftmargin=10mm,commandchars=\\\{\}]
$ \graybox{contour --out ks_A5 HK5_1}
Result: 61
\end{Verbatim}
with the counting done in $A_5$ up to conjugacy by a permutation in $S_5$.

Unfortunately, the computation of $\piout$ performed by appcontour gives a presentation with no information
about the correspondance of the generators with actual loops in $\outside$.
For example, for $\HKfiveone$ we get the following presentation of $\piout$
\begin{Verbatim}[fontsize=\small,frame=single,xleftmargin=10mm,commandchars=\\\{\}]
$ \graybox{contour --out fg HK5_1}
Finitely presented group with 3 generators
<a,b,c; abAcaBAbCbcB>
\end{Verbatim}
with no information about the loops corresponding to the three generators.
Here capital letters are used as a quick way to refer to the inverse of the generators.
For this reason we need to carefully construct by hand an analogue of a Wirtinger presentation
of our scene. For instance, for $\HKfiveone$,
we could use the one shown in Fig. \ref{fig:HK5_2_HK6_2_wirtinger} (left).
\begin{figure}[ht]
    \centering
    \subfloat{ {\includegraphics[scale=.12]{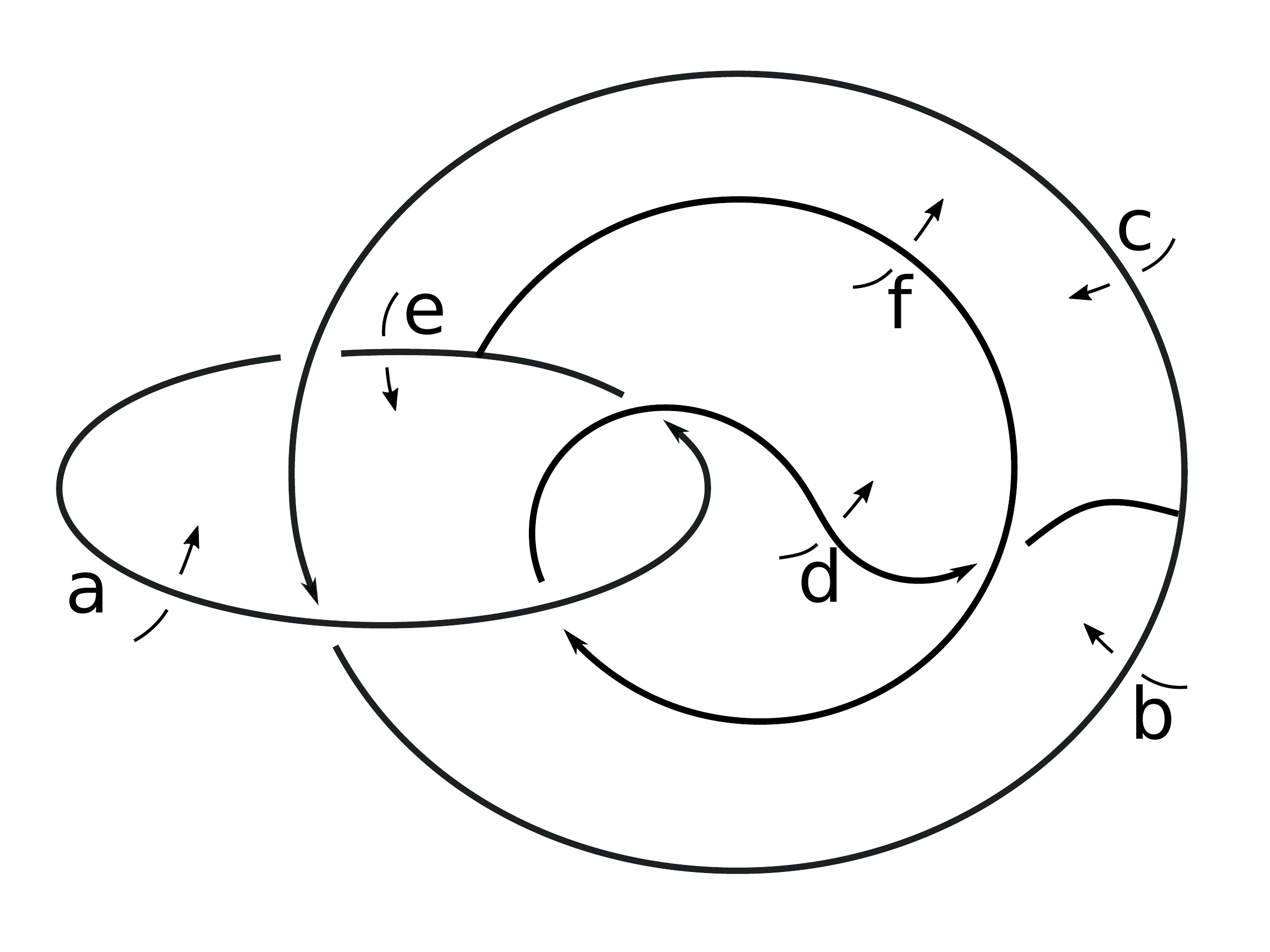} } }%
    \qquad
    \subfloat{ {\includegraphics[scale=.12]{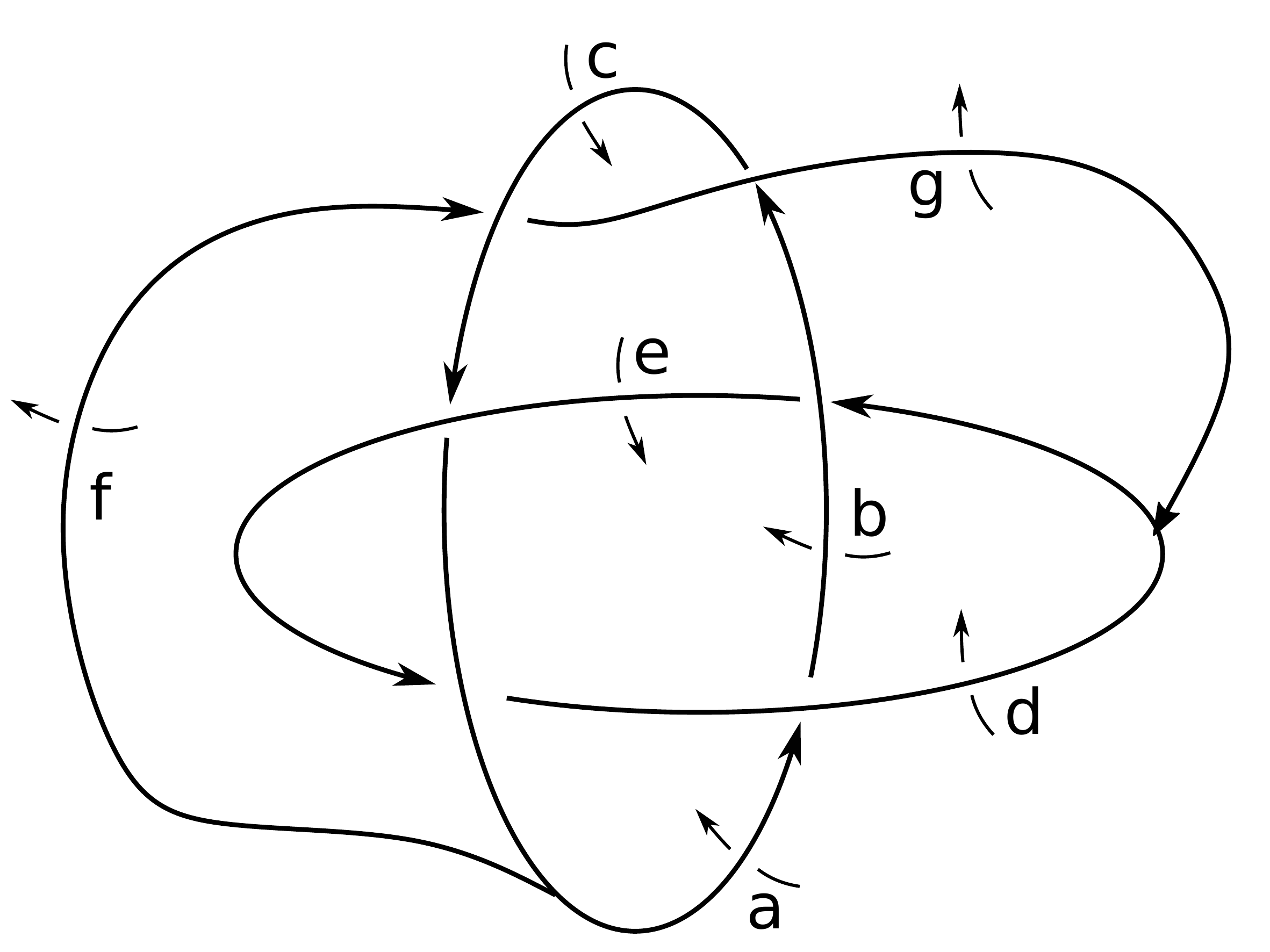}} }%
    \caption{Wirtinger presentations of $\HKfiveone$ and $\HKsixtwo$.}
    \label{fig:HK5_2_HK6_2_wirtinger}
\end{figure}

The syntax that can
be fed to the software is as follows:
\begin{Verbatim}[fontsize=\small,frame=single,xleftmargin=10mm,commandchars=\\\{\}]
fpgroup {<a,b,c,d,e,f; baCA,faDA,feDAd,acEC,BcFDf; b,FADadaf,A,FACdaf>}
\end{Verbatim}
where we have the possibility to add a list of selected elements---elements after the second
semicolon---that allows us to keep track of specific elements in $\piout$.
Here selected elements \#1 (\texttt{b}) and \#2 (\texttt{FADadaf})%
\footnote{%
The base point is chosen near the letter \texttt{b} in the diagram, so that
the second meridian (generator \texttt{a}) requires a connecting path from the
base point (\texttt{FAD}) and then back to the base point (\texttt{daf}) along the curve
connecting the circles of the handcuff.}
($m_1$ and $m_2$ below)
correspond to the meridians
of $\HKfiveone$ induced by the two circles in the handcuff graph diagram (Fig. \ref{fig:HK5_2_HK6_2_wirtinger}, left); selected elements
\#3 and \#4, denoted by $l_1$ and $l_2$ in the following, 
are induced by the two circles, 
which are the other two generators of the fundamental group of
$\partial \inside$ that pair
with \#1 and \#2, respectively.

We create a file named, say, \texttt{HK5\_1.wirtinger} containing the description above
to be used as input to \texttt{appcontour}
and ask for the description of all elements in $\mathcal{H}(\outside)_{A_5}$ with

\begin{Verbatim}[fontsize=\small,frame=single,xleftmargin=10mm,commandchars=\\\{\}]  
$\graybox{contour ks_A5 HK5_1.wirtinger -v} 
\end{Verbatim}
which results in a long list of all $61$ homomorphisms described by indicating the image
of the three generators followed by the corresponding image of the four selected elements.

To get the table \ref{The_A_5_image_HK5_1}, we locate the proper homomorphisms in
$\mathcal{H}(\outside)_{A_5}$:
\begin{Verbatim}[fontsize=\tiny,frame=single,xleftmargin=10mm,commandchars=\\\{\}]
[...]
====== Homomorphism #10 defined by the permutations:
(3 4 5)
(2 3 4)
(1 3)(2 4)
Selected element #1 -> (2 5 4)
Selected element #2 -> (2 4 3)
Selected element #3 -> (2 3 5)
Selected element #4 -> (2 4)(3 5)
[...]
====== Homomorphism #12 defined by the permutations:
(3 4 5)
(2 3 5)
(1 2)(3 5)
Selected element #1 -> (2 3 5)
Selected element #2 -> (2 5 3)
Selected element #3 -> (2 4 5)
Selected element #4 -> (2 4 5)
[...]
====== Homomorphism #28 defined by the permutations:
(2 3)(4 5)
(3 4 5)
(1 5)(3 4)
Selected element #1 -> (2 5)(3 4)
Selected element #2 -> (3 5 4)
Selected element #3 -> (2 4 5)
Selected element #4 -> (2 3 5)
[...]
====== Homomorphism #32 defined by the permutations:
(2 3)(4 5)
(1 2)(3 4)
(1 4 5 3 2)
Selected element #1 -> (1 5)(2 4)
Selected element #2 -> (1 5)(2 4)
Selected element #3 -> (1 3)(2 5)
Selected element #4 -> (1 3 5 4 2)
[...]
====== Homomorphism #48 defined by the permutations:
(1 2 3 4 5)
(2 5)(3 4)
(1 2 4 3 5)
Selected element #1 -> (1 5 4 3 2)
Selected element #2 -> (1 3)(4 5)
Selected element #3 -> (1 4)(2 3)
Selected element #4 -> (1 4)(2 3)
[...]
\end{Verbatim} 
As an example, the first entry in table \ref{The_A_5_image_HK5_1} is the normal closure of
the group generated by $m_1$ and $m_2$ of homomorphism \#10 in the group generated by
$m_1$, $m_2$, $l_1$, $l_2$, which in this case coincide and are isomorphic to $A_4$.
 
To compute the $A_5$-image of meridians for each of the twisted $\HKfiveone$s 
in Fig. \ref{twistedHK51_A1} and \ref{twistedHK51_A2}, it suffices to know the image of the meridians
in a complete system of meridians under the proper homomorphisms. 
We may choose the complete system of meridians induced 
from Fig. \ref{twistedHK51_A1} and \ref{twistedHK51_A2} (regarded as handcuff graph diagrams) 
and use the twist construction to identify 
these meridians with elements in the fundamental group 
of the closure of the complement of $\HKfiveone$ as follows: 
\begin{align*}
m_1^{-A_1}&\dashrightarrow m_1  &  m_1^{-A_2} &\dashrightarrow m_1^{-1}l_1^{-1} m_1\\
m_2^{-A_1}&\dashrightarrow m_2l_2  &  m_2^{-A_2}&\dashrightarrow m_2l_2m_2\\
m_1^{+A_1}&\dashrightarrow m_1  &   m_1^{+A_2}&\dashrightarrow l_1 m_1 m_1\\
m_2^{+A_1}&\dashrightarrow m_2l_2^{-1} & m_2^{+A_2}&\dashrightarrow l_2^{-1}.   
\end{align*} 
This enables us to compute the $A_5$-image of meridians for twisted $\HKfiveone$ using
proper homomorphisms from the fundamental group of the complement of $\HKfiveone$ to
$A_5$.

Similarly, running the command
\begin{Verbatim}[fontsize=\small,frame=single,xleftmargin=10mm,commandchars=\\\{\}]
$ \graybox{contour ks_A5 HK6_2.wirtinger -v}
\end{Verbatim}
where file \texttt{HK6\_2.wirtinger} contains the Wirtinger presentation of $\HKsixtwo$ in
Fig. \ref{fig:HK5_2_HK6_2_wirtinger} (right), we can get proper homomorphisms of the fundamental group of the complement of $\HKsixtwo$
to $A_5$. Table \ref{The_A_5_image_HK6_2} can then be completed by
identifying the meridians in the 
complete systems of meridians of twisted $\HKsixtwo$'s in
Fig. \ref{twistedHK62_A1} and \ref{twistedHK62_A2} (viewed as handcuff graph diagrams)
with combinations of $m_1$, $m_2$, $l_1$, and $l_2$, meridians and two circles in Fig. \ref{fig:HK5_2_HK6_2_wirtinger} (right),  
of $\HKsixtwo$ using the twist construction:
\begin{align*}
 m_1^{-A_1}&\dashrightarrow m_1      &   m_1^{-A_2} &\dashrightarrow m_1m_1m_1l_1\\
 m_2^{-A_1}&\dashrightarrow m_2l_2   &   m_2^{-A_2} &\dashrightarrow m_2^{-1}l_2^{-1}\\
 m_1^{+A_1}&\dashrightarrow m_1      & m_1^{-A_2} &\dashrightarrow m_1l_1^{-1}m_1^{-1}m_1^{-1}\\
 m_2^{+A_1}&\dashrightarrow m_2l_2^{-1} &   m_2^{+A_2} &\dashrightarrow m_2l_2m_2m_2.   
\end{align*} 
In fact, for $A_2$-twisted $\HKsixtwo$ as well as $A_2$-twisted
 $\HKfiveone$, it is easier to represent $m_i^{\pm A_2}$
in terms of $\tilde{l}_1$ and $\tilde{l}_2$, the boundary of $A_2$, 
instead of $l_1$ and $l_2$.

The $S_4$-index and the $A_5$-index for a transverse disk in a prime handlebody knot
can also be computed with the same command

\texttt{contour ks$\_G$ }file-containing-\texttt{$\{$group presentation; selected elements$\}$}  

\noindent
with one of the meridians
$m$ corresponding
to the associated meridian of the transverse disk
in the selected elements. In our case, we could 
arrange the selected meridians to be the associated meridians
of the two transverse disks induced from a handcuff graph diagram.
One could add the multiple of $i$ copies of $m$
as an extra relator to get the coefficient of $x^i$
directly.

%%%%%%%%%%%%%%%%%%%%%%%%%%%%%%%%%%%%%%%%%%%%%%%%%%%%%%%%%%%%%%%%%%%%%%%%%%%%%

%%%%%%%%%%%%%%%%%%%%%%%%%%%%%%%%%%%%%%%%%%%%%%%%%%%%%%%%%%%%%%%%%%%%%%%%%%%%%
%\ifdraft
%\section{Draft material}\label{sec:draft}
%\input{draft}
%\fi
%%%%%%%%%%%%%%%%%%%%%%%%%%%%%%%%%%%%%%%%%%%%%%%%%%%%%%%%%%%%%%%%%%%%%%%%%%%%%

%%%%%%%%%%%%%%%%%%%%%%%%%%%%%%%%%%%%%%%%%%%%%%%%%%%%%%%%%%%%%%%%%%%%%%%%%%%%%
%\section{Fundamental structure when $\Sigma$ is not connected}\label{sec:fundamental2}
%%%%%%%%%%%%%%%%%%%%%%%%%%%%%%%%%%%%%%%%%%%%%%%%%%%%%%%%%%%%%%%%%%%%%%%%%%%%%
%Try to see whether the conjecture is true in the case of links. Next,
%in the general case...

%%%%%%%%%%%%%%%%%%%%%%%%%%%%%%%%%%%%%%%%%%%%%%%%%%%%%%%%%%%%%%%%%%%%%%%%%%%%%
%\section{Using $\R^3$ instead of $\sphere$}
%%%%%%%%%%%%%%%%%%%%%%%%%%%%%%%%%%%%%%%%%%%%%%%%%%%%%%%%%%%%%%%%%%%%%%%%%%%%%
%\input r3.tex

%%%%%%%%%%%%%%%%%%%%%%%%%%%%%%%%%%%%%%%%%%%%%%%%%%%%%%%%%%%%%%%%%%%%%%%%%%%%%

\end{document}

%% file: Signs_of_twistings_annulus.pdf_tex
%% Accompanies image file 'Signs_of_twistings_annulus.pdf' (pdf, eps, ps)
\begingroup%
  \makeatletter%
  \providecommand\color[2][]{%
    \errmessage{(Inkscape) Color is used for the text in Inkscape, but the package 'color.sty' is not loaded}%
    \renewcommand\color[2][]{}%
  }%
  \providecommand\transparent[1]{%
    \errmessage{(Inkscape) Transparency is used (non-zero) for the text in Inkscape, but the package 'transparent.sty' is not loaded}%
    \renewcommand\transparent[1]{}%
  }%
  \providecommand\rotatebox[2]{#2}%
  \newcommand*\fsize{\dimexpr\f@size pt\relax}%
  \newcommand*\lineheight[1]{\fontsize{\fsize}{#1\fsize}\selectfont}%
  \ifx\svgwidth\undefined%
    \setlength{\unitlength}{1984.2519685bp}%
    \ifx\svgscale\undefined%
      \relax%
    \else%
      \setlength{\unitlength}{\unitlength * \real{\svgscale}}%
    \fi%
  \else%
    \setlength{\unitlength}{\svgwidth}%
  \fi%
  \global\let\svgwidth\undefined%
  \global\let\svgscale\undefined%
  \makeatother%
  \begin{picture}(1,0.24285714)%
    \lineheight{1}%
    \setlength\tabcolsep{0pt}%
    \put(0,0){\includegraphics[width=\unitlength,page=1]{Signs_of_twistings_annulus.pdf}}%
    \put(0.1865346,0.11914767){\color[rgb]{0,0,0}\makebox(0,0)[lt]{\lineheight{1.25}\smash{\begin{tabular}[t]{l}$t^{-1}$\end{tabular}}}}%
    \put(0,0){\includegraphics[width=\unitlength,page=2]{Signs_of_twistings_annulus.pdf}}%
    \put(0.82566925,0.10678901){\color[rgb]{0,0,0}\makebox(0,0)[lt]{\lineheight{1.25}\smash{\begin{tabular}[t]{l}$t$\end{tabular}}}}%
    \put(0,0){\includegraphics[width=\unitlength,page=3]{Signs_of_twistings_annulus.pdf}}%
    \put(0.21568625,0.06302521){\color[rgb]{0,0,0}\makebox(0,0)[lt]{\lineheight{1.25}\smash{\begin{tabular}[t]{l}{\tiny $D_1\times I$}\end{tabular}}}}%
    \put(0.86694676,0.05532213){\color[rgb]{0,0,0}\makebox(0,0)[lt]{\lineheight{1.25}\smash{\begin{tabular}[t]{l}{\tiny $D_1\times I$}\end{tabular}}}}%
    \put(0.01470588,0.21078432){\color[rgb]{0,0,0}\makebox(0,0)[lt]{\lineheight{1.25}\smash{\begin{tabular}[t]{l}{\footnotesize $A$}\end{tabular}}}}%
    \put(0.66316528,0.21288515){\color[rgb]{0,0,0}\makebox(0,0)[lt]{\lineheight{1.25}\smash{\begin{tabular}[t]{l}{\footnotesize $A$}\end{tabular}}}}%
    \put(1.17514718,-0.35861642){\color[rgb]{0,0,0}\makebox(0,0)[lt]{\begin{minipage}{0.58492693\unitlength}\raggedright \end{minipage}}}%
  \end{picture}%
\endgroup%

%% file: HK5_1_A1.pdf_tex
%% Accompanies image file 'HK5_1_A1.pdf' (pdf, eps, ps)
\begingroup%
  \makeatletter%
  \providecommand\color[2][]{%
    \errmessage{(Inkscape) Color is used for the text in Inkscape, but the package 'color.sty' is not loaded}%
    \renewcommand\color[2][]{}%
  }%
  \providecommand\transparent[1]{%
    \errmessage{(Inkscape) Transparency is used (non-zero) for the text in Inkscape, but the package 'transparent.sty' is not loaded}%
    \renewcommand\transparent[1]{}%
  }%
  \providecommand\rotatebox[2]{#2}%
  \newcommand*\fsize{\dimexpr\f@size pt\relax}%
  \newcommand*\lineheight[1]{\fontsize{\fsize}{#1\fsize}\selectfont}%
  \ifx\svgwidth\undefined%
    \setlength{\unitlength}{2267.71653543bp}%
    \ifx\svgscale\undefined%
      \relax%
    \else%
      \setlength{\unitlength}{\unitlength * \real{\svgscale}}%
    \fi%
  \else%
    \setlength{\unitlength}{\svgwidth}%
  \fi%
  \global\let\svgwidth\undefined%
  \global\let\svgscale\undefined%
  \makeatother%
  \begin{picture}(1,0.375)%
    \lineheight{1}%
    \setlength\tabcolsep{0pt}%
    \put(0,0){\includegraphics[width=\unitlength,page=1]{HK5_1_A1.pdf}}%
    \put(0.08947608,0.1663595){\color[rgb]{0,0,0}\makebox(0,0)[lt]{\lineheight{1.25}\smash{\begin{tabular}[t]{l}$A_1$\end{tabular}}}}%
    \put(0.24347355,0.1816491){\color[rgb]{0,0,0}\makebox(0,0)[lt]{\lineheight{1.25}\smash{\begin{tabular}[t]{l}{\tiny $D_1$}\end{tabular}}}}%
    \put(0,0){\includegraphics[width=\unitlength,page=2]{HK5_1_A1.pdf}}%
    \put(0.61741931,0.04277376){\color[rgb]{0,0,0}\makebox(0,0)[lt]{\lineheight{1.25}\smash{\begin{tabular}[t]{l}$A_1$\end{tabular}}}}%
    \put(0,0){\includegraphics[width=\unitlength,page=3]{HK5_1_A1.pdf}}%
    \put(0.85964797,0.08377095){\color[rgb]{0,0,0}\makebox(0,0)[lt]{\lineheight{1.25}\smash{\begin{tabular}[t]{l}{\tiny $D_1$}\end{tabular}}}}%
    \put(0.67366947,0.27941176){\color[rgb]{0,0,0}\makebox(0,0)[lt]{\begin{minipage}{0.0182073\unitlength}\raggedright \end{minipage}}}%
    \put(0.09453782,0.2114846){\color[rgb]{0,0,0}\makebox(0,0)[lt]{\lineheight{1.25}\smash{\begin{tabular}[t]{l}$+$\end{tabular}}}}%
    \put(0.81582634,0.13235294){\color[rgb]{0,0,0}\makebox(0,0)[lt]{\lineheight{1.25}\smash{\begin{tabular}[t]{l}$+$\end{tabular}}}}%
    \put(0.27170868,0.00420166){\color[rgb]{0,0,0}\makebox(0,0)[lt]{\lineheight{1.25}\smash{\begin{tabular}[t]{l} \end{tabular}}}}%
  \end{picture}%
\endgroup%

%% file: HK5_1_A2.pdf_tex
%% Accompanies image file 'HK5_1_A2.pdf' (pdf, eps, ps)
\begingroup%
  \makeatletter%
  \providecommand\color[2][]{%
    \errmessage{(Inkscape) Color is used for the text in Inkscape, but the package 'color.sty' is not loaded}%
    \renewcommand\color[2][]{}%
  }%
  \providecommand\transparent[1]{%
    \errmessage{(Inkscape) Transparency is used (non-zero) for the text in Inkscape, but the package 'transparent.sty' is not loaded}%
    \renewcommand\transparent[1]{}%
  }%
  \providecommand\rotatebox[2]{#2}%
  \newcommand*\fsize{\dimexpr\f@size pt\relax}%
  \newcommand*\lineheight[1]{\fontsize{\fsize}{#1\fsize}\selectfont}%
  \ifx\svgwidth\undefined%
    \setlength{\unitlength}{2267.71653543bp}%
    \ifx\svgscale\undefined%
      \relax%
    \else%
      \setlength{\unitlength}{\unitlength * \real{\svgscale}}%
    \fi%
  \else%
    \setlength{\unitlength}{\svgwidth}%
  \fi%
  \global\let\svgwidth\undefined%
  \global\let\svgscale\undefined%
  \makeatother%
  \begin{picture}(1,0.4375)%
    \lineheight{1}%
    \setlength\tabcolsep{0pt}%
    \put(0.15554862,0.18217099){\color[rgb]{0,0,0}\rotatebox{0.38145403}{\makebox(0,0)[lt]{\lineheight{1.25}\smash{\begin{tabular}[t]{l}{\footnotesize $D_2$}\end{tabular}}}}}%
    \put(0,0){\includegraphics[width=\unitlength,page=1]{HK5_1_A2.pdf}}%
    \put(0.14050949,0.07129908){\color[rgb]{0,0,0}\rotatebox{-0.06075653}{\makebox(0,0)[lt]{\lineheight{1.25}\smash{\begin{tabular}[t]{l}$A_2$\end{tabular}}}}}%
    \put(0,0){\includegraphics[width=\unitlength,page=2]{HK5_1_A2.pdf}}%
    \put(0.179972,0.24719888){\color[rgb]{0,0,0}\makebox(0,0)[lt]{\lineheight{1.25}\smash{\begin{tabular}[t]{l}$+$\end{tabular}}}}%
    \put(0.16526609,0.3732493){\color[rgb]{0,0,0}\makebox(0,0)[lt]{\lineheight{1.25}\smash{\begin{tabular}[t]{l}$-$\end{tabular}}}}%
    \put(0.6792717,0.32913166){\color[rgb]{0,0,0}\makebox(0,0)[lt]{\lineheight{1.25}\smash{\begin{tabular}[t]{l}$+$\end{tabular}}}}%
    \put(0.6715686,0.38935574){\color[rgb]{0,0,0}\makebox(0,0)[lt]{\lineheight{1.25}\smash{\begin{tabular}[t]{l}$-$\end{tabular}}}}%
    \put(0.87464981,0.13655462){\color[rgb]{0,0,0}\makebox(0,0)[lt]{\lineheight{1.25}\smash{\begin{tabular}[t]{l}$-$\end{tabular}}}}%
    \put(0.94726547,0.10046675){\color[rgb]{0,0,0}\makebox(0,0)[lt]{\lineheight{1.25}\smash{\begin{tabular}[t]{l}$+$\end{tabular}}}}%
    \put(0.73249298,0.15476191){\color[rgb]{0,0,0}\makebox(0,0)[lt]{\lineheight{1.25}\smash{\begin{tabular}[t]{l}$\simeq$\end{tabular}}}}%
    \put(0.63583312,0.23846952){\color[rgb]{0,0,0}\rotatebox{-0.06075657}{\makebox(0,0)[lt]{\lineheight{1.25}\smash{\begin{tabular}[t]{l}{\footnotesize $A_2$}\end{tabular}}}}}%
    \put(0.65395416,0.29245649){\color[rgb]{0,0,0}\rotatebox{0.38145403}{\makebox(0,0)[lt]{\lineheight{1.25}\smash{\begin{tabular}[t]{l}{\tiny $D_2$}\end{tabular}}}}}%
    \put(0.87170736,0.01859841){\color[rgb]{0,0,0}\rotatebox{-0.06075657}{\makebox(0,0)[lt]{\lineheight{1.25}\smash{\begin{tabular}[t]{l}{\footnotesize $A_2$}\end{tabular}}}}}%
    \put(0.92274378,0.05214258){\color[rgb]{0,0,0}\rotatebox{0.38145403}{\makebox(0,0)[lt]{\lineheight{1.25}\smash{\begin{tabular}[t]{l}{\tiny $D_2$}\end{tabular}}}}}%
  \end{picture}%
\endgroup%

%% file: HK6_2_A1.pdf_tex
%% Accompanies image file 'HK6_2_A1.pdf' (pdf, eps, ps)
\begingroup%
  \makeatletter%
  \providecommand\color[2][]{%
    \errmessage{(Inkscape) Color is used for the text in Inkscape, but the package 'color.sty' is not loaded}%
    \renewcommand\color[2][]{}%
  }%
  \providecommand\transparent[1]{%
    \errmessage{(Inkscape) Transparency is used (non-zero) for the text in Inkscape, but the package 'transparent.sty' is not loaded}%
    \renewcommand\transparent[1]{}%
  }%
  \providecommand\rotatebox[2]{#2}%
  \newcommand*\fsize{\dimexpr\f@size pt\relax}%
  \newcommand*\lineheight[1]{\fontsize{\fsize}{#1\fsize}\selectfont}%
  \ifx\svgwidth\undefined%
    \setlength{\unitlength}{1417.32283465bp}%
    \ifx\svgscale\undefined%
      \relax%
    \else%
      \setlength{\unitlength}{\unitlength * \real{\svgscale}}%
    \fi%
  \else%
    \setlength{\unitlength}{\svgwidth}%
  \fi%
  \global\let\svgwidth\undefined%
  \global\let\svgscale\undefined%
  \makeatother%
  \begin{picture}(1,0.6)%
    \lineheight{1}%
    \setlength\tabcolsep{0pt}%
    \put(0,0){\includegraphics[width=\unitlength,page=1]{HK6_2_A1.pdf}}%
    \put(1.64826904,-0.44155469){\color[rgb]{0,0,0}\makebox(0,0)[lt]{\begin{minipage}{0.28413794\unitlength}\raggedright \end{minipage}}}%
    \put(0.25862423,0.26341912){\color[rgb]{0,0,0}\makebox(0,0)[lt]{\lineheight{1.25}\smash{\begin{tabular}[t]{l}$A_1$\end{tabular}}}}%
    \put(0.38351307,0.30870176){\color[rgb]{0,0,0}\makebox(0,0)[lt]{\lineheight{1.25}\smash{\begin{tabular}[t]{l} \end{tabular}}}}%
    \put(0.7236217,0.23879356){\color[rgb]{0,0,0}\makebox(0,0)[lt]{\lineheight{1.25}\smash{\begin{tabular}[t]{l}{\small $D_1$}\end{tabular}}}}%
    \put(0.83004008,0.49522673){\color[rgb]{0,0,0}\makebox(0,0)[lt]{\lineheight{1.25}\smash{\begin{tabular}[t]{l} \end{tabular}}}}%
    \put(0.57928577,0.29214285){\color[rgb]{0,0,0}\makebox(0,0)[lt]{\lineheight{1.25}\smash{\begin{tabular}[t]{l}$+$\end{tabular}}}}%
  \end{picture}%
\endgroup%

%% file: HK6_2_A2.pdf_tex
%% Accompanies image file 'HK6_2_A2.pdf' (pdf, eps, ps)
\begingroup%
  \makeatletter%
  \providecommand\color[2][]{%
    \errmessage{(Inkscape) Color is used for the text in Inkscape, but the package 'color.sty' is not loaded}%
    \renewcommand\color[2][]{}%
  }%
  \providecommand\transparent[1]{%
    \errmessage{(Inkscape) Transparency is used (non-zero) for the text in Inkscape, but the package 'transparent.sty' is not loaded}%
    \renewcommand\transparent[1]{}%
  }%
  \providecommand\rotatebox[2]{#2}%
  \newcommand*\fsize{\dimexpr\f@size pt\relax}%
  \newcommand*\lineheight[1]{\fontsize{\fsize}{#1\fsize}\selectfont}%
  \ifx\svgwidth\undefined%
    \setlength{\unitlength}{1417.32283465bp}%
    \ifx\svgscale\undefined%
      \relax%
    \else%
      \setlength{\unitlength}{\unitlength * \real{\svgscale}}%
    \fi%
  \else%
    \setlength{\unitlength}{\svgwidth}%
  \fi%
  \global\let\svgwidth\undefined%
  \global\let\svgscale\undefined%
  \makeatother%
  \begin{picture}(1,0.6)%
    \lineheight{1}%
    \setlength\tabcolsep{0pt}%
    \put(0,0){\includegraphics[width=\unitlength,page=1]{HK6_2_A2.pdf}}%
    \put(0.53285712,0.51933809){\color[rgb]{0,0,0}\makebox(0,0)[lt]{\lineheight{1.25}\smash{\begin{tabular}[t]{l}$+$\end{tabular}}}}%
    \put(0.81714288,0.29500003){\color[rgb]{0,0,0}\makebox(0,0)[lt]{\lineheight{1.25}\smash{\begin{tabular}[t]{l}$-$\end{tabular}}}}%
    \put(0.26142856,0.31357144){\color[rgb]{0,0,0}\makebox(0,0)[lt]{\lineheight{1.25}\smash{\begin{tabular}[t]{l}$A_2$\end{tabular}}}}%
    \put(-0.8152981,0.03786951){\color[rgb]{0,0,0}\makebox(0,0)[lt]{\begin{minipage}{0.47784027\unitlength}\raggedright \end{minipage}}}%
    \put(0.57999371,0.42674469){\color[rgb]{0,0,0}\makebox(0,0)[lt]{\lineheight{1.25}\smash{\begin{tabular}[t]{l}{\footnotesize $D_2$}\end{tabular}}}}%
  \end{picture}%
\endgroup%

%% file: HK5_1_tunnel_view_and_twisted_version.pdf_tex
%% Accompanies image file 'HK5_1_tunnel_view_and_twisted_version.pdf' (pdf, eps, ps)
\begingroup%
  \makeatletter%
  \providecommand\color[2][]{%
    \errmessage{(Inkscape) Color is used for the text in Inkscape, but the package 'color.sty' is not loaded}%
    \renewcommand\color[2][]{}%
  }%
  \providecommand\transparent[1]{%
    \errmessage{(Inkscape) Transparency is used (non-zero) for the text in Inkscape, but the package 'transparent.sty' is not loaded}%
    \renewcommand\transparent[1]{}%
  }%
  \providecommand\rotatebox[2]{#2}%
  \newcommand*\fsize{\dimexpr\f@size pt\relax}%
  \newcommand*\lineheight[1]{\fontsize{\fsize}{#1\fsize}\selectfont}%
  \ifx\svgwidth\undefined%
    \setlength{\unitlength}{3401.57480315bp}%
    \ifx\svgscale\undefined%
      \relax%
    \else%
      \setlength{\unitlength}{\unitlength * \real{\svgscale}}%
    \fi%
  \else%
    \setlength{\unitlength}{\svgwidth}%
  \fi%
  \global\let\svgwidth\undefined%
  \global\let\svgscale\undefined%
  \makeatother%
  \begin{picture}(1,0.29166667)%
    \lineheight{1}%
    \setlength\tabcolsep{0pt}%
    \put(0,0){\includegraphics[width=\unitlength,page=1]{HK5_1_tunnel_view_and_twisted_version.pdf}}%
    \put(0.09116899,0.09498095){\color[rgb]{0,0,0}\makebox(0,0)[lt]{\begin{minipage}{0.02223633\unitlength}\raggedright \end{minipage}}}%
    \put(0,0){\includegraphics[width=\unitlength,page=2]{HK5_1_tunnel_view_and_twisted_version.pdf}}%
    \put(0.61951445,0.26050419){\color[rgb]{0,0,0}\makebox(0,0)[lt]{\begin{minipage}{0.03408028\unitlength}\raggedright \end{minipage}}}%
    \put(0.31812177,0.25078524){\color[rgb]{0,0,0}\makebox(0,0)[lt]{\lineheight{1.25}\smash{\begin{tabular}[t]{l}$\HKfiveone$\end{tabular}}}}%
    \put(0.56897762,0.00204248){\color[rgb]{0,0,0}\makebox(0,0)[lt]{\lineheight{1.25}\smash{\begin{tabular}[t]{l}$+A_1$-twisted $\HKfiveone$\end{tabular}}}}%
    \put(0.35644259,0.12324931){\color[rgb]{0,0,0}\makebox(0,0)[lt]{\lineheight{0.234375}\smash{\begin{tabular}[t]{l}$l$\end{tabular}}}}%
    \put(0.71846436,0.11394246){\color[rgb]{0,0,0}\makebox(0,0)[lt]{\lineheight{0.23437501}\smash{\begin{tabular}[t]{l}$l$\end{tabular}}}}%
    \put(0.20655961,0.1195447){\color[rgb]{0,0,0}\makebox(0,0)[lt]{\lineheight{0.23437501}\smash{\begin{tabular}[t]{l}$l$\end{tabular}}}}%
    \put(0.05179772,0.0740265){\color[rgb]{0,0,0}\makebox(0,0)[lt]{\lineheight{0.23437501}\smash{\begin{tabular}[t]{l}$m$\end{tabular}}}}%
  \end{picture}%
\endgroup%

%% file: Blow_up_construction2.pdf_tex
%% Accompanies image file 'Blow_up_construction2.pdf' (pdf, eps, ps)
\begingroup%
  \makeatletter%
  \providecommand\color[2][]{%
    \errmessage{(Inkscape) Color is used for the text in Inkscape, but the package 'color.sty' is not loaded}%
    \renewcommand\color[2][]{}%
  }%
  \providecommand\transparent[1]{%
    \errmessage{(Inkscape) Transparency is used (non-zero) for the text in Inkscape, but the package 'transparent.sty' is not loaded}%
    \renewcommand\transparent[1]{}%
  }%
  \providecommand\rotatebox[2]{#2}%
  \newcommand*\fsize{\dimexpr\f@size pt\relax}%
  \newcommand*\lineheight[1]{\fontsize{\fsize}{#1\fsize}\selectfont}%
  \ifx\svgwidth\undefined%
    \setlength{\unitlength}{1700.78740157bp}%
    \ifx\svgscale\undefined%
      \relax%
    \else%
      \setlength{\unitlength}{\unitlength * \real{\svgscale}}%
    \fi%
  \else%
    \setlength{\unitlength}{\svgwidth}%
  \fi%
  \global\let\svgwidth\undefined%
  \global\let\svgscale\undefined%
  \makeatother%
  \begin{picture}(1,0.36666667)%
    \lineheight{1}%
    \setlength\tabcolsep{0pt}%
    \put(0,0){\includegraphics[width=\unitlength,page=1]{Blow_up_construction2.pdf}}%
    \put(0.07044694,0.18777671){\color[rgb]{0,0,0}\makebox(0,0)[lt]{\lineheight{1.25}\smash{\begin{tabular}[t]{l} $\ast$\end{tabular}}}}%
    \put(0,0){\includegraphics[width=\unitlength,page=2]{Blow_up_construction2.pdf}}%
    \put(0.00642381,0.33469709){\color[rgb]{0,0,0}\makebox(0,0)[lt]{\lineheight{1.25}\smash{\begin{tabular}[t]{l}{\tiny $\mathfrak{N}(K(\ast))$}\end{tabular}}}}%
    \put(0,0){\includegraphics[width=\unitlength,page=3]{Blow_up_construction2.pdf}}%
  \end{picture}%
\endgroup%

%% file: Blow_up_construction3.pdf_tex
%% Accompanies image file 'Blow_up_construction3.pdf' (pdf, eps, ps)
\begingroup%
  \makeatletter%
  \providecommand\color[2][]{%
    \errmessage{(Inkscape) Color is used for the text in Inkscape, but the package 'color.sty' is not loaded}%
    \renewcommand\color[2][]{}%
  }%
  \providecommand\transparent[1]{%
    \errmessage{(Inkscape) Transparency is used (non-zero) for the text in Inkscape, but the package 'transparent.sty' is not loaded}%
    \renewcommand\transparent[1]{}%
  }%
  \providecommand\rotatebox[2]{#2}%
  \newcommand*\fsize{\dimexpr\f@size pt\relax}%
  \newcommand*\lineheight[1]{\fontsize{\fsize}{#1\fsize}\selectfont}%
  \ifx\svgwidth\undefined%
    \setlength{\unitlength}{3401.57480315bp}%
    \ifx\svgscale\undefined%
      \relax%
    \else%
      \setlength{\unitlength}{\unitlength * \real{\svgscale}}%
    \fi%
  \else%
    \setlength{\unitlength}{\svgwidth}%
  \fi%
  \global\let\svgwidth\undefined%
  \global\let\svgscale\undefined%
  \makeatother%
  \begin{picture}(1,0.33333333)%
    \lineheight{1}%
    \setlength\tabcolsep{0pt}%
    \put(0,0){\includegraphics[width=\unitlength,page=1]{Blow_up_construction3.pdf}}%
    \put(0.0301235,0.1728737){\color[rgb]{0,0,0}\makebox(0,0)[lt]{\lineheight{1.25}\smash{\begin{tabular}[t]{l}{\footnotesize $B$}\end{tabular}}}}%
    \put(0.13341182,0.14230679){\color[rgb]{0,0,0}\makebox(0,0)[lt]{\lineheight{1.25}\smash{\begin{tabular}[t]{l}{\tiny $D$}\end{tabular}}}}%
    \put(0,0){\includegraphics[width=\unitlength,page=2]{Blow_up_construction3.pdf}}%
    \put(0.62147055,0.12266108){\color[rgb]{0,0,0}\makebox(0,0)[lt]{\lineheight{1.25}\smash{\begin{tabular}[t]{l}{\footnotesize$B$}\end{tabular}}}}%
    \put(0,0){\includegraphics[width=\unitlength,page=3]{Blow_up_construction3.pdf}}%
    \put(0.75872338,0.15537135){\color[rgb]{0,0,0}\makebox(0,0)[lt]{\lineheight{1.25}\smash{\begin{tabular}[t]{l}{\tiny $D$}\end{tabular}}}}%
    \put(0.44226717,0.02979817){\color[rgb]{0,0,0}\makebox(0,0)[lt]{\lineheight{1.25}\smash{\begin{tabular}[t]{l}$K:$\end{tabular}}}}%
    \put(0,0){\includegraphics[width=\unitlength,page=4]{Blow_up_construction3.pdf}}%
  \end{picture}%
\endgroup%

%% file: Realization_of_E1.pdf_tex
%% Accompanies image file 'Realization_of_E1.pdf' (pdf, eps, ps)
\begingroup%
  \makeatletter%
  \providecommand\color[2][]{%
    \errmessage{(Inkscape) Color is used for the text in Inkscape, but the package 'color.sty' is not loaded}%
    \renewcommand\color[2][]{}%
  }%
  \providecommand\transparent[1]{%
    \errmessage{(Inkscape) Transparency is used (non-zero) for the text in Inkscape, but the package 'transparent.sty' is not loaded}%
    \renewcommand\transparent[1]{}%
  }%
  \providecommand\rotatebox[2]{#2}%
  \newcommand*\fsize{\dimexpr\f@size pt\relax}%
  \newcommand*\lineheight[1]{\fontsize{\fsize}{#1\fsize}\selectfont}%
  \ifx\svgwidth\undefined%
    \setlength{\unitlength}{1275.59055118bp}%
    \ifx\svgscale\undefined%
      \relax%
    \else%
      \setlength{\unitlength}{\unitlength * \real{\svgscale}}%
    \fi%
  \else%
    \setlength{\unitlength}{\svgwidth}%
  \fi%
  \global\let\svgwidth\undefined%
  \global\let\svgscale\undefined%
  \makeatother%
  \begin{picture}(1,0.55555556)%
    \lineheight{1}%
    \setlength\tabcolsep{0pt}%
    \put(0.19483269,0.53790766){\color[rgb]{0,0,0}\makebox(0,0)[lt]{\begin{minipage}{0.16772553\unitlength}\raggedright \end{minipage}}}%
    \put(0,0){\includegraphics[width=\unitlength,page=1]{Realization_of_E1.pdf}}%
    \put(0.45206433,0.44373856){\color[rgb]{0,0,0}\makebox(0,0)[lt]{\lineheight{1.25}\smash{\begin{tabular}[t]{l}$E$\end{tabular}}}}%
    \put(0,0){\includegraphics[width=\unitlength,page=2]{Realization_of_E1.pdf}}%
    \put(0.46378655,0.23930535){\color[rgb]{0,0,0}\makebox(0,0)[lt]{\lineheight{1.25}\smash{\begin{tabular}[t]{l}$H_2$\end{tabular}}}}%
    \put(-0.03787993,0.22730482){\color[rgb]{0,0,0}\makebox(0,0)[lt]{\lineheight{1.25}\smash{\begin{tabular}[t]{l} {\tiny $\partial D$}\end{tabular}}}}%
    \put(0.20118593,0.39248906){\color[rgb]{0,0,0}\makebox(0,0)[lt]{\lineheight{1.25}\smash{\begin{tabular}[t]{l}{\tiny $m_1$}\end{tabular}}}}%
    \put(0.79815809,0.48402756){\color[rgb]{0,0,0}\makebox(0,0)[lt]{\lineheight{1.25}\smash{\begin{tabular}[t]{l}{\tiny$m_2$}\end{tabular}}}}%
  \end{picture}%
\endgroup%

%% file: Realization_of_E2.pdf_tex
%% Accompanies image file 'Realization_of_E2.pdf' (pdf, eps, ps)
\begingroup%
  \makeatletter%
  \providecommand\color[2][]{%
    \errmessage{(Inkscape) Color is used for the text in Inkscape, but the package 'color.sty' is not loaded}%
    \renewcommand\color[2][]{}%
  }%
  \providecommand\transparent[1]{%
    \errmessage{(Inkscape) Transparency is used (non-zero) for the text in Inkscape, but the package 'transparent.sty' is not loaded}%
    \renewcommand\transparent[1]{}%
  }%
  \providecommand\rotatebox[2]{#2}%
  \newcommand*\fsize{\dimexpr\f@size pt\relax}%
  \newcommand*\lineheight[1]{\fontsize{\fsize}{#1\fsize}\selectfont}%
  \ifx\svgwidth\undefined%
    \setlength{\unitlength}{1332.28346457bp}%
    \ifx\svgscale\undefined%
      \relax%
    \else%
      \setlength{\unitlength}{\unitlength * \real{\svgscale}}%
    \fi%
  \else%
    \setlength{\unitlength}{\svgwidth}%
  \fi%
  \global\let\svgwidth\undefined%
  \global\let\svgscale\undefined%
  \makeatother%
  \begin{picture}(1,0.54255319)%
    \lineheight{1}%
    \setlength\tabcolsep{0pt}%
    \put(0.18654194,0.51501797){\color[rgb]{0,0,0}\makebox(0,0)[lt]{\begin{minipage}{0.16058827\unitlength}\raggedright \end{minipage}}}%
    \put(0,0){\includegraphics[width=\unitlength,page=1]{Realization_of_E2.pdf}}%
    \put(0.46847989,0.48972895){\color[rgb]{0,0,0}\makebox(0,0)[lt]{\lineheight{1.25}\smash{\begin{tabular}[t]{l}$E^{D,K}$\end{tabular}}}}%
    \put(0.49788052,0.18676053){\color[rgb]{0,0,0}\makebox(0,0)[lt]{\lineheight{1.25}\smash{\begin{tabular}[t]{l}$H_2$\end{tabular}}}}%
    \put(0,0){\includegraphics[width=\unitlength,page=2]{Realization_of_E2.pdf}}%
    \put(0.23541016,0.23320725){\color[rgb]{0,0,0}\makebox(0,0)[lt]{\lineheight{1.25}\smash{\begin{tabular}[t]{l}{\tiny $\partial D$}\end{tabular}}}}%
    \put(0,0){\includegraphics[width=\unitlength,page=3]{Realization_of_E2.pdf}}%
  \end{picture}%
\endgroup%